\documentclass[reqno]{amsart}
\usepackage{amssymb,  psfrag, multicol, float, graphicx, amsthm, mathrsfs, mathtools, bbm}
\usepackage[labelfont=bf]{caption}

\usepackage[all,cmtip]{xy}

\numberwithin{equation}{section}

\usepackage{xcolor}
\usepackage{pdfpages, enumerate}
\usepackage[numbers, sort]{natbib}
\usepackage[mathscr]{euscript}
\usepackage{soul}
\usepackage{comment}

\usepackage[colorlinks=true, linkcolor=teal, citecolor=teal]{hyperref}

\newcommand{\gx}[1]{\textcolor{magenta}{#1}}

\newtheorem{thm}{Theorem}[section]
\newtheorem{lemma}[thm]{Lemma}

\newtheorem{cor}[thm]{Corollary}
\newtheorem{prop}[thm]{Proposition}

\theoremstyle{remark}
\newtheorem{remark}[thm]{Remark}

\newtheorem{defn}[thm]{Definition}
\newtheorem{example}[thm]{Example}

\newcommand{\F}{{\mathbb{F}}}
\newcommand{\R}{{\mathbb{R}}}
\newcommand{\Z}{{\mathbb{Z}}}

\newcommand{\Q}{{\mathbb{Q}}}
\newcommand{\D}{{\mathbb{D}}}
\newcommand{\bK}{{\mathbb{K}}}

\newcommand{\wh}[1]{\widehat{#1}}
\newcommand{\wt}[1]{\widetilde{#1}}

\newcommand{\ol}[1]{\overline{#1}}

\newcommand{\zp}{{\mathbb{Z}/p}}

\newcommand{\rp}{{\mathcal{R}_p}}
\newcommand{\kzp}{{\mathcal{K}_{0, p}}}
\newcommand{\kp}{{\mathcal{K}_p}}

\newcommand{\fp}{{{\mathbb F}_p}}
\newcommand{\beqn}{\begin{equation*}}
\newcommand{\eeqn}{\end{equation*}}

\newcommand{\pst}{[\![t]\!]} 

\newcommand{\mb}{\mathbb}
\newcommand{\mc}{\mathcal}
\newcommand{\ms}{\mathscr}

\newcommand{\ov}{\overline}

\newcommand{\beq}{\begin{equation}}
\newcommand{\eeq}{\end{equation}}

\newcommand{\om}{\omega}

\newcommand{\eps}{\epsilon}

\newcommand{\cO}{\mathcal{O}}

\newcommand{\fix}{\mathrm{Fix}}

\def\mrm#1{{\mathrm{#1}}}

\def\cl#1{{\mathcal{#1}}}

\def\til#1{{\widetilde{#1}}}

\newcommand{\brat}[1]{{\left< #1 \right>}}

\newcommand{\loc}{{\rm loc}}
\newcommand{\fst}{{\it FSt}}
\newcommand{\qst}{{\it QSt}}
\newcommand{\pss}{{\it PSS}}
\newcommand{\ssp}{{\it SSP}}
\newcommand{\inv}{{\rm Inv}}

\DeclareMathOperator{\rank}{\mathrm{rank}}

\DeclareMathOperator{\Ham}{\mathrm{Ham}}

\DeclareMathOperator{\id}{\mathrm{id}}
\DeclareMathOperator{\spec}{\mathrm{Spec}}

\DeclareMathOperator{\CZ}{\mathrm{CZ}}

\def\H2{H^{(2)}}

\newif\ifmoditem
\newcommand{\setupmodenumerate}{%
  \global\moditemfalse
  \let\origmakelabel\makelabel
  \def\moditem##1{\global\moditemtrue\def\mesymbol{##1}\item}%
  \def\makelabel##1{%
    \origmakelabel{##1\ifmoditem\rlap{\mesymbol}\fi\enspace}%
    \global\moditemfalse}%
}

\usepackage{float}
\begin{document}
	
\title{Quantum Steenrod powers and Hamiltonian maps}	

\author{Shaoyun Bai}
\address{MIT, 77 Massachusetts Avenue Cambridge, MA 02139, USA}
\email{shaoyunb@mit.edu}

\author{Egor Shelukhin}
\address{Department of Mathematics, University of Montreal, Canada}
\email{egorshel@gmail.com}

\author{Nicholas Wilkins}
\email{nick.mick.wilk@gmail.com}

\author{Guangbo Xu}
\address{Department of Mathematics, Rutgers University, Hill Center--Busch Campus, 110 Frelinghuysen Road, Piscataway, NJ 08854-8019, USA}
\email{guangbo.xu@rutgers.edu} 
	
\begin{abstract}
We prove a series of new results in Hamiltonian dynamics on a general closed symplectic manifold
$(M, \omega)$, including:

1. If $M$ admits a Hamiltonian diffeomorphism which is either a pseudo-rotation or has
finite order, then $M$ is geometrically uniruled. This resolves a variant of Problem 24 in
McDuff–Salamon’s list which predicts an obstruction to the existence of such special Hamiltonian diffeomorphisms in terms of genus zero numerical invariants.

2. If a Hamiltonian possesses a periodic orbit in a non-torsion homology class, then it has infinitely many simple periodic points. The same conclusion holds if the manifold is not geometrically uniruled and the diffeomorphism is minimal for rational Floer homology. These two general results complement the known cases of the Hofer–Zehnder conjecture.

3. If a Hamiltonian diffeomorphism possesses a symplectically degenerate maximum, then it has infinitely many simple periodic points. This resolves an open question which stems from the work of Ginzburg and G\"urel.

We also establish new cases of the generic Conley conjecture: infinitely many periodic points for generic Hamiltonian diffeomorphisms.

The proofs rely on a systematic application of the integral Hamiltonian Floer theory package developed by the first and fourth author, a K\"unneth isomorphism in equivariant Floer homology, and new quantitative analysis of quantum power maps, which is of independent interest.
\end{abstract}

\maketitle

\setcounter{tocdepth}{1}
\tableofcontents



\section{Introduction}

\subsection{Overview}
In recent years there has been active research into equivariant Floer theory for symplectomorphisms, which takes into account the action of a cyclic group on the loop space of a symplectic manifold via discrete rotation. This theory has its origins in the work of Fukaya \cite{fukaya}, who first suggested a quantum version of Steenrod square operations. This theory has applications to Gromov--Witten invariants of symplectic manifolds \cite{wilk,wilk2,Seidel-formal,covariant-constant, jaehee, jae-hee-lee,  chen2024exponential, chen2024quantum, seidel2025p,bai20253d, bai2026arithmetic} as well as to the topological, dynamical, and algebraic properties of groups of Hamiltonian diffeomorphisms and, more generally, of symplectomorphisms \cite{seidel,SZhao-pants,S-HZ,CGG2,S-PRQS,covariant-constant,S-PRQSR,AS-torsion,Sugimoto, Allais1, Allais2,BaiXu, AtallahLou, shelukhin2024, cineli2024closed}. In this paper we continue the Floer-theoretic study of the rotational $\Z/p$-symmetry on loops and deduce applications to Hamiltonian dynamics on \emph{all} closed symplectic manifolds using the foundational package \cite{Bai_Xu_2025} developed by the first and fourth author.

Extending the results of \cite{wilk,CGG2,S-PRQS,S-PRQSR,AS-torsion} to arbitrary primes $p$ and to arbitrary closed symplectic manifolds, in this paper we provide new obstructions to the existence of Hamiltonian pseudo-rotations and Hamiltonian diffeomorphisms of finite order. Along the way, we establish a new Kunneth formula for equivariant power operations in Floer theory and quantum cohomology, and prove a new divisibility bound in local equivariant Floer cohomology. We also prove the existence of infinitely many periodic points for generic Hamiltonian diffeomorphisms for certain symplectic manifolds beyond already known cases
(see \cite{GG-generic, Sugimoto-generic}). Furthermore, we prove a new result on the existence of infinitely many periodic orbits of a Hamiltonian diffeomorphism. It is of complementary nature to that in \cite{S-HZ} and provides further evidence for the Chance--McDuff conjecture according to which the existence of a Hamiltonian diffeomorphism with finitely many periodic points on a closed symplectic manifold must imply the non-vanishing of a Gromov--Witten invariant, or more generally, the existence of a non-constant pseudo-holomorphic sphere. Finally, extending the results of \cite{Sugimoto,GG-ai} to full generality, we prove the existence of infinitely many periodic points of a Hamiltonian diffeomorphism which either possesses a periodic orbit in a non-torsion homology class with non-vanishing local Floer cohomology or a symplectically degenerate maximum, for \emph{any} closed symplectic manifold. This completes two directions of research initiated in the work of Ginzburg and G\"urel.

\subsection{Main results}

Let us now discuss our main results in more detail and provide contexts in the literature. Let $(M, \omega)$ be a closed symplectic manifold. For any prime $p$, let ${\mb F}_p$ be the finite field of $p$ elements.

\subsubsection{Hamiltonian pseudo-rotations and Hamiltonian torsions}

The first main result 
deals with Hamiltonian pseudo-rotations: those Hamiltonian diffeomorphisms that have the homologically minimal number of periodic points of all periods from the perspective of the Arnol'd conjecture (i.e. the total Betti number). Let $\phi$ be a Hamiltonian diffeomorphism with a finite set $\mathrm{Fix}_c(\phi)$ of contractible fixed points $x$--those for which the loop $\{ \phi^t_H x \}$ is contractible for every Hamiltonian $H$ with time-one map $\phi^1_H = \phi$. Given a field $\bK$, following \cite{S-HZ}, we define the homological count of contractible fixed points of $\phi$ over $\bK$ as follows: 
\beqn
N(\phi;\bK) = \sum_{x \in \mathrm{Fix}_c(\phi)} \dim_{\mb K} HF^{\loc}(\phi,x; \bK),
\eeqn
where $HF^{\loc}(\phi,x; \bK)$ denotes the local Floer cohomology with $\bK$-coefficients.

\begin{defn}\label{def: pr}
A Hamiltonian diffeomorphism $\phi \in \Ham(M,\om)$ of a closed symplectic manifold $(M,\om)$ is called an {\bf $\F_p$ Hamiltonian pseudo-rotation} if it has finitely many periodic points of periods being powers of $p$ and for all $k \geq 1$ there holds 
\beqn
N(\phi^{p^k}; \F_p) = \dim_{\F_p} H^*(M;\F_p).
\eeqn
We call $\phi$ a {\bf pseudo-rotation} if it has finitely many periodic points, and for every period $m \geq 1$, over every field $\bK$ we have
\beqn
N(\phi^m;\bK) = \dim_{\bK} H^*(M;\bK)
\eeqn
for all such iterations. 
\end{defn}

It is believed (such as in the well-known Chance--McDuff conjecture) that the existence of Hamiltonian pseudo-rotations implies the abundance of holomorphic spheres. Such a connection was established in terms of the quantum Steenrod square by  
\cite{S-PRQS,CGG2,S-PRQSR} and 
was recently applied to a new example of the monotone $6$-point blowup of $\mb{CP}^2$ in \cite{covariant-constant}. Indeed, it was proven in \cite{S-PRQS, S-PRQSR, CGG2} that if $(M,\om)$ is a spherically monotone symplectic manifold and admits an $\F_2$ Hamiltonian pseudo-rotation, then it is uniruled in the sense of quantum Steenrod operations. Namely, the quantum Steenrod square $\qst_2(\mu)$ of the counit $\mu \in H^{2n}(M;\F_2)$ is deformed by holomorphic spheres in the sense that 
\beqn
\qst_2(\mu) \neq St_2(\mu),
\eeqn
where $St_2$ denotes the total Steenrod power $St_p$ for $p=2$. It is easy to see 
that this implies that $(M,\om)$ is geometrically uniruled: for every $\om$-compatible almost complex structure $J,$ and every point $x \in M,$ there exists a non-constant $J$-holomorphic sphere passing through $x.$ 

Our result below, which was proved in Section \ref{subsection_proof_uniruled}, generalizes this criterion to all primes $p$ in terms of quantum Steenrod $p$-th powers for arbitrary closed symplectic manifolds. In the following discussion, we denote by $QSt_p$ the $p$-th quantum Steenrod power as defined in \cite{wilk,Seidel-formal,covariant-constant,Bai_Xu_2025}.

\begin{thm}[Pseudo-rotation $\Rightarrow$ uniruledness]\label{thm: pr}
Let $p$ be a prime. Let $(M,\om)$ be a closed symplectic manifold which admits an $\F_p$ Hamiltonian pseudo-rotation. Then it is $\F_p$ Steenrod uniruled. Namely, the quantum Steenrod $p$-th power $\qst_p(\mu)$ of the counit $\mu \in H^{2n}(M;\F_p)$ is deformed by holomorphic spheres in the sense that \[ \qst_p(\mu) \neq St_p(\mu).\]   
\end{thm}

Extending the main result in \cite{AS-torsion}, similar ideas in the proof of Theorem \ref{thm: pr} allow us to prove the following result regarding Hamiltonian torsion, which provides an answer to \cite[Chapter 14, Problem 24]{mcduff-salamon-intro} in full generality (see proof in Section \ref{subsection_proof_no_torsion}).

\begin{thm}[Hamiltonian torsion $\Rightarrow$ uniruledness]\label{thm_torsion}
Let $(M,\om)$ be a compact symplectic manifold, which admits a nontrivial Hamiltonian action by a finite group $G$. Then $(M, \omega)$ is $\F_p$ Steenrod uniruled for any prime $p$ which does not divide $|G|$.
\end{thm}

\begin{remark}
In other words, we prove that the Gromov--Witten type invariant coming from the ``quantum part" of $\qst_p(\mu)$ is nontrivial. This is arguably not the only form of answer to \cite[Chapter 14, Problem 24]{mcduff-salamon-intro}, but we believe that the condition $\qst_p(\mu) \neq St_p(\mu)$ should imply the non-triviality of certain ordinary Gromov--Witten invariants, possibly with gravitational descendants. However, establishing such a statement is more in the realm of enumerative geometry, and Theorem \ref{thm_torsion} should reflect the essence of the interaction between Hamiltonian torsion and abundance of rational curves. See \cite{Rezchikov-gw} for related work in this direction.
\end{remark}

We can also deduce the following result concerning finite group actions via Hamiltonian diffeomorphisms, extending \cite[Theorem H]{AS-torsion} to full generality.

\begin{thm}[Contractibility of fixed points of finite-order Hamiltonian diffeomorphisms]\label{thm:contractible}
    Let $(M, \omega)$ be a compact symplectic manifold. If $\phi$ is a Hamiltonian diffeomorphism of $(M, \omega)$ of finite order, then all the fixed points of $\phi$ are contractible.
\end{thm}

\subsubsection{Results towards the generic Conley conjecture}

Hamiltonian pseudo-rotation and Hamiltonian torsion both represent the rare case when a Hamiltonian diffeomorphism has only finitely many periodic points. In fact, the so-called ``generic Conley conjecture'' asserts that on any symplectic manifold, generic Hamiltonian diffeomorphisms should have infinitely many periodic points. The method used to prove Theorem \ref{thm: pr} and Theorem \ref{thm_torsion} can be modified to prove the following theorem and two corollaries in the scope of this conjecture. We remind the reader that by the Birkhoff--Moser theorem \cite{Moser}, the potential failure of the generic Conley conjecture could come from Hamiltonian diffeomorphisms with only hyperbolic fixed points. Hence the main argument (see Subsection \ref{subsection_proof_hyperbolic}) is to use quantum Steenrod operations to rule out such transformations.


\begin{thm}\label{thm: hyperbolic}
Let $(M,\om)$ be a closed symplectic manifold and $\phi \in \Ham(M,\om)$. 
Suppose $\phi$ has finitely many periodic points, all of which are hyperbolic. Then for any prime $p$, and a class $u \in H^r(M;\F_p)$ with $r>n$ satisfying 
\begin{equation}\label{eq: triv St}
St_p(u) = c_p t^{r(p-1)/2}u,
\end{equation}
for some $c_p \in \fp\setminus \{0\}$, we have $\qst_p(u) \neq St_p(u).$  
\end{thm}

The conceptual reason for having Theorem \ref{thm: hyperbolic} is that the hyperbolicity of periodic points puts a strong restriction (via grading) on the local version of the equivariant pair-of-pants product, which further constrains $\qst_p$.


\begin{cor}[Generic Conley conjecture from $\qst_p$]\label{cor: generic}
Let $(M,\om)$ be a closed symplectic manifold such that there exists a prime $p$ and a class $u \in H^r(M;\F_p)$ with $r>n$ such that 
\beqn
\qst_p(u) = St_p(u) = c_p t^{r(p-1)/2}u
\eeqn
for $0 \neq c_p \in \F_p$. Then any $C^2$-generic  Hamiltonian diffeomorphism $\phi \in \Ham(M,\om)$ has infinitely many periodic points.
\end{cor}

\begin{remark}\label{rmk: triv St}
Note that \eqref{eq: triv St} holds whenever $2(p-1)+r > 2n$ and $\beta_p(u) = 0$ where $\beta_p: H^r(M;\F_p) \to H^{r+1}(M;\F_p)$ is the Bockstein operator. However, it also holds in certain other cases, such as the setting of the proof of Corollary \ref{cor: generic new example} below.
\end{remark}

The rationale of Corollary \ref{cor: generic} is to provide new examples of manifolds for which generic Conley conjecture holds, which are also not covered by existing results. Indeed, by \cite{GG-generic, Sugimoto-generic}, generic Conley conjecture holds for any symplectic manifold $M$ for which one of the following three conditions fails to hold.
\begin{enumerate}

\item The minimal Chern number is $1$.

\item The dimension of $M$ is divisible by 4.

\item The odd rational cohomology $H^{\rm odd}(M; \Q)$ vanishes.\footnote{
(We thank Semon Rezchikov for his contribution to this observation.) By modifying the arguments of \cite{GG-generic, Sugimoto-generic} using the general Floer package of \cite{Bai_Xu_2025}, this condition can be replaced by an integral version, namely, the generic Conley conjecture holds for $M$ unless $H^*(M; {\mb Z})$ is torsion free and $H^{\rm odd}(M; {\mb Z}) = 0$. } 
\end{enumerate}
An explicit example beyond the existing cases is designed as follows. Let $X = {\rm Bl}_6(\mb{CP}^2)$ and $\mb{CP}^2$ be equipped with monotone symplectic forms such that $M = X \times X \times \mb{CP}^2$ is also monotone. Then $M$ satisfies all the above three conditions.



\begin{cor}\label{cor: generic new example}
$M$ admits no hyperbolic pseudo-rotations, and any $C^2$-generic Hamiltonian diffeomorphism $\phi \in \Ham(M,\om)$ has infinitely many periodic points.
\end{cor}

\begin{remark}\label{rem:even-prime}
In this paper, many technical constructions are presented only for odd primes $p$. However, Corollary \ref{cor: generic new example} uses quantum Steenrod square. Indeed, all the constructions for $p>2$ work perfectly well for $p=2$, with the only difference being that $H^*(B \mathbb{Z}/p; {\mb F}_p)$ has different ring structures when $p=2$ or $p>2$. In our notations, the further relation we need to impose for the generators is that $t = \theta^2$.
\end{remark}

The proof of this corollary relies on the following technical result in the form of a version of the Kunneth isomorphism for equivariant cohomology, and a ``splitting" of $\qst_p$ over this Kunneth isomorphism. For brevity we denote $\rp = H^*(B \mathbb{Z}/p; {\mb F}_p).$ Also, denote by $\Lambda^{\rm univ}$ the universal Novikov field over $\F_p$.
Denote by $\kappa$ the Kunneth map on cohomology.

\begin{thm}[Kunneth for $\qst_p$]\label{thm: Kunneth}
Let $X_1$ and $X_2$ be monotone symplectic manifolds such that $X_1 \times X_2$ is also monotone. Then for each prime $p$, the following diagram commutes up to the Koszul sign.
\beq\label{equation:kunneth-diagram-to-prove}
\vcenter{
\xymatrix{ H^*(X_1) \otimes H^*(X_2)
\ar[rrr]^-{\qst_p^{X_1} \otimes \qst_p^{X_2} }
\ar[d]_{\kappa}
&
&
&
H^*(X_1;\Lambda_\rp^{\rm univ}) \otimes H^*(X_2; \Lambda_\rp^{\rm univ}) \ar@{->}^-{\kappa}[d]
\\
H^*(X_1 \times X_2 )
\ar[rrr]_{\qst_p^{X_1\times X_2}}
&
&
&
H^*(X_1 \times X_2;\Lambda_\rp^{\rm univ})
}}
\eeq
Here $\kappa$ is the ordinary Kunneth map. More precisely, if $a_1 \in H^*(X_1)$, $a_2 \in H^*(X_2)$, then 
\beqn
\kappa \left( \qst_p^{X_1}(a_1) \otimes \qst_p^{X_2}(a_2) \right) = (-1)^{|a_1||a_2|p(p-1)/2} \qst_{p}^{X_1 \times X_2} (\kappa(a_1 \otimes a_2)).
\eeqn
\end{thm}

We will use this, as well as the related result Proposition \ref{prop55} for the equivariant pair-of-pants, to perform calculations such as Corollary \ref{cor: generic new example}.

\begin{remark}
Note that we prove Theorem \ref{thm: Kunneth} only in the monotone setting, even though our discussions about quantum Steenrod operations are concerned with general symplectic manifolds elsewhere in the paper. We expect that proving Theorem \ref{thm: Kunneth} in general is a subtle problem. A related result, the Kunneth theorem for quantum cohomology in the symplectic setting \cite{hirschi-swaminathan}, demonstrates part of the difficulties: the moduli spaces of stable maps of a product can be different from the product of the respective moduli spaces, which makes the comparison between the virtual fundamental cycles a nontrivial problem. The interaction between such differences and the perturbation scheme in \cite{BaiXu-FPO1} needs to be sorted out in order to find the generalization in our setting. Nevertheless, Theorem \ref{thm: Kunneth} suffices for the applications in this paper.
\end{remark}

\subsubsection{Generalizations of the Hofer--Zehnder conjecture}

Next, we provide a new criterion for the existence of infinitely many periodic points of a Hamiltonian diffeomorphism, which is quite different from other criteria in the literature (see Section \ref{subsection_proof_anti_HZ} for the proof).

\begin{thm}[``Reversed" Hofer--Zehnder conjecture]\label{thm: infinitely many}
Let $(M,\om)$ be a closed symplectic manifold such that $\qst_p(\mu) = St_p(\mu)$ for 
infinitely many primes $p$. Suppose $\phi \in \Ham(M,\om)$ has finitely many contractible fixed points and $N(\phi;\Q) = \dim_{\Q} H^*(M;\Q)$. Then $\phi$ has infinitely many periodic points.
\end{thm}

\begin{remark}
The proof shows that under an additional non-degeneracy assumption on the fixed points of $\phi$ (for instance if none of them is a symplectically degenerate maximum), there exist infinitely many primes $p$ such that $\phi$ possesses a simple $p$-periodic point.
\end{remark}

Using a standard Gromov compactness argument, Theorem \ref{thm: infinitely many} shows that a Hamiltonian diffeomorphism $\phi$ of a closed symplectic manifold which is not geometrically uniruled, having the minimal possible number of fixed points (corresponding to contractible orbits as constrained by the Arnol'd conjecture) must in fact have infinitely many periodic points. This result is interesting in comparison with prior results in this direction. For instance, the Hofer--Zehnder conjecture, proved in \cite{S-HZ} for closed spherically monotone manifolds with semi-simple quantum cohomology (see also \cite{AtallahLou} for the generalization to weakly monotone symplectic manifolds, as well as \cite{Allais2, BaiXu} for weighted projective spaces and all toric varieties respectively), asserts the existence of infinitely many periodic points under the contrary assumption that $\phi$ has strictly more fixed points than necessary.

\begin{remark}
    Using the integral Gromov--Witten-type invariants defined in \cite{BaiXu-FPO1}, one can define an analog of quantum cohomology ring, whose Floer-theoretic counterpart is discussed in detail in \cite{Bai_Xu_2025}. It is straightforward to adapt the methods from \cite{S-HZ,AtallahLou} to prove that under the assumption that this quantum cohomology ring is semi-simple, the Hofer--Zehnder conjecture as discussed above holds.
\end{remark}

Our final applications are concerned with two sufficient conditions for the existence of infinitely many periodic points. 

The first sufficient condition has its roots in the Hofer--Zehnder conjecture, in the way that noncontractible $1$-periodic orbits of a Hamiltonian should be considered as more than necessary in view of the Arnol'd conjecture. For a Hamiltonian $H$ on $M$, if $x: S^1 \to M$ is a noncontractible 1-periodic orbit of $H$, then one can also define the local Floer cohomology $HF^\loc(H, x)$ (over ${\mb Z}$ or any field). The details for the definition can be found in \cite[Section 2.3]{Sugimoto}.

\begin{thm}[Non-torsion orbit in $H_1$ $\Rightarrow$ Conley conjecture]\label{thm_noncontractible}
Let $(M, \omega)$ be a compact symplectic manifold and $H: S^1 \times M \to \mathbb{R}$ be a smooth Hamiltonian. If there exists a $1$-periodic orbit with underlying homology class $0 \neq \gamma \in H_1(M;\Z)\setminus {\rm Torsion}$ which has non-vanishing local Floer cohomology, then for any sufficiently large prime number $p$, the time-one map $\phi = \phi_H^1$ must admit a simple periodic point in the class $p \gamma$ of period $p$ or $p'$, where $p'$ is the next prime number after $p$. In particular, $\phi$ has infinitely many periodic points.
\end{thm}

Theorem \ref{thm_noncontractible} extends the main result in \cite{Sugimoto} to full generality, establishing a form of the Hofer--Zehnder conjecture for general symplectic manifolds. For deducing infinitely many periodic orbits from noncontractible Hamiltonian orbits, see also \cite{gurel-noncontractible, gg2016}.

The second sufficient condition closes up the search of infinitely many Hamiltonian periodic orbits from the point of view of symplectically degenerate maxima, which was a phenomenon introduced to symplectic dynamics by Hingston \cite{Hingston_2009} and explored by Ginzburg in his proof of the Conley conjecture for symplectically aspherical manifolds \cite{Ginzburg-CC} as well as in numerous further works. Recall that an isolated capped 1-periodic orbit $\overline{x}$ of a Hamiltonian $H$ is called a symplectically degenerate maximum if the mean index $\Delta(\overline{x}, H) = 0$ and the local Floer cohomology satisfies $HF^{\loc, 2n} (H, \overline{x}; {\mb Q}) \neq 0$ (here we use the cohomological grading $n+\CZ$ to grade the generators so that the index of a nondegenerate maximum of a $C^2$-small Morse function is $2n$). 

Finally, the following theorem extends \cite[Theorem 1.18]{GG-ai}, \cite[Theorem 1.2]{Hein-CCCY} to full generality, solving a question which stems from \cite[Remark 1.19]{GG-ai}. It is proved in Section \ref{proof_thm_sdm}. The possibility of working with different primes allows us to prove the stronger result on infinitely many periodic points: the number of periodic points of period $\leq k$ grows at least as $k^2/\log(k)$, which is the bound from the prime number theorem.

\begin{thm}[SDM $\Rightarrow$ Conley conjecture]\label{thm: sdm}
Let $(M,\om)$ be any compact symplectic manifold. If a Hamiltonian $H$ has a capped 1-periodic orbit which is a symplectically degenerate maximum, then its time-1 map $\phi$ has infinitely many periodic points.
\end{thm}

\subsection{Methods}
On the foundational side, fitting into the theme of recent extensive investigations, our results use operations in equivariant Floer theory. Historically, the earliest example of a power operation in Floer theory goes back to Fukaya, \cite{fukaya}, where he describes the construction of a version of the Steenrod square on quantum cohomology. This construction was generalized and explored in the third author's work \cite{wilk} in relation with the Cartan and Adem relations. In the context of fixed-point Floer cohomology of exact symplectomorphisms of Liouville manifolds, a power operation was introduced by Seidel in \cite{seidel}. The third author has shown an isomorphism between the last two operations in \cite{wilk2}. The construction of \cite{seidel} was extended to odd primes in \cite{SZhao-pants,S-HZ}, and that of \cite{wilk} was extended to odd primes in \cite{Seidel-formal,covariant-constant}. The constructions mentioned so far in this paragraph are in the monotone or weakly monotone setting.

The crucial technical input of this paper is the foundational package of integral Hamiltonian Floer theory  from \cite{Bai_Xu_2025} built by the first and fourth author. In \cite{Bai_Xu_2025}, the first and fourth author constructed quantum Steenrod operations and power operations on Hamiltonian Floer cohomology for arbitrary closed symplectic manifolds. It is also shown in \cite{Bai_Xu_2025} that these two operations are compatible (via a PSS type map) with each other. The construction is based on mod $p$ reduction of the integral virtual counts of pseudo-holomorphic curves initiated in \cite{BaiXu-FPO1}, which is genuinely different from the ${\mb Q}$-valued counts under the classical framework \cite{Fukaya_Ono,Liu_Tian_Floer,HWZ_book, Pardon_virtual} in general. Such definition forms the base point of all the applications in this paper. Furthermore, quantitatively,
the virtual framework in \cite{Bai_Xu_2025} has a certain advantage over the traditional approach. Indeed, the traditional transversality argument requires inhomogeneous perturbations of the Floer equation, which alter the quantitative feature of solutions (by an arbitrarily small but nonzero amount). In the equivariant setting where one typically has an infinite-dimensional parameter space, the accumulations of these small defects become cumbersome to control. Such an issue does not arise when defining chain-level operations using the techniques in \cite{Bai_Xu_2025} because transversality of moduli spaces is achieved by abstract perturbations, which do not interfere with the energy estimates of Floer-type equations. This is why our results have maximum level of generality, especially for those applications that remove technical assumptions like monotonicity, and why the arguments presented in this paper are largely algebraic based on the black box built in \cite{Bai_Xu_2025}.

Nevertheless, many of our arguments go well beyond directly plugging known arguments into the package \cite{Bai_Xu_2025}. There are new conceptual inputs except for the technical advancements. Most notably, to deal with degenerate isolated periodic points of Hamiltonian diffeomorphisms, we establish a divisibility result for the local equivariant pants product which is essential for many applications. Moreover, our arguments showcase how to fruitfully combine equivariant cohomology, filtered Floer theory, and the compatibility between quantum and Floer-theoretic operations to deduce strong constraints on Hamiltonian diffeomorphisms. For instance, our proof of Theorem \ref{thm: sdm} consists in a novel application of the equivariant pants product combined with considerations of filtered Floer theory. We expect the methods developed in this paper to be useful in other contexts.

\subsection{Outline of the paper}
In Section \ref{sec:prelim}, we recall the necessary preliminary knowledge about Floer theory and quantum Steenrod operations that will be used in the paper, with an emphasis on highlighting some formal aspects of the theory built in \cite{Bai_Xu_2025} and their connection with the more classical approach. In Section \ref{section_hp}, we apply the homological perturbation method to extend many useful constructions from the nondegenerate case to the isolated but possibly degenerate case. In Section \ref{sec-proofsappl}, we prove our main results modulo some technical proofs deferred to later sections. In Section \ref{sec:equivariant-kunneth}, we establish Kunneth type results for both the quantum Steenrod operations and local equivariant pair-of-pants product, which are used for proving Corollary \ref{cor: generic new example} and the divisibility result of local equivariant pants product respectively. This paper has three appendices. In Appendix \ref{sec:alternative-qst}, we discuss the agreement of quantum Steenrod operations constructed using different frameworks. In Appendix \ref{app: hp}, we provide the proofs of various homological perturbation results that are used in Section \ref{section_hp}. In Appendix \ref{app:MB}, we provide technical details of homological perturbation in the Morse--Bott setting that are used for discussing finite-order Hamiltonian diffeomorphisms.

\subsection*{Acknowledgments}
We thank Paul Seidel for numerous useful discussions and for suggesting a first version of Theorem \ref{thm: infinitely many}, as well as Leonid Polterovich for numerous inspiring discussions. S.B. was supported by the NSF standard grant DMS-2404843 and CAREER grant DMS-2540393. E.S. was supported by an NSERC Discovery Grant, an FRQNT Teams Grant, the Fondation Courtois, and a Sloan Research Fellowship. N.W. was supported initially by a Heilbronn Research Fellowship, then by the Simons Foundation through award 652299, and finally a postdoctoral fellowship at the Max Planck Institute for Mathematics in Bonn. G.X. was supported by the NSF grant DMS-2345030 and DMS-2506403, and Simons Foundation Travel Grant.

\section{Preliminaries}\label{sec:prelim}\label{section2}

We start by recalling some basic terminology and setting up our conventions. Let $(M, \omega)$ be a closed symplectic manifold with first Chern class $c_1$. First, set \[\Gamma = \pi_2(M)/\ker([\om]) \cap \ker(c_1).\] Let $\bK$ be a commutative unital ring. Set
\beq\label{eqn_Novikov_ring}
\Lambda_{\bK}^{\Gamma} = \Big\{  \sum_{i = 1}^\infty a_i q^{A_i}\ |\ a_i \in {\mb K},\ A_i \in \Gamma,\ \lim_{i \to \infty} \omega(A_i) = +\infty \Big\}.
\eeq
It will also be useful to consider the universal Novikov ring
\beqn
\Lambda^{\mrm{univ}}_{\bK} = \Big\{ \sum_{i =1}^\infty a_i q^{\lambda_i}\ |\ a_i \in {\mb K},\ \lambda_i \in {\mb R},\ \lim_{i \to \infty} \lambda_i = + \infty \Big\}
\eeqn
These rings are fields when $\bK$ is a field. Moreover, $\Lambda_{\mb K}^\Gamma$ resp. $\Lambda_{\mb K}^{\rm univ}$ contains the subring $\Lambda_{0, {\mb K}}^\Gamma$ resp. $\Lambda_{0, {\mb K}}^{\rm univ}$ of sums as above with $\omega(A_i) \geq 0$ resp. $\lambda_i \geq 0$ for all $i$. There exists a ring homomorphism
\begin{align}\label{Lambda_base_change}
&\ \Lambda_{\mb K}^{\Gamma} \to \Lambda_{\mb K}^{\rm univ}\ &\ 
    \sum_{A_i \in \Gamma} a_i q^{A_i} \mapsto \sum a_i q^{\omega(A_i)},
    \end{align}
by which we view $\Lambda_{\mb K}^{\rm univ}$ as a $\Lambda_{\mb K}^{\Gamma}$-module.

As we primarily use a ground field $\fp$ for $p$ being a prime number, when we drop the ground ring from notations, we always mean the Novikov ring/field over $\fp$. Namely
\begin{align*}
&\ \Lambda^\Gamma:= \Lambda_\fp^\Gamma,\ &\ \Lambda^{\rm univ}:= \Lambda^{\rm univ}_{\fp}.
\end{align*}

As different coefficient rings/fields will be used, one needs to be sensitive about taking tensor products. As a convention, when we write $A \otimes B$ without referring to the coefficient ring/field, one always means the tensor product as abelian groups. Notice that if $A$ and $B$ are vector spaces over $\fp$, $A \otimes B$ coincides with the vector space tensor product $A \otimes_{\fp} B$.

\subsection{Floer theory}
\label{subsec:floer-theory}

For our applications, we use cohomological operations in Floer theory. Therefore, we use Floer cohomology throughout our discussion.

\subsubsection{Conventions}

We set up notations for Hamiltonians. Let $(M, \omega)$ be a compact symplectic manifold. For any function $H$ on $M$, we would like to choose a convention for the associated Hamiltonian vector field $X_H$ (different conventions differ by a sign). Then for any 1-periodic Hamiltonian $H$, let ${\mc O}(H)$ be the set of contractible 1-periodic orbits and $\tilde {\mc O}(H)$ be the set of capped 1-periodic orbits which has a free $\Gamma$-action with $\tilde {\mc O}(H)/\Gamma \cong {\mc O}(H)$. Let $x$ resp. $\ov{x}$ denote a typical element of ${\mc O}(H)$ resp. $\tilde {\mc O}(H)$. 

One would also like to define a symplectic action functional ${\mc A}_H$ on the space of capped loops whose critical point set coincides with $\tilde {\mc O}(H)$, and, if a capped orbit $\ov{x} \in \tilde {\mc O}(H)$ is nondegenerate, a cohomological degree $|\ov{x}| \in {\mb Z}$. Without explicitly describing our conventions, one can choose them to satisfy the following conditions.
\begin{enumerate}
    \item Differential of the Floer cochain complex increases the action and the degree.

    \item The Floer cochain group 
    \beqn
    CF(H; \Lambda_{\mb K}^\Gamma):= \Big\{ \sum_{i=1}^\infty a_i \ov{x}_i\ |\ a_i \in {\mb K},\ \ov{x}_i \in \tilde {\mc O}(H),\ \lim_{i\to \infty} {\mc A}_H(\ov{x}_i) = +\infty \Big\}
    \eeqn
    together with $-{\mc A}_H$ is a ${\mb Z}$-graded non-archimedean normed space over $\Lambda_{\mb K}^\Gamma$, where for $A \in \Gamma$, $q^A \in \Lambda_{\mb K}^\Gamma$ has grading $2c_1(A)$ and valuation $\omega(A)$.

    \item The variable of the differential corresponds to the ``input,'' or the negative end of the domain ${\mb R}\times S^1$ for the Floer cylinder.

    \item If $H$ is a $C^2$-small Morse function, the minimum (resp. maximum) of $H$, when equipped with the trivial capping, has cohomological degree $0$ (resp. $2n$). An ascending flow line of $H$ is also a Floer cylinder.
\end{enumerate}

With such conventions fixed, there are the following notations. The {\bf action spectrum} of a Hamiltonian $H$ is the set of critical values of ${\mc A}_H$:
\beqn
{\rm Spec}(H):= {\mc A}_H( \tilde {\mc O}(H)) \subset {\mb R}.
\eeqn
On the other hand, for a possibly degenerate capped orbit $\ov{x}\in \tilde {\mc O}(H)$, there is a notion of {\bf mean index}:
\beqn
\Delta(H, \ov{x}) \in {\mb R}
\eeqn
In the nondegenerate case, one has
\beqn
\Delta(H,\ol{x}) = \lim_{k \rightarrow \infty} \tfrac{1}{k} \mathrm{CZ}(H^{(k)}, \ol{x}^{(k)})
\eeqn
where $H^{(k)}(t, x):= k H(kt, x)$ is called the $k$-th iteration of $H$.


\subsubsection{Floer complex}

If $(M, \omega)$ is a general compact symplectic manifold, the Gromov--Floer compactification of moduli spaces of Floer trajectories fail to be manifolds even for generic choices of $(H,J)$ due to the presence of nontrivial automorphisms and thus the failure of transversality. Also, the traditional approach to moduli problems of $J$-holomorphic maps necessarily introduces rational numbers in the enumeration (see \cite{Fukaya_Ono} \cite{Liu_Tian_Floer} \cite{Pardon_virtual}). Nevertheless, in \cite{Bai_Xu_2025}, based on \cite{BaiXu-FPO1, Bai_Xu_Arnold} and \cite{Fukaya_Ono_integer}, the following construction is realized.

\begin{thm}\label{thm:Floer-cochain}(\cite[Theorem A, Theorem B]{Bai_Xu_2025})
Given a nondegenerate $H$ and a compatible almost complex structure $J$, there exists a cochain complex \beqn
CF(H, J, \Xi; \Lambda_{\mb Z}^\Gamma)
\eeqn
of free graded $\Lambda_{\mb Z}^\Gamma$-modules, generated by capped 1-periodic orbits of $H$, whose differential depends on a system of choices $\Xi$. 
     
Moreover, for any two choices of almost complex structures $J_1$ and $J_2$, and two choices of the auxiliary data $\Xi_1$ and $\Xi_2$, there exists a chain homotopy equivalence
\beq\label{continuation_map}
     \Phi^{12}: CF(H, J_1, \Xi_1; \Lambda_{\mb Z}^\Gamma) \to CF(H, J_2, \Xi_2; \Lambda_{\mb Z}^\Gamma)
     \eeq
     whose homotopy class is independent of choices.
\end{thm}

In particular, for any choice $(J, \Xi)$, the cohomology of $CF(H, J, \Xi; \Lambda_{\mb Z}^\Gamma)$ is independent of $(J, \Xi)$, and we denote  the cohomology by $HF(H; \Lambda_{\mb Z}^\Gamma)$. When the choice $(J, \Xi)$ is irrelevant to the discussion, we often abbreviate the complex by $CF(H; \Lambda_{\mb Z}^\Gamma)$. 

We often need to switch to coefficients in the (universal) Novikov ring/field with coefficients in a base field ${\mb K}$ (often ${\mb K} = \fp$). Denote
\begin{align*}
&\ CF(H; \Lambda_{\mb K}^\Gamma):= CF(H; \Lambda_{\mb Z}^\Gamma) \underset{\Lambda_{\mb Z}^\Gamma} {\otimes} \Lambda_{\mb K}^\Gamma,\ &\ CF(H; \Lambda_{\mb K}^{\rm univ}):= CF(H; \Lambda_{\mb Z}^\Gamma) \underset{\Lambda_{\mb Z}^\Gamma}{\otimes} \Lambda_{\mb K}^{\rm univ}
\end{align*}
and denote the corresponding Floer cohomology by $HF(H; \Lambda_{\mb K}^\Gamma)$ or $HF(H; \Lambda_{\mb K}^{\rm univ})$.

From now on, we will no longer use integer coefficients unless explicitly saying so. 

\begin{remark}
We refer the reader to the introduction of \cite{Bai_Xu_2025} for a complete exposition as we will mostly use formal properties of Hamiltonian Floer theory in this paper. Strictly speaking, the above statement is not exactly the same as the original statement in \emph{loc. cit.}, and we spell out the difference and discuss how to adapt it to the above setting. First, Theorem \ref{thm:Floer-cochain} is for Floer \emph{cohomology} rather than Floer homology, the latter of which is used in \cite[Theorem A]{Bai_Xu_2025}. 
Second, \cite[Theorem A]{Bai_Xu_2025} is stated for continuation maps which also allow the Hamiltonian $H$ to vary, and we state Theorem \ref{thm:Floer-cochain} in the above form as the theme of this paper is concerned with Hamiltonian dynamics. For all our discussions later on, similar comparisons need to be carried out in order to adapt the results from \cite{Bai_Xu_2025} to our setting. We will take this as an implicit assumption as the modifications are cosmetic.
\end{remark}

\subsubsection{Barcodes and bar-length spectrum}\label{subsubsec:bar}

By \cite[Theorem K]{Bai_Xu_2025}, the differential on $CF(H; \Lambda_{\mb K}^\Gamma)$ increases the action of critical points of $\mathcal{A}_H$ and the continuation map $\Phi^{12}$ (see \eqref{continuation_map}) respects the action filtration on $CF(H; \Lambda_{\mb K}^\Gamma)$. Therefore, we can look at Floer cohomology with action filtrations as in the classical setting.

A few more notations need to be specified. 
Let ${\mc A}_H: CF(H; \Lambda_{\mb K}^\Gamma) \to {\mb R} \cup \{+\infty\}$ be the energy filtration induced from the symplectic action. For any $a \in {\mb R}$, 
write
\beqn
CF(H; \Lambda_{\mb K}^\Gamma)^{>a} := \cl A_H^{-1}((a, +\infty]).
\eeqn
Then the filtered version of Floer cohomology, denoted $HF(H; \Lambda_{\mb K}^\Gamma)^{>a}$, is the cohomology of $CF(H; \Lambda_{\mb K}^\Gamma)^{>a}$, which only depends on $H$. More generally, given some interval $I=  (a,b) \subset \mathbb{R}$ such that $a,b \notin \mathrm{Spec}(H)$, one similarly defines 
\beqn
CF(H; \Lambda_{\mb K}^\Gamma)^I :=  CF(H; \Lambda_{\mb K}^\Gamma)^{>a} / CF(H; \Lambda_{\mb K}^\Gamma)^{>b},
\eeqn
and its associated cohomology by $HF(H)^I$.

Note that for all $a\in {\mb R}$ and $a\geq b$  there are maps 
\begin{align*}
&\ \iota_a: HF(H; \Lambda_{\mb K}^\Gamma)^{>a}  \rightarrow HF(H; \Lambda_{\mb K}^\Gamma),\ &\ \iota_{a,b}: HF(H; \Lambda_{\mb K}^\Gamma)^{>a} \to HF(H; \Lambda_{{\mb K}}^\Gamma)^{>b}.
\end{align*}

Important quantitative information of Hamiltonian diffeomorphisms is contained in the {\bf Floer barcodes}. In the current generality, one cannot apply the usual approach of defining barcodes via persistence modules. Instead, we follow the approach of Usher--Zhang \cite{usher-zhang}. Converting to the universal Novikov field $\Lambda_{\mb K}^{\rm univ}$, $CF(H; \Lambda_{\mb K}^{\rm univ})$ is a Floer-type complex (where the ${\mb Z}$-grading is collapsed) in the sense of \cite[Definition 4.1]{usher-zhang}, hence has well-defined barcodes. While the number of infinite bars coincides with the dimension of the cohomology over $\Lambda_{\mb K}^{\rm univ}$, the lengths of finite bars can be arranged in the non-decreasing order, which is called the {\bf bar-length spectrum} of $CF(H; \Lambda_{\mb K}^{\rm univ})$
\beqn
0 < \beta_1(H) \leq \cdots \leq \beta_K (H).
\eeqn
One can prove that the barcode spectrum only depends on the time-one map $\phi = \phi_{H}^1$ generated by $H$; hence we also denote it by 
\beqn
0 < \beta_1(\phi) \leq \cdots \leq \beta_K(\phi).
\eeqn

The bar-length spectra can also be defined in a different way. Such an idea originated from the Lagrangian Floer setting in \cite{FOOO-ring}. Suppose ${\mc O}(H) = \{ x_1, \ldots, x_N\}$. Choose their cappings $\ov{x}_1, \ldots, \ov{x}_N \in \wt{\mc O}(H)$. Then the free $\Lambda_{0, {\mb K}}^{\rm univ}$-module
\beq\label{canonical_submodule}
CF(H; \Lambda_{0, {\mb K}}^{\rm univ}) := {\rm span}_{\Lambda_{0, {\mb K}}^{\rm univ}} \Big\{ q^{-\cl A_H(\overline{x_1})} \overline{x_1}, \dots, q^{-\cl A_H(\overline{x_N})} \overline{x_N}  \Big\} \subseteq CF(H; \Lambda_{\mb K}^{\rm univ}).
\eeq
is canonically associated to $H$. The corresponding Floer cohomology $HF(H; \Lambda_{0, {\mb K}}^{\rm univ})$ splits as $\Lambda_{0, {\mb K}}^{\rm univ}$-modules into the free and torsion parts
\beqn
HF(H;\Lambda_{0, {\mb K}}^{\rm univ}) \cong \mathcal{F} \oplus \mathcal{T} \cong {\mc F} \oplus  \bigoplus_{1 \leq j \leq K} \Lambda_{0, {\mb K}}^{\rm univ} / q^{\beta_j(H)} \Lambda_{0, {\mb K}}^{\rm univ}
\eeqn
with $0 < \beta_1(H) \leq \cdots \leq \beta_K(H)$. Then \cite[Theorem 4.13]{usher-zhang} (see also \cite[Lemma 9]{S-HZ}) shows that $\beta_1(H) \leq \cdots \leq \beta_K(H)$ also gives rise to the bar-length spectrum of $\phi$. 

For our applications, we use the following two invariants constructed from the bar-length spectra. The first is $\beta_1(\phi)$, the length of the shortest finite bar. The second is
$$
\beta_{\mathrm{tot}} = \beta_1(\phi) + \cdots + \beta_{K}(\phi),
$$
called the {\bf total bar-length}. 

\subsubsection{Spectral invariants}
One can define important invariants of filtered Floer cohomology, known as spectral invariants. 
Recall that by \cite[Theorem A and B]{Bai_Xu_2025}, the Floer cohomology $HF(M, \omega; \Lambda_{\mb K}^\Gamma)$ is a well-defined symplectic invariant and there is a canonical isomorphism $HF(M, \omega; \Lambda_{\mb K}^\Gamma) \cong HF(H; \Lambda_{\mb K}^\Gamma)$.

\begin{defn}[Spectral invariants]\label{defn:spec-inv}
Suppose $\alpha \in HF(M; \Lambda^\Gamma_{\mb K})$. We define the {\bf spectral invariant} for a nondegenerate Hamiltonian $H$ by
\beqn
c(\alpha,H) := \text{sup} \Big\{ a \in \mathbb{R} \ | \ \alpha \in \text{im}(\iota_a: HF(H; \Lambda_{\mb K}^\Gamma)^{>a} \to HF(M; \Lambda_{\mb K}^\Gamma)) \Big\}.
\eeqn
\end{defn}

By \cite[Theorem L]{Bai_Xu_2025}, the spectral invariants satisfy the usual properties verified in other settings. In particular, the Lipschitz continuity allows us to extend $c(\alpha, H)$ to degenerate $H$.

\subsubsection{Local Floer cohomology}
\label{subsec:local-floer}
By ``local", we mean the following: suppose we fix some Hamiltonian $H$, and consider an isolated Hamiltonian orbit $x$ of $H$ (not necessarily non-degenerate). If we restrict to a small neighborhood $U$ of $x$, and perturb $H$ over that small neighborhood, then there is a well-defined notion of (${\mb Z}/2$-graded) local Floer cohomology of $H$ at $x$ with integer coefficients $HF^{\loc}(H,x; {\mb Z})$. If $\phi = \phi_{H}^1$ is the associated Hamiltonian diffeomorphism, we also use the notation $HF^{\loc}(\phi,x; {\mb Z})$. The idea is that one is conducting Floer cohomology ``near" $x$, by locally perturbing $H$ so that $x$ splits into some collection of non-degenerate orbits. We will give some details, but leave the bulk of the technical description to, e.g., \cite[Section 3.2]{Ginzburg-CC}.

More formally, we choose some non-degenerate $C^2$-small perturbation $G$ of $H$, such that $G = H$ outside of some small isolating neighbourhood (i.e. there are no other critical points of $H$ within) of $x$. This perturbation from $H$ to $G$ ``splits'' $x$ into nondegenerate 1-periodic orbits $x_1,\dots,x_j\in {\mc O}(G)$. Moreover, any capping $\ov{x}$ of $x$ canonically induces cappings $\ov{x}_1, \ldots, \ov{x}_j$. 
Then define the chain complex
\beqn
CF^{\loc}(H,G,\ov{x}; {\mb Z}) := {\mb Z} \langle \ov{x}_1, \cdots, \ov{x}_j \rangle.
\eeqn
One observes that if $G$ and $H$ are sufficiently close, then action bounds ensure the moduli space of flow lines for $G$ between any $x_{i_1}$ and $x_{i_2}$ which are contained in this isolating neighborhood is compact up to breaking.  
Hence there is a well defined notion of $HF^{\loc}(H,G,\ov{x}; {\mb Z})$. One can also show (by a standard continuation map argument) that this is independent of the choice of $G$ (as long as they remain close to $H$). Hence, we label the resulting local Floer cohomology $HF^{\loc}(H,\ov{x}; {\mb Z})$. Regarded as a ${\mb Z}/2$-graded module, $HF^{\loc}(H, \ov{x}; {\mb Z})$ is independent of the capping and only depends $x$ as a fixed point of the time-one map $\phi$. We denote this by
\beqn
HF^{\rm loc}(\phi, x; {\mb Z}).
\eeqn
By tensoring with ${\mb K}$ for any base field ${\mb K}$ on the chain complexes, one obtains the resulting local Floer cohomology with ${\mb K}$-coefficients. 

On the other hand, we can define $CF^{\rm loc}(H, G, x; \Lambda_{\mb K}^\Gamma)$ to be the action-completed direct sum
\beqn
\widehat{\bigoplus_{\ov{x}}} CF^{\rm loc}(H, G, \ov{x}; {\mb K})
\eeqn
over all cappings $\ov{x}$ of $x$, which is a module over $\Lambda_{\mb K}^\Gamma$; its homology is denoted by 
\beqn
HF^{\rm loc}(H, x; \Lambda_{\mb K}^\Gamma)
\eeqn
which has a well-defined ${\mb Z}$-grading.

We explain how the usual approach to local Floer cohomology is compatible with the Hamiltonian Floer theory construction in \cite{Bai_Xu_2025}. For more details, we refer the reader to the proof of \cite[Theorem G]{Bai_Xu_2025} (which compares the virtual approach therein with the classical approach for semipositive symplectic manifolds). It is known that for any $\omega$-compatible $J$, there exists an $\hbar > 0$ such that for any non-constant $J$-holomorphic sphere $v: S^2 \to M$, its energy satisfies $E(v) \geq \hbar$ (see, e.g., \cite[Proposition 4.1.4]{McDuff-Salamon}). Therefore, for a sufficiently small perturbation $G$, the moduli spaces of Floer trajectories defining $CF^{\loc}(H,G,x)$ do not have sphere bubbles in their compactification. Furthermore, the classical approach to Hamiltonian Floer theory ensures that for a generic choice of $G$, the relevant moduli spaces are cut out transversely. In this setting, running through the technical construction from \cite{Bai_Xu_2025}, one sees that one can take the FOP perturbations over the flow categories to be the unperturbed Kuranishi sections, whose zero loci are exactly these transverse moduli spaces. Therefore, when discussing local Floer cohomology, we do not distinguish the two approaches.

\subsubsection{Filtered Floer cohomology for degenerate Hamiltonians}
\label{subsec:filtered-floer}

We have discussed filtered Floer cohomology for nondegenerate Hamiltonians. For degenerate Hamiltonians, we use the work of Hein \cite{Hein-CCCY} to define their filtered Floer cohomology as follows.

Suppose we are given a degenerate Hamiltonian $H$ on 
a 
symplectic manifold $(M, \omega)$, and $a < b \notin \text{Spec}(H)$. For two Hamiltonians $K, K': S^1 \times M \rightarrow \mathbb{R}$, define
\beqn
\begin{split}
    K \leq K'\ &\ {\rm if}\ K(t, x) \leq K'(t, x)\ \forall (t, x) \in S^1 \times M,\\
    K < K'\ &\ {\rm if}\ K(t, x) < K'(t, x)\ \forall (t, x) \in S^1 \times M
    \end{split}
\eeqn
Further, denote by $P_{a, b}(H)$ the set of Hamiltonians $K$ satisfying 
\begin{itemize}
    \item $K$ is nondegenerate,
    \item $K < H$ \footnote{In \cite{Hein-CCCY} it was taken $K \leq H$. However,  proofs in \cite{Hein-CCCY}, for example, that of \cite[Corollary 3.1]{Hein-CCCY}, only work if we use the condition $K<H$. In fact, it is hard to find a cofinal sequence if we use $K \leq H$.}, and
    \item $a, b \notin {\rm Spec}(K)$.
 
    \end{itemize}
    $P_{a, b}(H)$ has a partial order defined by $K \leq K'$ pointwise. Then by \cite[Theorem B, Theorem K]{Bai_Xu_2025}, there is a well-defined continuation map 
    \beqn
    HF(K; \Lambda_{\mb K}^\Gamma)^{(a,b)} \rightarrow HF(K'; \Lambda_{\mb K}^\Gamma)^{(a,b)}
    \eeqn
induced from any monotone homotopy from $K$ to $K'$. Therefore, we can define 
\beq\label{eq: filtered-irrational} 
HF(H; \Lambda_{\mb K}^\Gamma)^{(a,b)} := \varinjlim_{K \in P_{a, b}(H)} HF(K; \Lambda_{\mb K}^\Gamma)^{(a,b)}.
\eeq
The details are provided in the cited \cite[Section 2.3]{Hein-CCCY}, which only uses formal properties of (filtered) continuation maps. In particular, one can show that this definition reduces to the previous definition in the case of $H$ nondegenerate or $(M, \omega)$ rational, so we will take \eqref{eq: filtered-irrational} to be the general definition.

\subsection{Equivariant cohomology}
\label{subsec:equivariant-cohomology}

In this subsection, we recall some classical topological and algebraic constructions of $\Z/p$-equivariant cohomology.\footnote{The algebraic constructions presented here all have a characteristic $2$ version, which takes slightly different forms. }

\subsubsection{The algebraic Borel construction}\label{subsubsec:algebraic-Borel}

We first describe the algebraic way of defining equivariant cohomology from a $\zp$-action on a complex. Note that for any odd prime $p$, 
\beqn
H^*(B \mathbb{Z}/p;\mathbb{F}_p) \cong \mathbb{F}_p[\![t]\!][\theta]/(\theta^2),
\eeqn
where $t$ has degree $2$ and $\theta$ has degree $1$. In this paper, the following notations will be used frequently. Denote
\begin{align*}
&\ \kzp:= \fp [\![t]\!],\ &\ \kp:= \fp (\!(t)\!);
\end{align*}
\begin{align*}
&\ E(\theta):= \fp[\theta]/\langle \theta^2 \rangle,\ &\ \rp:= \fp[\![t]\!]\otimes E(\theta),
\end{align*}
so we have
\beqn
H^*(B \mathbb{Z}/p;\mathbb{F}_p) \cong \rp.
\eeqn

\begin{defn}[Algebraic Borel Construction] Let $(C, d)$ be a cochain complex over $\fp$ with a $\zp$-action generated by a chain map $\iota: C \to C$. The {\bf Borel equivariant cochain complex} associated to $C$ and $\iota$ is the  graded cochain complex 
\begin{equation}\label{eq:equivariant-cohomology}
C_{\zp}:= C \otimes
 \rp,
\end{equation}
whose generators are of the form $x t^\alpha \theta^\beta$ for $x \in C$, $\alpha \geq 0$ and $\beta = 0, 1$. The  equivariant differential $d_{\mathbb{Z}/p}$ on $C_{\zp}$ is $t$-linear and determined by the following rules \begin{equation} \label{equation:equivariant-differential}\begin{array}{lclll} d_{\mathbb{Z}/p}(x \otimes t^i) & = & (dx)  \otimes  t^i & + & (-1)^{|x|}(\iota x - x) \otimes   \theta t^i \\ d_{\mathbb{Z}/p}(x \otimes \theta t^i) & = & (dx)  \otimes  \theta t^i & + & (-1)^{|x|} (1+\iota + \cdots + \iota^{p-1}) x \otimes  t^{i+1}.\end{array}\end{equation}
The equivariant cohomology $H_{\zp}(C)$ is the cohomology of $(C_{\zp}, d_{\zp})$.
\end{defn}

\begin{remark}
Notice that the equivariant differential is not linear in $\theta$; hence $C_{\zp}$ is only a complex of $\kzp$-modules. However, the cohomology of $C_{\zp}$ is a module over $\rp$. Indeed, a cochain map $\tilde\theta: C_{\zp}^* \to C_{\zp}^{*+1}$ is defined in \cite[Section 2d]{covariant-constant} by 
\beqn
\begin{split}
\tilde\theta( xt^k) = &\ (-1)^{|x|} xt^k \theta,\\
\tilde\theta( xt^k \theta) = &\ (-1)^{|x|} \Big( \iota x + 2 \iota^2 x + \cdots + (p-1) \iota^{p-1} x \Big) t^{k+1}.
\end{split}
\eeqn
On cohomology level it induces a well-defined ``multiplication by $\theta$'' map and hence $H_{\zp}(C)$ becomes an $\rp$-module.
\end{remark}

There are two special cases we will use in this paper. 

\begin{example}
If $C$ is equipped with the trivial $\zp$-action, then the equivariant differential $\rp$-linearly extends the differential on $C$ and the equivariant cohomology is 
\beqn
H_{\zp}(C) \cong H(C) \otimes \rp.
\eeqn
Notice that this is an isomorphism of $\rp$-modules.
\end{example}

\begin{example}
For a general cochain complex $C$, the $p$-fold tensor power $C^{\otimes p} = C \otimes \dots \otimes C$ has a natural $\zp$-action by cyclically permuting the factors, resulting in an equivariant cochain complex with underlying space
\beqn
C^{\otimes p}_{\zp}:= C^{\otimes p} \otimes \rp.
\eeqn
\end{example}

The case of complexes over Novikov ring/fields deserves more careful treatment. We work over the universal Novikov field abbreviated as $\Lambda = \Lambda^{\rm univ}$ where we suppress the base field ${\mb K} = {\mb F}_p$ from the notation; the cases with other versions of Novikov rings are similar. Let $C$ be a complex of modules over $\Lambda$ equipped with a filtration
\beqn
{\mc A}_C: C \setminus \{0\}  \to {\mb R}
\eeqn
satisfying
\beqn
{\mc A}_C(q^\lambda x) = {\mc A}_C(x) + \lambda.
\eeqn
If $(C, {\mc A}_C)$ and $(D, {\mc A}_D)$ are two such complexes, then the naive tensor product $C \otimes D$ over $\Lambda$ has a completion 
\beqn
C \widehat{\otimes}_{\Lambda} D
\eeqn
with respect to ${\mc A}_C + {\mc A}_D$, which is again a $\Lambda$-module. Notice that when both $C$ and $D$ are finitely generated, the completed tensor product agrees with the algebraic one; this is the only case we consider in this paper. Then 
\beqn
C^{\otimes p}:= \overbrace{ C \underset{\Lambda}{\otimes} \cdots \underset{\Lambda}{\otimes} C}^{p}
\eeqn
denotes the $p$-fold (completed) tensor product of $C$, which is still a $\Lambda$-module with a filtration ${\mc A}_C^{{\otimes} p}$. Moreover, we can form the tensor product over
\beqn
C^{\otimes p}  {\otimes} \rp = C^{{\otimes} p}  {\otimes} \fp \pst  {\otimes} E(\theta)
\eeqn
which is a module over the (algebraic) tensor product $\Lambda \otimes \fp \pst = \Lambda \otimes \kzp$. The tensor product $\Lambda \otimes \fp \pst$ can be completed to 
\beqn
\Lambda_{\kzp} = \Lambda_{\fp\pst}.\footnote{There are many reasons for choosing this completion over $\Lambda \pst$. For example, there is still a valuation $\nu: \Lambda_{{\mb F}_p\pst} \to {\mb R} \cup \{+\infty\}$.}
\eeqn
Then define
\beqn
C^{{\otimes} p}_{\zp}:= \left( C^{{\otimes}p}  \underset{\fp}{\otimes} \rp \right) \underset{ \Lambda \otimes \kzp}{\otimes} \Lambda_{\kzp}
\eeqn
which has a $\Lambda_{\kzp}$-linear differential $d_{\zp}$ which takes the same form as \eqref{equation:equivariant-differential}. Finally, the corresponding definition of $C^{\otimes p}_{\zp}$ for the case that $C$ is a finitely generated free $\Lambda_0^{\rm univ}$-module is the same, which gives a complex of modules over $\Lambda_{0, \kzp}^{\rm univ}$. The two versions of equivariant cohomology will be distinguished in notations as 
\begin{align*}
&\ H_{\zp}( C^{\otimes p}; \Lambda_{\kzp}),\ &\ H_{\zp}(C^{\otimes p}; \Lambda_{0, \kzp}^{\rm univ}).
\end{align*}

Now we introduce the so-called {\bf quasi-Frobenius map}. We explain the case without Novikov coefficients. Let $C$ be a complex of $\fp$-modules with the trivial $\zp$-action. For any cocycle $x \in C$, the $p$-fold power $x^{\otimes p}$ is an equivariant cocycle in $C_{\zp}^{\otimes p}$. The association $x \mapsto x^{\otimes p}$ induces a well-defined map
\beqn
qF: H(C) \to H_{\zp}(C^{\otimes p})
\eeqn
(see \cite[Lemma 2.5]{covariant-constant}). The map is only additive after inverting the variable $t$. In other words, it is additive after turning to the Tate version of equivariant cohomology, which is introduced right below.

\subsubsection{Tate cohomology}
To discuss equivariant localizations, we work with the Tate construction as follows. As above, consider a cochain complex $C$ over $\mathbb{F}_p$, equipped with a $\mathbb{Z}/p$-action generated by $\iota: C \to C$. Algebraically, the Tate complex is the graded cochain complex
\beqn
\widehat{C}_{\Z/p} := C_{\Z/p} \underset{\kzp}{\otimes} \kp \cong C_{\zp}[t^{-1}]
\eeqn
with $C_{\Z/p}$ defined in \eqref{eq:equivariant-cohomology}, equipped with the differential which is the $\kp$-linear extension of \eqref{equation:equivariant-differential}. We denote the cohomology by $\widehat{H}_{\Z / p}(C)$. 

When $C$ is a finitely generated module over Novikov rings, the Tate cohomology is defined similarly. For example, suppose $C$ is defined over the universal Novikov ring $ \Lambda^{\rm univ}$. Then one defines
\beqn
\wh C_{\zp}: = C \underset{\Lambda^{\rm univ}}{\otimes} \Lambda_{\kp}^{\rm univ}
\eeqn
which contains the Borel version $C_{\zp}$. The differential $d_{\zp}$ extends to a $\Lambda_{\kp}^{\rm univ}$-linear one and leads to the Tate cohomology $\wh{H}_{\zp}(C; \Lambda_{\kp}^{\rm univ})$.

Regarding applications in Hamiltonian dynamics, one considers the quasi-Frobenius map for modules over the universal Novikov ring $\Lambda_0^{\rm univ}$. 

\begin{lemma}\label{lemma:quasi-Fr}(\cite[Section 7]{S-HZ})
Let $C$ be a cochain complex of finitely generated free $\Lambda_0^{\rm univ}$-modules. Then the quasi-Frobenius map $qF$ induces an isomorphism of $\Lambda_{0, \kp}^{\rm univ}$-modules
\begin{equation}\label{eqn:Tate-power-iso}
  qF:  r_p^* \widehat{H}_{\Z/p}(C; \Lambda_{0, \kp}^{\rm univ}) \to \widehat{H}_{\Z/p}(C^{\otimes p}; \Lambda_{0, \kp}^{\rm univ} )
\end{equation}
where the left-hand-side of \eqref{eqn:Tate-power-iso} is regarded as the base-change
\beqn
\widehat H_{\zp}(C; \Lambda_{0, \kp}^{\rm univ} ) \underset{\Lambda_{0, \kp}^{\rm univ}}{\otimes} \Lambda_{0, \kp }^{\rm univ}
\eeqn
via the ring isomorphism $r_p: \Lambda_{0, \kp}^{\rm univ} \to \Lambda_{0, \kp}^{\rm univ}$ induced by $q \mapsto q^{\frac{1}{p}}$.
\end{lemma}

The above lemma has the following consequence for the Floer complex $C = CF(H; \Lambda_0^{\rm univ})$ as defined in \eqref{canonical_submodule}.
Notice that the Tate cochain complex
\beqn
\widehat{CF(H; \Lambda_0^{\rm univ})_{\zp}^{\otimes p}}
\eeqn
still has the energy filtration $\wh{\mc A}_C^{\otimes p}$ induced from the energy filtration on $C$. As the Tate cohomology is a finitely generated module over $\Lambda_{0, \kp}^{\rm univ}$ while $\kp$ is a field, one hence has a bar-length spectrum 
\beqn
\wh\beta_1^{\otimes p}(H) \leq \cdots \leq \wh \beta_{\tilde{K}}^{\otimes p}(H).
\eeqn
It is clear that $\wh \beta_j^{\otimes p}(H)$ only depends on the time-1 map $\phi$, hence we also denote them by $\wh \beta_j^{\otimes p}(\phi)$. Then by Lemma \ref{lemma:quasi-Fr}, when $p$ is odd, $\tilde{K} = 2K$\footnote{When $p=2$, one shall have $\tilde K = K$.} and
\beqn
\wh \beta^{\otimes p }_{2j-1}(\phi) = \wh \beta^{\otimes p }_{2j}(\phi) = p   \beta_j(\phi), \quad \quad \quad \forall 1 \leq j \leq K.
\eeqn

\subsection{Equivariant cohomology via Morse theory}
\label{subsubsec:borel-construction}

It is often more convenient to realize the Borel construction on the chain level using equivariant Morse theory. We choose the $\zp$-equivariant Borel model as follows. Consider
\beqn
S^\infty:= \Big\{ v = (v_0, v_1, \cdots)\ |\ v_i \in {\mb C},\ i \gg 0 \Longrightarrow v_i = 0,\ \sum_{i=0}^\infty |v_i|^2 = 1 \Big\}
\eeqn
which has a free $\zp$-action by simultaneously multiplying coordinates $v_i$ by $p$-th roots of unity. The quotient space is a classifying space for $\zp$, denoted by $B\zp$. On the other hand, notice that there is a free ${\mb Z}_{\geq 0}$-action generated by 
\beqn
\tau(v_0, v_1, \cdots) = (0, v_0, v_1, \cdots)
\eeqn
which commutes with the $\zp$-action. 

$S^\infty$ admits a finite-dimensional approximation using smooth manifolds $S^{2n-1}$. Any object on $S^\infty$ is called smooth if its restriction to each truncation $S^{2n-1}$ is smooth in the usual sense. Now we choose a special Morse function on $S^\infty$; we do not intend to compare it with other choices. Define
\begin{equation}
\begin{split}\label{equation:morse-function-s-inf}
g:S^{\infty} \to &\ \mathbb{R}\\
    (v_0, v_1, \cdots) \mapsto &\ \sum_{l \geq 0} \Big( l  |v_l|^2 + \epsilon \cdot \text{Re}(v_l^p) \Big)
    \end{split}
\end{equation}
which was the choice used in \cite[Section 4]{SZhao-pants}. Notice that $g$ is invariant under the $\zp$-action and satisfies $g(\tau(v)) = g(v) + 1$. Define the Morse index in the usual way, which is the number of negative eigenvalues of the Hessian. 

One can explicitly write down the coordinates of critical points. For each even degree $2l$, the $p$ critical points of index $2l$ are
\beq\label{critical_points_1}
w_{2l}^k = ( \underbrace{0, \ldots, 0}_{l}, e^{\frac{2k \pi {\bf i}}{p} - \frac{\pi {\bf i}}{p}}, 0, \cdots),\ k = 1, 2, \ldots, p;
\eeq
for each odd degree $2l+1$, the $p$ critical points of index $2l+1$ are
\beq\label{critical_point_2}
w_{2l+1}^k = (\underbrace{0, \ldots, 0}_{l}, e^{\frac{2k \pi {\bf i}}{p}}, 0, \cdots),\ k = 1, 2, \ldots, p.
\eeq

Then one obtains a Morse cochain complex associated to $g$ and a suitably chosen Riemannian metric invariant under $\zp \times {\mb Z}_{\geq 0}$-action. More precisely, one defines
\beqn
CM(g) = \Big\{ \sum_{i=1}^\infty a_i w_i\ |\ a_i \in \fp,\ w_i \in {\rm crit}(g),\ \lim_{i \to \infty} g (w_i) = +\infty \Big\}
\eeqn
which is a formal completion of the usual Morse cochain complex of finite sums. For each degree $k\geq 0$, there are exactly $p$ generators. The $\zp$-invariant part of the chain complex recovers the $\zp$-equivariant cohomology of a point; indeed, in $\fp$-coefficients, the $\zp$-invariant part has a trivial differential and is isomorphic to the group $\rp$.

\subsubsection{Ring structure of $\rp$}\label{subsubsec:ring-BZ/p}

We also realize the ring structure of $\rp \cong H_\zp^*({\rm pt})$ using the Morse model. Let $T_{2, 1}$ denote the tree which is the disjoint union of two incoming edges $T_1\cong T_2 \cong (-\infty, 0]$ and an outgoing edge $T_\infty \cong [0, +\infty)$. Let $s \in T_{2,1}$ denote the coordinate on each edge; when there is no ambiguity, $s$ corresponds to a real number. To achieve transversality, we choose a $\zp$-invariant family of vector fields $Y_s$ parametrized by $s \in T_{2,1}$
\begin{enumerate}

\item $Y (s) = 0$ when $|s|\gg 0$. 

\item For each $n$, $Y|_{S^{2n-1}}$ is tangent to $S^{2n-1} \subset S^\infty$ and $\nabla Y$ is sufficiently small in the normal direction of $S^{2n-1}$.
\end{enumerate}
Let $Y_i$ denote the restriction of $Y$ to $T_i \subset T_{2,1}$. A {\bf $Y$-perturbed flow tree} is a triple $\zeta = (\zeta_1, \zeta_2, \zeta_\infty): T_{2,1} \to S^\infty$ which solves the equation
\beqn
\zeta' (s) = \nabla g(\zeta(s)) + Y_s(\zeta(s))
\eeqn
subject to the incidence relation
\beqn
\zeta_1(0) = \zeta_2(0) = \zeta_\infty(0).
\eeqn
Each flow tree converges to a triple of critical points of $g$ at $\infty$ of the edges.

Given a triple of critical points $w_1, w_2, w_\infty \in {\rm crit}(g)$, let $S^{2n-1}\subset S^\infty$ be the smallest truncation which contains all of them. One can see that a flow tree connecting $w_1, w_2, w_\infty$ is contained in $S^{2n-1}$. We say that a $Y$-perturbed flow tree is regular for $(w_1, w_2, w_\infty)$ if it is regular as a flow tree contained in $S^{2n-1}$.\footnote{In fact if it is regular in $S^{2n-1}$, then it is regular in any larger truncation.}
 
For a generic perturbation $Y$ and each triple of critical points $w_1, w_2, w_\infty$, one obtains a regular moduli space of solutions, whose mod $p$ count induces a $\zp$-equivariant map
\beqn
CM(g) \otimes CM(g) \to CM(g).
\eeqn
Hence after restriction to the $\zp$-invariant part, one obtains
\beqn
\rp \otimes  \rp \cong CM_{\zp}(g) \otimes  CM_{\zp} (g) \to CM_{\zp} (g) \cong \rp,
\eeqn 
from which one obtains a bilinear map $\rp \times \rp \to \rp$. The usual Morse theoretic technique can be used to show that the bilinear map is independent of the choices of $Y$, and satisfies associativity. Hence $\rp$ is equipped with a structure of a ${\mb Z}$-graded, associative and supercommutative algebra over $\fp$.

Using any finite truncation of $S^\infty$, and because we are passing to the formal completion when defining the cochain complex, we see that the ring structure defined by Morse cohomology agrees with the standard ring structure on $\rp \cong \fp \pst \otimes E(\theta)$ (when $p>2$) or $\rp \cong {\mb F}_2 [\![\theta]\!]$ (when $p = 2$).



\subsubsection{Equivariant Morse cohomology when equivariant transversality holds}

Recall that given a cochain complex $C$ with a $\zp$-action, one has defined an algebraic version of the $\zp$-equivariant cochain complex $C_{\zp}$ as shown. When $C$ is realized as a Morse cochain complex for a $\zp$-invariant Morse--Smale pair $(f, h)$ on a smooth manifold $M$ with a $\zp$-action, this algebraic construction can be realized Morse-theoretically. Indeed, the product function $f\times g$, by which we mean $\pi_{M}^*f + \pi_{S^{\infty}}^*g$ where $\pi_{M}, \pi_{S^{\infty}}$ are the projections from the product, and the product metric on $M \times S^\infty$ form a Morse--Smale pair and hence a Morse complex $CM(f\times g) \cong CM (f) \otimes CM (g)$. The $\zp$-invariant part is a subcomplex, denoted by 
\beqn
CM_{\zp} (f):= CM(f \times g)^{\inv}.
\eeqn
Notice that this complex is isomorphic to the algebraically defined one, but the isomorphism is not canonical. For example, the following
\beqn
\begin{split}
    CM(f) \otimes \rp &\ \to CM(f \times g)^\inv \\
    x \otimes t^i & \mapsto \sum_{l=1}^{p} (\iota^l x) \otimes w_{2i}^l,\\
    x \otimes t^i \theta & \mapsto \sum_{l=1}^p (\iota^l x) \otimes w_{2i+1}^l
    \end{split}
\eeqn
is an isomorphism of complexes, where $\iota \in \zp$ is the generator.

Notice that when $\zp$ acts trivially on $M$, the equivariant Morse  complex is then isomorphic to $CM(f) \otimes \rp$ and the equivariant Morse cohomology is isomorphic to $HM(f) \otimes \rp$. 


In the current situation, the structure of $\rp$-modules on the equivariant cohomology $HM_{\zp}(f)$ can be defined both algebraically and Morse-theoretically. Indeed they coincide.

\begin{prop}
As $\rp$-modules, the equivariant Morse cohomology $HM_{\zp}(f)$ and the equivariant cohomology of the complex $CM(f)$ coincide. In particular, when $\zp$ acts trivially on $M$, as $\rp$-modules, one has 
\beqn
HM_{\zp}(f) \cong HM(f) \otimes \rp.
\eeqn
\end{prop}

\begin{proof}
Indeed, the algebraic definition of $C_\zp(CM(f))$ uses a specific $\zp$-invariant cellular decomposition of $S^\infty$ (see details in \cite{covariant-constant}) and the $\rp$-module structure (indeed only the cochain map $\tilde\theta$) depends on a $\zp$-equivariant cellular approximation of the diagonal map of $S^\infty$. The Morse-theoretic definition, which counts Y-shaped flow trees, gives another equivariant cellular approximation. As the cellular approximations are (chain) homotopic, the induced multiplications agree on the cohomology level. 
\end{proof}

\subsubsection{Equivariant Morse cohomology without assuming equivariant transversality}

Now let $M$ be a manifold with a $\zp$-action and $f: M \to {\mb R}$ be a $\zp$-invariant Morse function.\footnote{In general, by perturbation one can achieve the Morse condition equivariantly but not the Morse--Smale condition.} In order to define the equivariant Morse cochain complex, we use the following Borel type construction. Consider again the function $f\times g: M \times S^\infty \to {\mb R}$. Now the $\zp$-action on $M \times S^\infty$ is free, allowing us to achieve equivariant transversality. 

There are many concrete choices of perturbations. For example, one can perturb the function $f$ to allow it to depend on $v \in S^\infty$, i.e., a $\zp$-invariant function $(v, x) \mapsto f_v(x)$ on $S^\infty \times M$ such that $f_v(x) = f(x)$ when $(v, x)$ is close to a critical point of $f \times g$. One can further ensure that the set of critical points is unchanged under the perturbation provided that the perturbation is sufficiently small in $C^1$-norm. Then one can consider the coupled flow line equation for $(\eta, \zeta): {\mb R} \to M \times S^\infty$:
\begin{align*}
    &\ \eta'(s) = \nabla f_{\zeta (s)} (\eta (s)),\ &\ \zeta'(s) = \nabla g(\zeta (s)),
\end{align*}
which has again an obvious translation symmetry. Define the Morse cochain complex as the (completed) tensor product
\beqn
CM (f) \otimes CM (g)
\eeqn
whose differential counts rigid perturbed flow lines. Define
\beqn
CM_{\zp}(f):= (CM (f) \otimes CM (g))^\inv
\eeqn
as a subcomplex, whose chain homotopy type is independent of the perturbation. 

By using perturbed flow trees in $S^\infty$, one can also define a module structure of $HM_{\zp}(f)$ over $\rp$ induced by a chain-level bilinear map
\beqn
CM_{\zp} (f) \otimes CM _{\zp} (g) \to CM_{\zp}(f).
\eeqn
More precisely, choose a perturbation $Y$ over the tree $T_{2,1}$ which can be used to define the ring structure on $S^\infty$, cf. Section \ref{subsubsec:ring-BZ/p}. Choose another generic family of functions 
\beqn
f_{s, v}: M \to {\mb R}
\eeqn
parametrized by $s \in {\mb R}$ and $v \in S^\infty$ such that 1) $f_{s, v} = f$ near $s = 0$ and 2) $f_{s, v} = f_v$ for $|s|\gg 0$ and 3) $f_{s, \gamma v} (\gamma x) = f_{s, v}(x)$ for all $\gamma \in \zp$. Consider pairs $(\xi, \eta)$ where $\eta = (\zeta_1, \zeta_2, \zeta_\infty): T_{2,1} \to S^\infty$ is a $Y$-perturbed flow tree and $\xi: {\mb R} \to M $ solves the equation 
\beqn
\xi'(s) = \left\{ \begin{array}{cc}  \nabla f_{s, \zeta_1(s)}(\xi(s)),\ &\ s \leq 0,\ \\
 \nabla f_{s, \zeta_\infty(s)}(\xi(s)),\ &\ s \geq 0 \end{array}  \right.
\eeqn
The coefficients of the equation on $\xi$ are smooth. We then impose critical point asymptotics $\xi(\pm \infty) \in {\rm crit}(f)$. By choosing a generic $f_{s, v}$, one obtains equivariant transversality and hence defines a chain map
\beqn
\big( CM(f) \otimes CM(g) \big) \otimes CM(g) \to CM(f) \otimes CM(g).
\eeqn
Its restriction to the $\zp$-invariant part gives
\beqn
CM_{\zp}(f) \otimes CM_{\zp}(g) \hookrightarrow (CM(f) \otimes CM(g) \otimes CM(g))^{\zp} \to CM_{\zp}(f)
\eeqn
hence a bilinear map
\beqn
HM_{\zp}(f) \otimes \rp \to HM_{\zp}(f).
\eeqn
One can further show that this map defines a $\rp$-module structure on $HM_{\zp}(f)$ using a standard TQFT argument.

One can further define a multiplicative structure on $HM_{\zp}(f)$ via the cochain-level map
\beqn
CM_{\zp} (f) \otimes  CM_{\zp}(f) \to CM_{\zp}(f)
\eeqn
via gradient flow trees in $M \times S^\infty$, whose associativity can be verified in the standard way. Lastly, because the multiplicative structure respects the $\rp$-module structure, it descends to a map
\beqn
HM_{\zp}(f) \underset{\rp}{\otimes} HM_{\zp}(f) \to HM_{\zp}(f).
\eeqn

We comment on the case when $\zp$ acts trivially on $M$. In this case, $HM_{\zp}(f)$ is isomorphic to $HM(f) \otimes \rp$ as $\rp$-algebras.

\subsection{The quantum Steenrod power operation}
\label{subsec:quantumStpower}

There are a few \emph{a priori} different definitions of the quantum Steenrod operation $\qst_p$ involved in this paper as well as related literature. For (weakly) monotone $(M, \omega)$, Seidel--Wilkins \cite{covariant-constant} defined $\qst_p$ using the cellular chain model of the classifying space $B\zp$, which matches the one used in \cite{Seidel-formal}. In this paper, for our purpose, we will combine the construction in \cite{covariant-constant} with the Morse model of $B\zp$ to give an alternate definition for (weakly) monotone $(M, \omega)$, which will be used only in proving the Kunneth property (Theorem \ref{thm: Kunneth}) of $\qst_p$ in the monotone case. In Appendix \ref{sec:alternative-qst} we will show that the two definitions agree. On the other hand, the first and the fourth authors \cite{Bai_Xu_2025} defined $\qst_p$ for general compact $(M, \omega)$, which will be recalled and used in other applications in this paper. In \cite{Bai_Xu_2025}, it was shown that the general definition agrees with the above more traditional definitions for (weakly) monotone targets.

We first give a formal description of the quantum Steenrod operation using the Morse model, including the particular moduli spaces, without discussing perturbations. Then we recall the theorem in \cite{Bai_Xu_2025} for the construction in the general case. The construction using inhomogeneous perturbations will be recalled in the section where we prove the Kunneth property of the quantum Steenrod operation for monotone targets. 

Consider a closed symplectic manifold $(M, \omega)$ which has the associated Novikov ring $\Lambda^\Gamma$ (see \eqref{eqn_Novikov_ring}) with base field ${\mb K} = \fp$. Choose a Morse function $f$ and a Morse--Smale metric on $M$, which give rise to a Morse chain complex $CM(f)$. Choose an $\omega$-compatible almost complex structure $J$.

We use a specific domain for defining the moduli space. On the other hand, consider the $p$-fold cover $z \mapsto z^p$ from $S^2 \to S^2$. Denote the source $S^2$ by $\Sigma_p$, which has a $\zp$-action. Specify the special points
\beqn
z_k = e^{\frac{2\pi k {\bf i}}{p}},\ k = 1, 2, \ldots, p,\ z_\infty = \infty \in \Sigma_p.
\eeqn 

Now we describe the moduli spaces. For each collection of critical points $y_1, \ldots, y_p, y_\infty \in {\rm crit}(f)$ and a given homology class $A \in H_2(M; {\mb Z})$, one can consider tuples $(u, \gamma_1, \ldots, \gamma_p, \gamma_\infty)$ where $u: \Sigma_p \to M$ is a $J$-holomorphic map in class $A$, $\gamma_1, \ldots, \gamma_p: (-\infty, 0] \to M$ are gradient lines of $f$ which are asymptotic to $y_1, \ldots, y_p$ at $-\infty$ and $\gamma_\infty: [0, +\infty) \to M$ is a gradient line of $f$ which is asymptotic to $y_\infty$ at $+\infty$; they are also required to satisfy the matching condition 
\beqn
u(z_k) = \gamma_k(0),\ k = 1, \ldots, p, \infty.
\eeqn
To count such objects equivariantly, we need to use the Morse-theoretic Borel construction and achieve $\zp$-equivariant transversality. To this end, we couple the map $u$ and the gradient rays $\gamma_1, \ldots, \gamma_p, \gamma_\infty$ with (parametrized) gradient lines in $S^\infty$. More precisely, given $y_k \in {\rm crit}(f)$, $k = 1, \ldots, p, \infty$, $A \in H_2(M; {\mb Z})$, and $w_\pm \in {\rm crit}(g) \subset S^\infty$, consider the moduli space 
\beqn
{\mc M}_{w_-, w_+} (A; y_1, \ldots, y_p, y_\infty)
\eeqn
of tuples $(u, \gamma_1, \ldots, \gamma_p, \gamma_\infty, w)$ where $(u, \gamma_1, \ldots, \gamma_p, \gamma_\infty)$ are described as above while $w: {\mb R} \to S^\infty$ is a gradient flow line of $g$ which converges to $w_\pm$ at $\pm \infty$. Notice that as the ${\mb Z}_{\geq 0} \times \zp$-action on $S^\infty$ is free, there is a free action on the disjoint union of these moduli spaces over a fixed $A \in H_2(M; {\mb Z})$ and all possible asymptotic limits of the Morse flow lines/rays. This is the reason why equivariant transversality is possible. 

Using either inhomogeneous terms or abstract FOP perturbations, one can define $\fp$-valued signed counts of the moduli spaces ${\mc M}_{w_-, w_+}(A; y_1, \ldots, y_p, y_\infty)$ (which vanish by definition if the expected dimension is not zero). One then obtains a $\fp$-linear chain map
\beqn
CM(f)^{\otimes p} \otimes CM (g) \to CM (f \times g)   \cong CM (f) \otimes CM (g)  
\eeqn
whose value on the generator $y_1\otimes \cdots \otimes y_p \otimes w_-$ is 
\beqn
\sum_{w_+} \# {y_\infty, w_+}_{w_-, w_+}(A; y_1, \ldots, y_p, y_\infty) (y_\infty \otimes w_+).
\eeqn
The ${\mb Z}_{\geq 0} \times \zp$-equivariance on the perturbation implies that the above chain map admits a well-defined restriction to the $\zp$-invariant part (with respect to the action on $C(g)$ while permuting the $p$ factors in $C(f)^{\otimes p}$), defining an $\fp$-linear cochain map
\beqn
\wt{\qst}_{p, A}: CM(f)^{\otimes p}_{\zp} \to \Big( CM(f) \otimes CM(g) \Big)^\inv \cong CM(f) \otimes \rp
\eeqn
which induces a map on cohomology
\beqn
\wt{\qst}_{p, A}: H(CM (f)^{\otimes p}_{\zp}) \to HM (f) \otimes \rp.
\eeqn
The degree $A$ part of the quantum Steenrod operation is the composition of the above map with the quasi-Frobenius map
\beqn
qF: H (M; \fp) \cong HM(f) \to H (CM (f)^{\otimes p}_{\zp}).
\eeqn
Define 
\beqn
\qst_{p, A}:= \wt{\qst}_{p, A} \circ qF: H(M;\fp) \to H(CM (f)) \otimes \rp \cong H(M) \otimes  \rp.
\eeqn
Finally, the following weighted sum is well-defined after suitable completion:
\beqn
\qst_p:= \sum_{A \in H_2(M; {\mb Z})} q^A \qst_{p, A}: H(M; \Lambda^\Gamma) \to H(M; \Lambda_\rp^\Gamma).
\eeqn

\begin{thm}(\cite[Theorem H]{Bai_Xu_2025}) For any compact symplectic manifold $(M, \omega)$, the quantum Steenrod operation
\beqn
\qst_p: H (M; \Lambda^\Gamma) \to H^*(M; \Lambda_\rp^\Gamma)
\eeqn
is well-defined. Moreover, its classical part coincides with the total Steenrod operation
\beqn
St_p: H (M; {\mb F}_p) \to H (M;\fp) \otimes \rp.
\eeqn
When $(M, \omega)$ is weakly monotone, this definition agrees with that of \cite{covariant-constant}.
\end{thm}

\begin{remark}
    The comparison result \cite[Theorem 24.6]{Bai_Xu_2025} is established by taking the Morse-theoretic model in the definition of $\qst_p$ in the weakly semi-positive setting, which is different from the cellular models in \cite{covariant-constant}. The claimed result follows by combining \cite[Theorem 24.6]{Bai_Xu_2025} and the results in Appendix \ref{appendixa}.
\end{remark}

\subsection{Equivariant Floer cohomology}\label{subsec:equiv-Floer-coho}

To study prime iterations of Hamiltonian diffeomorphisms, we need to study equivariant Hamiltonian Floer cohomology. In this subsection, we review the abstract construction of \cite[Part 4]{Bai_Xu_2025} in the general setting using virtual perturbations. We also review the definition of the local equivariant cohomology (which is equivalent to the definition of \cite{SZhao-pants}). 

We quickly set up the relevant notations. Let $p$ be a prime. Given the Hamiltonian $H: M \times S^1 \to {\mb R}$, define its $p$-th iteration $H^{(p)}: M \times S^1 \to {\mb R}$ by 
$$
H^{(p)}(x, t) := pH(x, {pt}),
$$
whose associated time-1 map is the Hamiltonian diffeomorphism $\phi^1_{H^{(p)}} = (\phi^1_H)^p$. Then the action functional ${\mc A}_{H^{(p)}}$ is invariant under the $\Z/p$-action on $\widetilde{LM}$ whose generator acts by 
\beqn
(x(t), \overline{x}(z)) \mapsto (x(t+\frac1p), \overline{x}(e^{\frac{2\pi i}{p}} \cdot z)).
\eeqn
The equivariant Floer cohomology $HF_{\zp}(H^{(p)})$ is formally the equivariant Morse cohomology of ${\mc A}_{H^{(p)}}$, similar to the viewpoint that the ordinary Floer cohomology is formally the Morse cohomology of the symplectic action functional.

Now we review the definition of the equivariant Floer cohomology provided in \cite{Bai_Xu_2025}. Assume that $H^{(p)}$ is nondegenerate. 
Choose a compatible almost complex structure $J$. Given a pair of capped orbits $\ov{y}_\pm \in \tilde {\mc O}(H^{(p)})$ and a pair of critical points $w_\pm \in {\rm crit}(g) \subset S^\infty$, consider the moduli space of pairs
\beqn
u: {\mb R}\times S^1 \to M,\ w: {\mb R} \to S^\infty
\eeqn
where $u$ is a Floer cylinder for the Hamiltonian $H^{(p)}$ which connects $\ov{y}_\pm$ and $w$ is a gradient line of $g$ which connects $w_\pm$. Let 
\beq\label{eqn13}
{\mc M}_{w_-, w_+}(\ov{y}_-, \ov{y}_+)
\eeq
be the moduli space of such pairs modulo {\it simultaneous translation}. Notice that the disjoint union over all possible asymptotic limits has a free ${\mb Z}_{\geq 0} \times \zp$-action. Similar to the case of quantum Steenrod operation, this is the reason why equivariant transversality is possible. 

The differential of the equivariant Floer complex is defined via virtual counts of moduli spaces \eqref{eqn13}. First, one can compactify the moduli space by allowing breaking and sphere bubbles. The disjoint union of  compactifications still carries a free ${\mb Z}_{\geq 0} \times \zp$-action. Then the FOP perturbation method allows us to define integer, hence mod $p$ counts of the moduli spaces \eqref{eqn13} which are invariant under the action. One considers the cochain group
\beqn
\left\{ \sum_i a_i \ov{y}_i \otimes w_i\ \left| \begin{array}{l} a_i \in \fp,\ \ov{y}_i \in \tilde{\mc O}(H^{(p)}), w_i \in {\rm crit}(g) \\
\displaystyle \inf_i  {\mc A}_{H^{(p)}}(\ov{y}_i) > -\infty, \\
\displaystyle \lim_{i \to \infty} \big( g(w_i) + {\mc A}_{H^{(p)}} (\ov{y}_i) \big) = +\infty 
\end{array} \right. \right\} \cong CF(H^{(p)}) \wh\otimes CM(g)
\eeqn
which is a certain completion of the tensor product $CF(H^{(p)}) \otimes  CM(g)$. It has a $\zp$-equivariant module structure over $\Lambda_{\kzp}^\Gamma$ where the formal variable $t$ acts via the ${\mb Z}_{\geq 0}$-action on $S^\infty$. Then the counts give a $\zp$-equivariant, $\Lambda_{\kzp}^\Gamma$-linear map
\beqn
d_{\zp}: CF(H^{(p)}) \wh\otimes CM(g) \to CF(H^{(p)}) \wh\otimes CM(g).
\eeqn
The $\zp$-invariant part is a subcomplex, denoted by 
\beq\label{equivariant_Floer_complex}
CF_{\zp}(H^{(p)}):= \left( CF(H^{(p)}) \wh\otimes CM(g) \right)^\inv.
\eeq
Its cohomology is the $\zp$-equivariant Floer cohomology of $H^{(p)}$, denoted by $HF_{\zp}(H^{(p)})$. {\it A priori} it is a module over $\Lambda_{\kzp}^\Gamma$. By coupling with Y-shaped gradient trees in $S^\infty$, one can define a canonical $\Lambda_{\rp}^\Gamma$-module structure on the cohomology. 

On the other hand, one can also define equivariant continuation maps for two different sets of choices. In particular, for two different choices of Hamiltonians $H_1$ and $H_2$, one can also build a continuation map from $CF_{\zp}(H_1^{(p)})$ to $CF_{\zp}(H_2^{(p)})$ and prove that it is a chain homotopy equivalence. We summarize the relevant results in \cite{Bai_Xu_2025} below.

\begin{thm}(\cite[Theorem I]{Bai_Xu_2025})\label{thm211}
    Let $(M, \omega)$ be a compact symplectic manifold. 
    
    \begin{enumerate}
    
        \item Let $H$ be a $1$-periodic Hamiltonian such that $H^{(p)}$ is nondegenerate. Then there exists a ${\mb Z}$-graded cochain complex of modules over $\Lambda_{\kzp}^\Gamma$, called the $\Z/p$-equivariant Floer cochain complex,
        \beqn
        (CF_{\zp} (H^{(p)}), d_{\zp})
        \eeqn
        with underlying space 
        \beqn
        CF_{\zp}(H^{(p)}) = \Big( CF(H^{(p)}) \wh{\otimes} CM(g) \Big)^\inv
        \eeqn
        where $d_{\zp}$ depends on an $\omega$-compatible almost complex structure $J$ and a system of choices. Moreover, its cohomology has a well-defined $\Lambda_{\rp}^\Gamma$-module structure which extends the $\Lambda_{\kzp}^\Gamma$-module structure\footnote{The proof in \emph{loc.cit.} establishes the $\Lambda_{\kzp}^\Gamma$-module structure in the Morse setting, but the same argument applies to the Floer setting.}.

        \item For Hamiltonians $H_1, H_2$ with both $H_1^{(p)}$ and $H_2^{(p)}$ nondegenerate, there is a $\Lambda_{\kzp}^\Gamma$-linear chain homotopy equivalence
        \beqn
        (CF_{\zp} (H_1^{(p)}), d_{\zp}) \to (CF_{\zp}(H_2^{(p)}), d_{\zp})
        \eeqn
        (called an equivariant continuation map) whose homotopy class does not depend on any choices. Moreover, the resulting map on cohomology is linear over $\rp$. Therefore, there is a well-defined invariant of $M$, called the {\bf $\zp$-equivariant Floer cohomology}, denoted by 
        \beqn
        HF_\zp(M; \Lambda_{\kzp}^\Gamma).
        \eeqn

\item For any Morse complex $CM(f)$ on $M$ and a Hamiltonian $H$ such that $H^{(p)}$ is nondegenerate, there exist $\Lambda_{\kzp}^\Gamma$-linear cochain maps
\beqn
\pss_{\zp}: CM(f) \otimes \Lambda_{\kzp}^\Gamma \otimes E(\theta) \to CF_\zp(H^{(p)})
\eeqn
and 
\beqn
\ssp_{\zp}: CF_{\zp}(H^{(p)}) \to CM(f) \otimes \Lambda_{\rp}^\Gamma
\eeqn
(the equivariant PSS and SSP maps) whose homotopy classes are well-defined. Moreover, the induced maps on cohomology are isomorphisms as graded $\Lambda_{\rp}^\Gamma$-modules. In particular,
\beqn
HF_{\zp}(M) \cong H(M) \otimes \Lambda_{\rp}^\Gamma.
\eeqn

\item (Notice that the equivariant PSS and SSP maps are not necessarily inverses to each other.) The difference
\beqn
\ssp_\zp \circ \pss_\zp - {\rm Id}: H(M) \otimes \Lambda_{\rp}^\Gamma \to H(M) \otimes \Lambda_{\rp}^\Gamma.
\eeqn
strictly increases either the energy filtration or the $(t,\theta)$-grading.
\end{enumerate}
\end{thm}

The above formulation of the equivariant Floer cochain complex looks different from the usual description, for example, in \cite{SZhao-pants}\cite{covariant-constant}. We give a reformulation of the complex. We focus on the $p>2$ case. Note that the cohomological grading on $CM(g)$ induces an identification of graded module over $\Lambda_{\kzp}^\Gamma$
\beqn
CF_\zp(H^{(p)}) \cong CF(H^{(p)}) \otimes \rp \cong \Big( CF(H^{(p)}) \oplus CF(H^{(p)}) \otimes \theta \Big) \otimes \kzp.
\eeqn
If $w_k^1, \ldots, w_k^p\in S^\infty$ are the index $k$ critical points of $g$, then a concrete correspondence is 
\begin{align*}
&\ \ov{x} \otimes t^l \mapsto \sum_{\alpha=1}^p \iota^\alpha(\ov{x}) \otimes w_{2l}^\alpha,\ &\ \ov{x}\otimes t^l \theta \mapsto \sum_{\alpha=1}^p \iota^\alpha(\ov{x}) \otimes w_{2l+1}^\alpha.
\end{align*}
However, such a correspondence is not unique as only the cyclic order of  $(w_k^1, \ldots, w_k^p)$ is well-defined. On the other hand, fixing such a correspondence, one can translate the equivariant differential $d_{\zp}$ to a familiar form. Namely, $d_{\zp}$ is $t$-linear and is determined by 
\beqn
\begin{split}
d_\zp( \ov{x} \otimes 1) = &\ \sum_{l=0}^\infty d_{\zp}^{2l} (\ov{x}) \otimes t^l + \sum_{l=0}^\infty d_{\zp}^{2l+1}(\ov{x}) \otimes t^l \theta\\
= &\ d_{\zp}^0 (\ov{x}) \otimes 1 + d_{\zp}^1 (\ov{x}) \otimes \theta + {\rm higher\ order\ terms},\\
d_\zp( \ov{x} \otimes \theta) = &\ \sum_{l=0}^\infty d_{\zp}^{'2l}(\ov{x})\otimes t^l \theta + \sum_{l=0}^\infty d_{\zp}^{'2l+1}(\ov{x}) \otimes t^{l+1} \\
= &\ d_{\zp}^{'0} (\ov{x}) \otimes \theta + {\rm higher\ order\ terms}
\end{split}
\eeqn
where ``higher order terms'' are measured in terms of $(t, \theta)$-degrees. Then $d_{\zp}^0$ and $d_{\zp}^{'0}$ are both differentials on the ordinary Floer cochain group $CF(H^{(p)})$. In \cite{Bai_Xu_2025} the following additional features of $d_{\zp}$ are proved.

\begin{thm}\label{thm213}
One can arrange the equivariant differentials $d_{\zp}$ on $CF_\zp(H^{(p)})$ such that under the above identification with $CF(H^{(p)}) \otimes \rp$, $d_{\zp}^0 = d_{\zp}^{'0}  = d_{H^{(p)}}$ such that the complex $(CF(H^{(p)}), d_{H^{(p)}})$ of graded $\Lambda^\Gamma$-modules is filtered chain homotopy equivalent to a Floer complex $CF(H^{(p)})$ of the nondegenerate Hamiltonian $H^{(p)}$ (see Theorem \ref{thm:Floer-cochain}). 

\begin{enumerate}

    \item $d_{\zp}^1$ is a null-homotopic cochain map from $CF(H^{(p)})$ to itself.

    \item $d_{\zp}$ strictly increases the energy filtration.
\end{enumerate}
\end{thm}

\subsubsection{Quantitative theory and Tate spectral invariants}

Now we talk about the quantitative feature of the equivariant theory. Under the framework of \cite{Bai_Xu_2025}, the differential of $d_{\zp}$ respects the action filtration 
\beqn
{\mc A}_{H^{(p)}}: CF_{\zp} (H^{(p)}) \to {\mb R} \cup \{+\infty\} 
\eeqn
by construction. Hence for all $a \in {\mb R}$, one has a subcomplex
\beqn
CF_{\zp}(H^{(p)})^{>a}:= {\mc A}_{H^{(p)}}^{-1}((a, +\infty]) \subseteq CF_{\zp} (H^{(p)});
\eeqn
given some interval $I = (a,b) \subset \R$, one similarly has the quotient complex
\beqn
CF_{\zp}(H^{(p)})^I:= CF_{\zp}(H^{(p)})^{>a}/ CF_{\zp }(H^{(p)})^{>b}.
\eeqn
By the invariance result discussed in the previous section, we can unambiguously denote the resulting cohomology groups as
\begin{align*}
&\ HF_{\zp}(H^{(p)})^{>a}, \ &\  HF_{\zp}(H^{(p)})^I.
\end{align*}

Note that the above definitions require $H^{(p)}$ to be nondegenerate. However, following Section \ref{subsec:filtered-floer}, we can consider the partially ordered set $P_{a,b}^{(p)}(H)$ which consists of Hamiltonians $K$ such that
\begin{itemize}
    \item $K^{(p)}$ is nondegenerate,
    \item $K < H$, and
    \item $a,b \notin \mathrm{Spec}(K^{(p)})$,
\end{itemize}
with partial order given by $K \leq K'$ pointwise. Then using the continuation maps in Theorem \ref{thm211} (2), we can define
\beqn
HF_{\zp}(H^{(p)})^I := \varinjlim_{K \in P^{(p)}_{a, b}(H)} HF_{\zp}(K^{(p)})^I.
\eeqn
Letting $b = \infty$, we can also make sense of $HF_{\zp}(H^{(p)})^{>a}$ via a colimit.

One can use the Tate version of the equivariant Floer theory to discuss spectral invariants as the inversion of the formal variable $t$ produces a Novikov field $\Lambda_{\kp}^\Gamma$. The underlying cochain complex replaces \eqref{equivariant_Floer_complex} by 
\begin{equation}\label{eqn:tate-complex}
    \widehat{CF}_{\zp}(H^{(p)}; \Lambda_{\kp}^\Gamma):= CF_{\zp}(H^{(p)}) \underset{\Lambda_{\kzp}^\Gamma}{\otimes} \Lambda_\kp^\Gamma
\end{equation}
and the differential is the $\Lambda_\kp^\Gamma$-linear extension of the equivariant Floer differential. We denote the resulting cohomology group by 
\beqn
\widehat{HF}_{\zp}(H^{(p)}; \Lambda_{{\mc K}_p}^\Gamma).
\eeqn
The action filtration ${\mc A}_{H^{(p)}}$ extends to a non-archimedean norm
\beqn
\wh{\mc A}_{H^{(p)}}: \wh{CF}_{\zp}(H^{(p)}) \to {\mb R} \cup \{+\infty\}.
\eeqn
Then using Theorem \ref{thm211}, one can consider the {\bf Tate spectral invariants} following Definition \ref{defn:spec-inv}
\beqn
\wh c(\cdot, H^{(p)}): \wh{HF}_{\zp}(M; \Lambda_{\kp}^\Gamma) \to {\mb R} \cup \{+\infty\}, 
\eeqn
which extend to all Hamiltonians using continuity of the spectral invariants.

On the other hand, one also has the {\bf Tate bar-length spectrum}, which can be seen by turning to the universal Novikov ring $\Lambda_{0, \kp}^{\rm univ}$. More precisely, consider
\beq\label{eqn16}
\wh{CF}_{\zp}(H^{(p)}; \Lambda_{0, \kp}^{\rm univ}):= \left( {\rm span}_{\Lambda_0^{\rm univ}} \Big\{ q^{-\cl A_{H^{(p)}}(\ov{x}_1)} \ov{x}_1, \dots, q^{-\cl A_{H^{(p)}}(\ov{x}_N)} \ov{x}_N \Big \} \otimes CM(g) \right)^\inv 
\eeq
where ${\mc O}(H^{(p)}) = \{ x_1, \ldots, x_N\}$ and $\ov{x}_1, \ldots, \ov{x}_N \in \wt{\mc O}(H^{(p)})$ are choices of cappings. This is a canonically defined $\Lambda_{0, \kp}^{\rm univ}$-submodule of $\wh{CF}_{\zp}(H^{(p)}; \Lambda_{\kp}^{\rm univ})$ (cf. Section \ref{subsubsec:bar}). Then the exponents of the torsion part of the Tate cohomology give rise to the \emph{Tate bar-length spectrum}. The bar-lengths only depend on the time-1 map $\phi^p$. An element of the bar-length spectrum is denoted by
\beqn
\widehat{\beta}_j(\phi^p) = \widehat{\beta}_j(H^{(p)}).
\eeqn

\begin{remark}
In the weakly monotone setting, if we use geometric perturbations to construct the equivariant Floer cochain complex, due to the infinite dimensionality of $S^\infty$, in order to achieve transversality, the resulting equivariant Floer differential may no longer respect the action filtration (induced from $H^{(p)}$) because of the ``largeness" of the perturbations. In the monotone or symplectically aspherical setting, this issue can be resolved by either using perturbations of $J$ or using the polynomial dependence of $d_{\zp}$ on the $t$ variable as a result of the index-energy relation. In the strictly weakly monotone setting, Sugimoto \cite{Sugimoto} introduced the notions of $X_K$- and $X_\infty$-modules to transfer the equivariant differentials defined over finite approximations $S^{2K+1}$ of $S^{\infty}$ that respect the action filtrations successively via algebraic tools. 
\end{remark}

\subsubsection{Local equivariant Floer cohomology}

The local version of the equivariant Floer cohomology was firstly defined in \cite{SZhao-pants} using classical methods; this is possible as the local consideration does not involve sphere bubbling. However, we will define the local version of equivariant Floer cohomology using abstract perturbations which fits into the construction of \cite{Bai_Xu_2025} as well as our homological perturbation given later in Proposition \ref{prop33b}. Let $x\in {\mc O}(H)$ be an isolated 1-periodic orbit such that $x^{(p)}\in {\mc O}(H^{(p)})$ is also isolated. Choose a small perturbation $H_1$ of $H$ such that $H_1^{(p)}$ is also nondegenerate. Then by Theorem \ref{thm211}, one can define the $\zp$-equivariant differential
\beqn
d_{\zp}: CF(H_1^{(p)}) \wh{\otimes}CM(g) \to CF(H_1^{(p)}) \wh{\otimes} CM(g)
\eeqn
whose $\zp$-invariant part is the complex $CF_{\zp}(H_1^{(p)})$. Moreover, one can identify those ``small'' differentials among orbits which are close to $x^{(p)}$, giving rise to 
\beqn
d_{\zp}^\loc: CF^\loc (H_1^{(p)}, x^{(p)}) \wh{\otimes} CM(g) \to CF^\loc (H_1^{(p)}, x^{(p)}) \wh{\otimes} CM(g).
\eeqn
Its $\zp$-invariant part is then a complex
\beqn
CF_{\zp}^\loc(H_1^{(p)}, x^{(p)}).
\eeqn
One can use the continuation map argument to prove (which we omit) that the chain homotopy type is independent of perturbations and the choice of the small perturbation $H_1$. Hence the cohomology, denoted by 
\beqn
HF_{\zp}^\loc(H^{(p)}, x^{(p)})
\eeqn
is a well-defined graded $\Lambda_{\kzp}^\Gamma$-module. Moreover, there is a well-defined $\rp$-module structure whose structural coefficients are counts of moduli spaces coupling Y-shaped gradient trees in $S^\infty$.

We compare our definition with the definition of \cite{SZhao-pants} using perturbations (which also needs choosing a nondegenerate perturbation $H_1$).

\begin{lemma}\label{lemma_local_eq_comparison}
The above definition of $HF_\zp^\loc(H^{(p)}, \ov{x}^{(p)})$ agrees with that of \cite{SZhao-pants}. More precisely, if $CF_\zp^\loc(H_1^{(p)}, \ov{x}^{(p)})_{\rm SZ}$ denotes the complex defined in \cite{SZhao-pants}, then there exists a $\kzp$-linear chain homotopy equivalence 
\beqn
CF_\zp^\loc (H_1^{(p)}, \ov{x}{}^{(p)})_{\rm SZ} \to CF_\zp^\loc(H_1^{(p)}, \ov{x}{}^{(p)})
\eeqn
which induces an isomorphism of $\rp$-modules\footnote{The $\rp$-module structure was claimed in \cite{SZhao-pants} without explicit construction.}
\beqn
HF_\zp^\loc(H^{(p)}, \ov{x}{}^{(p)})_{\rm SZ} \cong HF_\zp^\loc (H^{(p)}, \ov{x}{}^{(p)}).
\eeqn
\end{lemma}

\begin{proof}[Sketch of proof]
The idea is the same as the comparison in the nonequivariant case as discussed in \cite[Theorem G]{Bai_Xu_2025}. At the level of moduli spaces, the construction in \cite{SZhao-pants} differs from \cite{Bai_Xu_2025} in that the former achieves transversality by choosing a certain family of time-dependent almost complex structures parametrized by $S^\infty$ while the latter constructs the moduli spaces using a fixed almost complex structure, followed by abstract perturbations on Kuranishi charts of these moduli spaces. Nevertheless, the global Kuranishi chart construction in \cite[Section 25]{Bai_Xu_2025} applies equally well to the former situation. Then, following the argument in \cite[Theorem G]{Bai_Xu_2025}, one sees that the induced Kuranishi sections on the derived orbifold charts are already FOP transverse. By passing to the ${\mb Z}/p$-invariant part of the resulting cochain complex, one can directly identify the differential with the formulas in \cite[Section 6.1]{SZhao-pants}. Therefore, we have shown that the equivariant cochain complex produced from the procedure in \cite{Bai_Xu_2025} for this varying family of almost complex structures strictly agrees with $CF_\zp^\loc (H_1^{(p)}, \ov{x}{}^{(p)})_{\rm SZ}$. Finally, the desired chain map is obtained by a continuation map argument, which is a chain homotopy equivalence because we can also use continuation maps to construct a chain homotopy inverse.
\end{proof}


In general the local equivariant Floer cohomology need not be a free $\rp$-module. However, under certain local conditions it has a simple structure. Recall that a fixed point $x$ of $\phi$ is called {\bf $p$-admissible} if for each eigenvalue $\lambda \neq 1$ of $d\phi_x$, $\lambda^p \neq 1$. In other words, $d\phi_x$ and $d\phi_x^p$ have the same generalized eigenspace for the eigenvalue 1. 

\begin{prop}\label{prop_local_eq_structure}
When $x$ is $p$-admissible and an isolated fixed point of $\phi^p$, there is an isomorphism of $\rp$-modules 
\beq\label{eqn17}
HF_\zp^\loc(H^{(p)}, x^{(p)}) \cong HF^\loc(H^{(p)}, x^{(p)}) \otimes \rp.
\eeq
\end{prop}

\begin{proof}
The proof contains several steps. First, one uses a certain invariance property of the local (equivariant) Floer cohomology to reduce the consideration to certain normal forms of the germ of the Hamiltonian. Second, the normal form can be further deformed to a situation where the $\zp$-actions on the moduli spaces are all trivial.

We first consider the invariance property of the local (equivariant) Floer cohomology we will need. Let $\phi_s$, $s \in [0, 1]$ be a continuous family of Hamiltonian diffeomorphisms and $x \in {\rm Fix}(\phi_s)$ for all $s$. We say that $x$ is {\bf uniformly isolated} for the family $\phi_s$ if there exists a fixed open neighborhood $U\subset M$ of $x$ such that for all $s$, ${\rm Fix}(\phi_s) \cap U = \{x\}$. 

Now we follow a local argument of Ginzburg--Gurel \cite{GG-local-gap}. Let $x$ be an isolated and $p$-admissible fixed point of $\phi$. Let $X_0 \subset T_x M$ be the generalized eigenspace of $d\phi_x$ for eigenvalue $1$ and $X_1 \subset T_x M$ be the direct sum of other generalized eigenspaces. Both $X_0$ and $X_1$ are symplectic subspaces and one can identify a neighborhood of $x$ with an open ball of $X_0 \oplus X_1$. Then by the argument of \cite[Section 4.5]{GG-local-gap}, one can construct a family of Hamiltonian diffeomorphisms $\phi_s$ defined in such a neighborhood satisfying
\begin{enumerate}
    \item $\phi_0 = \phi$.

    \item $x$ is uniformly isolated for both the family $\phi_s$ and the family $\phi_s^p$.

    \item $d\phi_s$ remains constant at $x$.
    
    \item $\phi_1$ is the product of Hamiltonian diffeomorphisms $\psi_0$ on $X_0$ and $\psi_1$ on $X_1$ such that $0 \in X_0$ is totally degenerate for $\psi_0$ and $0 \in X_1$ is nondegenerate for $\psi_1$. Moreover, $\psi_0$ is autonomous, generated by a smooth function $f_0: X_0 \to {\mb R}$.
\end{enumerate}
Then by the persistence of local (equivariant) Floer cohomology, one can reduce the proof to the case of $\phi_1$ which has the split form. This is the situation where equivariant transversality can be achieved by using a constant almost complex structure. Hence one can see that 
\beqn
CF_\zp^\loc(\phi_1, x) \cong CF_\zp^\loc(\psi_0 \times \psi_1, (0,0)) \cong CF^\loc(\psi_0 \times \psi_1, (0, 0)) \otimes \rp.
\eeqn
One also does not need to perturb the Floer equation to achieve transversality for the moduli spaces defining the $\rp$-module structure on the cohomology. Hence one has the isomorphism \eqref{eqn17} as $\rp$-modules.
\end{proof}

\subsubsection{Local equivariant Floer cohomology using Hamiltonian perturbations}\label{subsubsec:local-equiv}

When $(M, \omega)$ is weakly monotone, the equivariant Floer cochain complex can be constructed using geometric perturbations. Though we will not need such a construction in its full generality, we will use it for the discussion of local equivariant Floer cohomology. 

Fix a Hamiltonian $H$ such that $H^{(p)}$ is nondegenerate. Now we consider a perturbation 
\beqn
{\ms H}^{eq}: S^\infty \times S^1 \times M \to {\mb R}
\eeqn
satisfying the following conditions.
\begin{enumerate}
    \item For each $v \in S^\infty$, ${\ms H}^{eq}(v, \cdot, \cdot)$ is sufficiently close to $H^{(p)}$ and agrees with $H^{(p)}$ near all 1-periodic orbits of $H^{(p)}$ such that for any $v \in S^{\infty}$, the Hamiltonian ${\ms H}^{eq}(v, \cdot, \cdot)$ has the same set of capped 1-periodic orbits as that of $H^{(p)}$.

    \item Let $\sigma \in \zp$ be the generator. Then 
    \beqn
    {\ms H}^{eq}\Big( \sigma v, t + \frac{1}{p}, x \Big) = {\ms H}^{eq} (v, t, x ).
    \eeqn

    \item Let $\tau: S^\infty \to S^\infty$ be the shifting map. Then 
    \beqn
    {\ms H}^{eq}(\tau v, t, x) = {\ms H}^{eq}(v, t, x).
    \eeqn
\end{enumerate}

Then for each pair $\ov{y}_\pm \in \tilde {\mc O}(H^{(p)})$ and $w_\pm \in {\rm crit}(g) \subset S^\infty$, consider the equation on pairs $(u, w)$ solving
\begin{align*}
&\ \partial_s u + J ( \partial_t u - X_{{\ms H}^{eq}(w(s), t)} (u)) = 0,\ &\ w'(s) + \nabla g (w(s)) = 0.
\end{align*}
The equation is obviously translation invariant. Moreover, finite energy solutions converge to 1-periodic orbits of $H^{(p)}$ as well as critical points of $g$. Then one can again define the moduli space
\beqn
{\mc M}_{w_-, w_+}(\ov{y}_-, \ov{y}_+; {\ms H}^{eq})
\eeqn
of solutions with the prescribed asymptotic limits modulo translation. 

Transversality can be achieved by taking generic ${\ms H}^{eq}$ while maintaining the required symmetry. The only problem is the issue of bubbling off spheres, which obstructs the construction beyond the weakly monotone case. If we assume $(M, \omega)$ is weakly monotone, then moduli spaces of expected dimension zero consist of finitely many points and moduli spaces of expected dimension one can be compactified by adding once-broken trajectories. Hence one can define mod $p$ count of these moduli spaces, which eventually lead to the construction of the equivariant Floer cochain complex $CF_{\zp}(H^{(p)})$. 

Following Lemma \ref{lemma_local_eq_comparison}, one can similarly show that in the weakly monotone setting, the above construction leads to equivariant Floer cochain complexes which are chain homotopy equivalent to the general version we use from \cite{Bai_Xu_2025}.


The perturbative approach is only used in this paper when discussing the local equivariant Floer theory. Let $(M, \omega)$ be a closed symplectic manifold and $H$ be a 1-periodic Hamiltonian with an isolated 1-periodic orbit $x$ such that its $p$-th iteration $x^{(p)}$ is isolated in ${\mc O}(H^{(p)})$. Choose a sufficiently small perturbation $H_1$ of $H$ such that $H_1^{(p)}$ is nondegenerate. Then $CF_{\zp}(H_1^{(p)})$ can be defined as described above. Moreover, one can identify a ``local'' differential which restricts to a differential of the subspace generated by orbits near $x^{(p)}$, giving the definition of the local equivariant Floer cochain complex, denoted by 
\beqn
CF_{\zp}^\loc(H^{(p)}, x^{(p)})
\eeqn

\subsection{The equivariant pair-of-pants}
\label{subsec:equivariant-pair-of-pants}

Next, we recall the construction of the 
equivariant pair-of-pants product following \cite{seidel, SZhao-pants}, and the generalization in \cite{Bai_Xu_2025}. As in the previous discussion of the quantum Steenrod power operation, we first describe the relevant moduli spaces without discussing details of the perturbation. 

Recall we used the domain $\Sigma_p$, which branches over $S^2$, as the domain of the quantum Steenrod operation. Let $\Sigma_p^\circ$ be the complement of the $p+1$ points at $p$-th roots of unity and $\infty$, which branches over ${\mb C}^* \cong {\mb R}\times S^1$. Then $\Sigma_p^\circ$ still carries a $\zp$-action. Moreover, the branched cover $\Sigma_p^\circ \to {\mb R}\times S^1$ given by $z \mapsto z^p$ pulls back $p$ negative cylindrical ends and one positive cylindrical end. Let $H$ be a Hamiltonian on $(M, \omega)$ such that both $H$ and $H^{(p)}$ are nondegenerate. Then the flat Hamiltonian connection $H dt$ on ${\mb R}\times S^1$ is pulled back via the covering $\Sigma_p^\circ \to {\mb R}\times S^1$ to a Hamiltonian connection $\sigma_H$ on $\Sigma_p^\circ$ which is $H dt$ at each negative end and $H^{(p)}dt$ at the positive end. Choose a compatible almost complex structure $J$. For each tuple $(x_1, \ldots, x_p)$ of 1-periodic orbits of $H$ and $y_\infty \in {\mc O}(H^{(p)})$, one can consider the Floer equation
\beqn
(\nabla^{\sigma_H} u)^{0,1} = 0,\ u: \Sigma_p^\circ \to M
\eeqn
such that $u$ is asymptotic to $x_j$ at the $j$-th negative end and asymptotic to $y_\infty$ at the positive end. Here $\nabla^{\sigma_H}$ is the covariant derivative associated to the Hamiltonian connection $\sigma_H$. By fixing cappings $\ov{x}_1, \ldots, \ov{x}_p$ and $\ov{y}_\infty$ one can further restrict the homotopy type of solutions $u$. 

The equivariant pair-of-pants product is, roughly speaking, defined by $\zp$-equivariant virtual counts of the above moduli spaces. To achieve equivariant transversality, one couples the above equation with parametrized gradient lines in $S^\infty$. For each tuple $(\ov{x}_1, \ldots, \ov{x}_p; \ov{y}_\infty)$ and a pair of critical points $w_\pm \in {\rm crit}(g) \subset S^\infty$, let 
\beqn
{\mc M}_{w_-, w_+}(\ov{x}_1, \ldots, \ov{x}_p; \ov{y}_\infty)
\eeqn
be the moduli space of pairs $(u, w)$ where $u: \Sigma \to M$ is a solution to the above equation with asymptotic data $\ov{x}_1, \ldots, \ov{x}_p; \ov{y}_\infty$ at cylindrical ends and $w$ is a gradient flow line in $S^\infty$ connecting $w_-$ and $w_+$. One compactifies the moduli space by adding broken and bubbled configurations. Here we consider breakings at negative and positive ends in different ways. At negative ends, the breakings at the $p$ cylindrical ends of $\Sigma$ and the negative end of the domain ${\rm dom}(w) = {\mb R}$ are treated independently. On the other hand, the breaking at the positive end of $\Sigma$ and the positive end of the domain of $w$ is regarded as {\it simultaneous}. In other words, breakings at negative ends need to take quotient by the translation of ${\mb R}^{p+1}$ while the breakings at positive ends are quotiented by a single ${\mb R}$. 

By using the FOP perturbation method, one can define integral, hence $\fp$-valued counts of the compactified moduli spaces. As the ${\mb Z}_{\geq 0} \times \zp$-action is free, one can choose such perturbations equivariantly. These counts can be assembled to chain maps
\beqn
CF(H; \Lambda^\Gamma)^{\otimes p} \wh{\otimes} CM(g) \to CF(H^{(p)}; \Lambda^\Gamma) \otimes CM(g).
\eeqn
We summarize the algebraic consequences in the following theorem.

\begin{thm}\cite[Theorem J]{Bai_Xu_2025}\label{thm_FSt}
Let $H$ be a nondegenerate Hamiltonian on $(M, \omega)$ such that $H^{(p)}$ is also nondegenerate. Let $CF(H; \Lambda^\Gamma)$ be the Hamiltonian Floer cochain complex and $CF_{\zp}(H^{(p)})$ be the equivariant Floer cochain complex. Then there exists a well-defined homotopy class of $\Lambda_{\kzp}^\Gamma$-linear cochain maps
    \beqn
    \xymatrix{
 CF(H )^{\otimes p}_{\zp} = \Big( CF(H )^{\otimes p} \wh{\otimes} CM(g) \Big)^\inv \ar[r]^-{\wt{\fst}_p} & \Big( CF(H^{(p)} ) \otimes CM(g) \Big)^\inv} =  CF_{\zp}(H^{(p)}) ,
    \eeqn
    called the {\bf equivariant pants product}.
    Its composition with the quasi-Frobenius map
    \beqn
    qF: H(CF(H)) \to H(CF(H )^{\otimes p}_{\zp})
    \eeqn
    defines a canonical $\fp$-linear map
    \beqn
    \fst_p:= \wt{\fst}_p \circ qF: HF(M; \Lambda^\Gamma) \to HF_{\zp}(M; \Lambda^\Gamma_{\kzp})
    \eeqn
    called the {\bf Floer-theoretic Steenrod operation}. 
\end{thm}

The first important property is the equivalence between the Floer-theoretic quantum Steenrod operation and the quantum Steenrod operation. In \cite{Bai_Xu_2025} the following theorem is proved. Remember that the ordinary and the equivariant Floer cohomology are both invariants of the symplectic manifold.

\begin{thm}\cite[Theorem J]{Bai_Xu_2025}\label{thm214}
Let $(M,\omega)$ be a closed symplectic manifold. Then for any prime $p$, 
the following diagram commutes.
    \begin{equation}\label{equivpssisom}
\vcenter{ \xymatrix{
H^*(M;\Lambda^\Gamma ) \ar[r]^-{\qst_p} \ar[d]_{\pss} &
H^*(M;\Lambda_{\rp}^\Gamma) 
\\ 
HF(M; \Lambda^\Gamma)
\ar[r]_-{\fst_p}
&
HF_{\zp}(M; \Lambda_{\kzp}^\Gamma) \ar[u]_{\ssp_{\zp}}
}}
\end{equation}
\end{thm}

Another important property of the equivariant pants product is that after passing to the Tate cohomology, it induces an isomorphism (cf. \cite{seidel,SZhao-pants,Sugimoto,Bai_Xu_2025}) which also preserves the energy filtration.

\begin{thm}\cite[Theorem M]{Bai_Xu_2025}\label{thm:pants-localization}
The equivariant pair-of-pants product induces an isomorphism on Tate cohomology which respects the energy filtration. As a consequence (which is not explicitly stated in \cite{Bai_Xu_2025}), after inverting $t \in \kzp$ and base change to the universal Novikov ring $\Lambda_{0,\kp}^{\rm univ}$, $\wt \fst_p$ (see \eqref{canonical_submodule} and \eqref{eqn16}) induces a $\Lambda_{0, \kp}^{\rm univ}$-linear chain map
    \beqn
    \wh{\fst}_p: \wh{CF(H; \Lambda_{0 }^{\rm univ})^{\otimes p}_{\zp}} \to 
    \wh{ CF}_{\zp}(H^{(p)}; \Lambda_{0, \kp}^{\rm univ});
    \eeqn
which is a chain homotopy equivalence.

\end{thm}

Recall that both the $p$-fold tensor product and the equivariant Floer complex of $H^{(p)}$ have their Tate bar-length spectra. It follows from the above theorem that 
\beqn
\wh\beta_j^{\otimes p}(H) = \wh\beta_j(H^{(p)}),
\eeqn
where $\wh\beta_j^{\otimes p}(H)$ constitutes the bar-length spectrum of $\wh{CF(H; \Lambda_{0 }^{\rm univ})^{\otimes p}_{\zp}}$. Combining with Lemma \ref{lemma:quasi-Fr}, the bar-length spectrum of $H$ and the Tate bar-length spectrum of $H^{(p)}$ satisfy the relation
\begin{equation}\label{eqn:bar-length-localization}
    \widehat{\beta}_{2j-1}(H^{(p)}) = \widehat{\beta}_{2j}(H^{(p)}) = p\beta_j(H), \quad \quad \quad 1 \leq j \leq N,
\end{equation}
and there are exactly $2N$ torsion components of the Tate cohomology $\widehat{HF}_{\zp} (H^{(p)};\Lambda_{0, \kp}^{\rm univ})$. As a result, we see that
\beq\label{eqn218}
\widehat{\beta}_{\mathrm{tot}}(H^{(p)}) = 2p \beta_{\mathrm{tot}}(H).
\eeq

Moreover, one has a quantitative version of the Smith inequality, which was proved as \cite[Theorem D]{S-HZ} in the monotone case.

\begin{thm}\label{thm219}
For any prime $p$ and any Hamiltonian $H$ on $(M, \omega)$ such that both $H$ and $H^{(p)}$ are nondegenerate, there holds
\begin{equation}\label{eqn219}
    \beta_{\mathrm{tot}}(H^{(p)}) \geq p\cdot\beta_{\mathrm{tot}}(H).
\end{equation}
\end{thm}

\begin{proof}
The comparison of total bar-lengths is via the Tate bar-length spectrum of $H^{(p)}$. By Theorem \ref{thm213}, one can write the differential $\wh d_{\zp}$ on the Tate complex 
\beqn
\wh CF_\zp(H^{(p)}) \cong \Big( CF(H^{(p)}) \oplus CF(H^{(p)}) \otimes \theta \Big) \otimes \kp
\eeqn
in the form 
\beqn
\wh d_\zp  =  \left[ \begin{array}{cc} d_{H^{(p)}} & S_p \\ 0 & d_{H^{(p)}} \end{array} \right] + t K
\eeqn
where $K$ does not decrease the $(t, \theta)$-degree. The leading $2\times 2$ matrix is the cone of $S_p$, whose differential is denoted temporarily by $d_{\rm Cone}$. Then one can use the argument in \cite{S-HZ}. More precisely, by \cite[Proposition 17]{S-HZ}, one has 
\beqn
\wh \beta_{\rm tot}(H^{(p)}) \leq \beta_{\rm tot}(d_{\rm Cone})
\eeqn
where the latter is the total bar-length of $d_{\rm Cone}$. Moreover, as $S_p$ is null-homotopic (see Theorem \ref{thm213}) and strictly increases the energy filtration, by \cite[Lemma 18]{S-HZ}, one has 
\beqn
2 \beta_{\rm tot}(H^{(p)}) \geq \beta_{\rm tot}(d_{\rm Cone}).
\eeqn
Together with \eqref{eqn218}, one obtains \eqref{eqn219}.
\end{proof}

\subsubsection{The local version}\label{subsubsec:local-floer-steenrod}

Let $x \in {\mc O}(H)$ be an isolated orbit such that $x^{(p)} \in {\mc O}(H^{(p)})$ is also isolated. Previously one has defined the local (equivariant) Floer cohomologies $HF^\loc(H, x)$ and $HF_\zp^\loc(H^{(p)}, x^{(p)})$. Now we sketch a definition of a local equivariant pants product which is equivalent to the original one given by \cite{SZhao-pants}.

Choose a small and nondegenerate perturbation $H_1$ of $H$ such that $H_1^{(p)}$ is also nondegenerate. Then, up to making choices one has the complexes $CF(H_1)$ and $CF_{\zp}(H_1^{(p)})$. The general construction of $\fst_p$ given in \cite{Bai_Xu_2025} provides a $\zp$-equivariant cochain map
\beqn
\wt{\wt{\fst}}_p: CF(H_1)^{\otimes p} \wh{\otimes} CM(g) \to CF(H_1^{(p)}) \wh{\otimes} CM(g)
\eeqn
which, after restricting to the $\zp$-invariant part, reducing to cohomology, and composing with the quasi-Frobenius map, is the map $\fst_p$. Here, one can identify a ``local part''
\beqn
\wt{\wt{\fst}}{}_{p, x}^\loc: CF^\loc (H_1, x)^{\otimes p} \wh{\otimes} CM(g) \to CF^\loc(H_1^{(p)}, x^{(p)}) \wh{\otimes} CM(g)
\eeqn
by energy consideration, whose restriction to the $\zp$-invariant part gives
\beqn
\wt{\fst}{}_{p, x}^\loc: CF^{\loc}(H_1, x)^{\otimes p}_{\zp} \to CF^\loc_{\zp}(H_1^{(p)}, x^{(p)}).
\eeqn
One can show that this is a cochain map and compatible with the local continuation maps. Hence one obtains a well-defined map on cohomology
\beqn
\wt{\fst}{}_{p, x}^\loc: H(CF^\loc(H_1, x)^{\otimes p}_{\zp}) \to HF^\loc_\zp(H^{(p)}, x^{(p)}).
\eeqn
By composing the quasi-Frobenius map $qF$ we obtain a well-defined map
\beqn
\fst_{p, x}^\loc: HF^\loc(H, x) \to HF^\loc_\zp(H^{(p)}, x^{(p)}).
\eeqn

Following Lemma \ref{lemma_local_eq_comparison}, the following statement is deduced from the same reasoning.

\begin{lemma}\label{lemma222}
The above definition of $\fst_{p, x}^\loc$ coincides with the definition of \cite{SZhao-pants}. More precisely, the following diagram commutes.
\beqn
\xymatrix{  &  &      HF_\zp^\loc(H^{(p)}, x^{(p)}) \ar[dd] \\
HF^\loc(H, x) \ar[urr]  \ar[drr]  &  &  \\
&  &   HF_\zp^\loc(H^{(p)}, x^{(p)})_{\rm SZ} }
\eeqn
\qed
\end{lemma}

Now we recall two useful facts proved in \cite{SZhao-pants}. The first is the invertibility of the local equivariant pants product. The second is about the behavior of the operator in the nondegenerate case. Both rely on a coproduct construction of \cite{SZhao-pants} which has not been fitted into the general framework of \cite{Bai_Xu_2025}. 

\begin{thm}(cf. \cite[Section 10]{SZhao-pants})\label{thm_local_pants}
\begin{enumerate}

\item $\fst_{p, x}^\loc$ is an isomorphism.

\item When $x$ and $x^{(p)}$ are nondegenerate, $HF^\loc(H, x) \cong CF^\loc(H, x)$ is generated over $\Lambda^\Gamma$ by any capping $\ov{x}\in \wt{\mc O}(H)$ and $HF_\zp^\loc(H^{(p)}, x^{(p)})$ is generated over $\Lambda_\rp^\Gamma $ by the corresponding capped orbit $\ov{x}^{(p)}$. Then there exists $\alpha$, $0 \leq \alpha \leq n(p-1)$ and $0 \neq c_x \in \fp$ such that 
\beqn
\fst_{p, x}^\loc( \ov{x} ) = c_x t^\alpha \ov{x}^{(p)}.
\eeqn
\end{enumerate}
\end{thm}

\begin{proof}
By Lemma \ref{lemma222}, one can use results from \cite{SZhao-pants}. There, a local equivariant pair-of-pants coproduct is constructed. When $x^{(p)}$ is nondegenerate, it has the form
\beqn
C_{p, x}^\loc: HF_\zp^\loc(H^{(p)}, x^{(p)}) \to H(CF^\loc(H, x)_\zp^{\otimes p}).
\eeqn
It was also proved that
\beqn
C_{p, x}^\loc \circ \wt{\fst}{}_{p,x}^\loc = (-1)^n t^{n(p-1)} {\rm Id}: H(CF^\loc(H, x)_\zp^{\otimes p}) \to H(CF^\loc(H, x)_\zp^{\otimes p}).
\eeqn
Note that $H(CF^\loc(H, x)_\zp^{\otimes p})$ is also 1-dimensional, generated by $\ov{x}^{\otimes p}$. Hence 
\beqn
\wt{\fst}{}_{p, x}^\loc (\ov{x}^{\otimes p}) = a_x \ov{x}^{(p)}
\eeqn
where $a_x \in \rp\setminus \{0\}$ is a homogeneous divisor of $t^{n(p-1)}$. Hence $a_x$ can only be a nonzero multiple of a power of $t$ whose degree is at most $n(p-1)$. Lastly, the quasi-Frobenius map is simply $\ov{x} \mapsto \ov{x}^{\otimes p}$.
\end{proof}


\subsubsection{Local version using Hamiltonian perturbations}\label{subsubsec:local-pants-perturb}

For the purpose of this paper, we also need to review the construction using classical methods for the weakly monotone case as well as the local version of the equivariant pants product. This will be used in Section \ref{subsec: local-kunneth} to compute the equivariant pants product in local Floer cohomology. We note that one can also use the approach in \cite{SZhao-pants} to perturb $J$ instead, and the upshot on cohomology is evidently equivalent.

Consider the $p$-fold covering $z \mapsto z^p$ from $S^2$ to $S^2$ branched over $0$ and $\infty$. Let $\tilde S$ be the domain of the branched covering with $\infty$ as well as the preimages of $1$ removed. Then $\zp$ acts freely on $\tilde S$ which cyclically  permutes the $p$-th roots of unity $z_1, \ldots, z_p$. Choose cylindrical ends at the $p$-th roots of unity as well as $\infty$ in a $\zp$-equivariant fashion. Choose also a $\zp$-equivariant Hamiltonian connection $\sigma$ on $\tilde S$ which is asymptotic to $H dt$ on each of the $p$ negative ends and asymptotic to $H^{(p)} dt$ on the positive end. Then the equivariant pair-of-pants product is defined, roughly speaking, by counting solutions to the Floer equation on $\tilde S$ with such a Hamiltonian perturbation, with suitable further perturbations to achieve transversality. More precisely, consider pairs $(u, \zeta)$ where
\begin{align*}
&\ u: \tilde S \to M,\ (\nabla^\sigma u)^{0,1} = 0,\ &\ \zeta: {\mb R} \to S^\infty,\ \zeta'(s) = \nabla g(\zeta(s)).
\end{align*}
The chain-level equivariant pair-of-pants product should be defined via appropriately counting solutions to the above equation.

To achieve equivariant transversality, we choose to perturb the Hamiltonian connection which is allowed to vary with a variable $v \in S^\infty$. First, choose a family of Hamiltonians
\beqn
{\ms H}^{eq}_v: S^1\times M \to {\mb R},\ v \in S^\infty
\eeqn
which can be used to define the equivariant Floer complex $CF_{\zp}(H^{(p)})$. Then choose a family of Hamiltonian connections
\beqn
\sigma_v^{eq},\ v \in S^\infty
\eeqn
on $\tilde S$ satisfying the following conditions.
\begin{enumerate}
\item Its restriction to the positive end coincides with ${\ms H}^{eq}_v dt$.

\item Its restriction to each negative end coincides with $H dt$.

\item It is invariant under the $\zp \times {\mb Z}_{\geq 0}$-action.
\end{enumerate}
Then for any given tuple of capped orbits $\ov{x}_1, \ldots, \ov{x}_p \in \tilde {\mc O}(H)$, $\ov{y} \in \tilde {\mc O}(H^{(p)})$ and critical points $w_\pm \in {\rm crit} g$, one considers the moduli space
\beqn
{\mc M}_{w_-, w_+} (\ov{x}_1, \ldots, \ov{x}_p; \ov{y})
\eeqn
of solutions $(u, \zeta)$ to the equation 
\begin{align*}
&\ (\nabla^{\sigma_{\zeta(0)}^{eq}} u)^{0,1} = 0,\ &\ \zeta'(s)  = \nabla g(\zeta(s))
\end{align*}
such that $u$ is asymptotic to $\ov{x}_k$ along the $k$-th negative end and asymptotic to $\ov{y}$ along the positive end, and such that $\zeta(s)$ converges to $w_\pm$ at $\pm \infty$. Notice that the disjoint union of moduli spaces over all combinations of asymptotic data has a natural free $\zp$-action. Therefore, by choosing a generic $\sigma_v^{eq}$, one can achieve $\zp$-equivariant transversality. Then under the weakly monotone assumption, when the expected dimension of the moduli space is zero, one obtains a mod $p$ count. When the expected dimension of the moduli space is one, it can be compactified by adding configurations with at most one breaking; notice that the breaking at the positive end should be regarded as simultaneously breaking both a Floer trajectory and a gradient line in $S^\infty$. Therefore, one obtains a linear map
\beqn
\wt{\wt{FSt}}_p: CF(H)^{\otimes p} \wh \otimes CM(g) \to CF(H^{(p)}) \wh\otimes CM(g)
\eeqn
which is $\zp$-equivariant. Here $\wh\otimes$ is a completion of the ordinary tensor product so both sides are modules over $\Lambda_{\kzp}^\Gamma$. The restriction to the $\zp$-invariant part is then a chain map
\beqn
\wt{FSt}_p: CF(H)^{\otimes p}_{\zp} \to CF_{\zp}(H^{(p)}).
\eeqn
Composing with the quasi-Frobenius map, it induces the $\fp$-linear map on cohomology
\beqn
FSt_p: HF(H) \to HF_{\zp}(H^{(p)}; \Lambda_{\kzp}^\Gamma).
\eeqn
It was proved in \cite{Bai_Xu_2025} that in the weakly monotone case, the two definitions agree. 

The approach using geometric perturbations in particular works for the construction of the local version of the equivariant pants product for isolated fixed points. Consider a general closed symplectic manifold $(M, \omega)$, a Hamiltonian $H$ with an isolated capped orbit $x \in \wt{\mc O}(H)$ such that $x^{(p)}$ is also an isolated orbit of $H^{(p)}$. Let $G$ be a sufficiently small perturbation of $H$. Then the above construction gives a local pair-of-pants operation
\beqn
{\it FSt}^{\loc}_{p, x}: HF^{\rm loc}(H, x) \to HF^{\loc}_{\zp} (H^{(p)}, x^{(p)}).
\eeqn

\section{Homological perturbations for degenerate Hamiltonians}\label{section_hp}\label{section3}

In this section, we provide preliminary technical results which will be used in the proofs of results about pseudo-rotations and periodic points of Hamiltonian diffeomorphisms.


\subsection{Hamiltonians with isolated fixed points}\label{subsec:hp}

First, we use homological perturbation techniques to extend our quantitative discussion of Floer theory to degenerate Hamiltonians with isolated fixed points.

To state the main constructions, we first equip the local Floer cohomology with a natural energy filtration. Let $H$ be a 1-periodic Hamiltonian and $x \in {\mc O}(H)$ be an isolated contractible 1-periodic orbit. Let ${\mb K}$ be any ground field. Recall that to construct the local Floer cohomology $HF^{\rm loc}(H, x; \Lambda_{\mb K}^\Gamma)$, one chooses a small nondegenerate perturbation $H_1$ of $H$. Then one can write down the Floer cochain group $CF(H_1; \Lambda_{\mb K}^\Gamma)$ which contains a $\Lambda_{\mb K}^\Gamma$-submodule
\beqn
CF^\loc (H_1, x; \Lambda_{\mb K}^\Gamma)\subset CF(H_1; \Lambda_{\mb K}^\Gamma)
\eeqn
generated by nondegenerate elements in ${\mc O}(H_1)$ which are close to $x$. We modify the action ${\mc A}_{H_1}$ on $CF^\loc(H_1, x; \Lambda_{\mb K}^\Gamma)$ by defining 
\beqn
{\mc A}_H(\ov{x}'):= {\mc A}_H (\ov{x})
\eeqn
if $\ov{x}'$ is a capped 1-periodic orbit of $H_1$ which is close to a capping $\ov{x}$ of $x$. Notice that this does not change the convergence condition required for chains in $CF^\loc (H_1, x; \Lambda_{\mb K}^\Gamma)$. 

Then the Floer differential $d_{H_1}$ on $CF(H_1; \Lambda_{\mb K}^\Gamma)$ (given by Theorem \ref{thm:Floer-cochain}) restricts to a local differential $d^\loc$ on $CF^\loc(H_1, x; \Lambda_{\mb K}^\Gamma)$ which preserves this modified filtration. We can verify that this modified energy filtration induces a well-defined filtration
\beq\label{eqn_local_action}
{\mc A}_{H, x}: HF^\loc (H, x; \Lambda_{\mb K}^\Gamma) \to {\mb R} \cup \{+\infty\}.
\eeq

We extend \cite[Theorem D]{S-PRQS} to the general setting and prove it in Appendix \ref{app: hp}. 

\begin{prop}\label{prop31}
Let $(M,\om)$ be a closed symplectic manifold. Consider a $1$-periodic Hamiltonian $H$ with $\fix(\phi^1_H)$ finite. Then there exists $\epsilon_H>0$, and, for each ground field $\bK$, there exists a ${\mb Z}$-graded complex of $\Lambda_{\mb K}^\Gamma$-modules
\beqn
C(H) = \left( \bigoplus_{x \in {\mc O}(H)}   HF^{\loc}(H, x;  \Lambda_{\mb K}^\Gamma), d_H \right)
\eeqn
graded by the (shifted) Conley--Zehnder index which satisfies the following conditions.

\begin{enumerate}

\item Let ${\mc A}_H: C(H) \to {\mb R} \cup \{+\infty\}$ be the energy filtration induced from ${\mc A}_{H, x}$ (see \eqref{eqn_local_action}). Then for any $y \in C(H)$, one has
\beqn
{\mc A}_H(d_H (y)) \geq  {\mc A}_H(y) + \epsilon_H.
\eeqn



\item $C(H)$ is chain homotopic to the Floer cochain complex $CF(G)$ (given by Theorem \ref{thm:Floer-cochain}) for any nondegenerate Hamiltonian $G$. 

\item For any $a, b \notin {\rm Spec}(H)$, the cohomology of the complex $C(H)^{(a, b)}$ agrees with $HF(H)^{(a, b)}$ which was defined by \eqref{eq: filtered-irrational}. 
\end{enumerate}
\end{prop}

\begin{proof}
See Appendix \ref{app: hp}.
\end{proof}

We have the following corollary immediately from the above statement. It will be used in the proofs of Corollary \ref{cor36} and Lemma \ref{lemma:SDM}.

\begin{cor}\label{cor32}
Let $p$ be a prime. Let $H$ be a Hamiltonian generating a Hamiltonian diffeomorphism $\phi$ which has finitely many contractible fixed points such that
\beqn
N(\phi; \fp) = {\rm dim}_{\fp} H^*(M; \fp).
\eeqn
(For example, when $\phi$ is an $\fp$ Hamiltonian pseudo-rotation.) Then the differential of the complex $C(H)$ vanishes. 
In particular, for all $a \in \R \setminus \spec(H),$ the cohomology $HF(H; \Lambda^\Gamma)^{>a}$ is given by the action-completion of 
\beqn
\bigoplus_{\cl{A}_H(\ol{x}) > a} HF^{\loc}(H,\ol{x}; {\mb F}_p).
\eeqn
\end{cor}

\begin{proof}
By definition, one has
\begin{multline*}
{\rm dim}_{\Lambda^\Gamma} C(H) = \sum_{x \in {\mc O}(H)} {\rm dim}_{\fp} HF^{\rm loc}(H, x; \fp) = N(\phi; \fp) \\
= {\rm dim}_{\fp} H^*(M; {\mb F}_p) = {\rm dim}_{\Lambda^\Gamma} HF(M; \Lambda^\Gamma)
\end{multline*}
which, by (2) of Proposition \ref{prop31}, is also equal to the dimension of the cohomology of $C(H)$. Hence the differential on $C(H)$ vanishes. 
\end{proof}

\subsection{Homological perturbation for equivariant Floer cohomology}

We turn to the equivariant situation. Here, the energy filtration is defined similarly as \eqref{eqn_local_action}.

\begin{prop}\label{prop33b}
Let $H$ be a Hamiltonian on $(M, \omega)$ whose time-1 map has finitely many fixed points. Let $p$ be an odd prime such that ${\rm Fix}_c(\phi) = {\rm Fix}_c(\phi^p)$ and such that all fixed points are $p$-admissible. Then there exist $\epsilon_H>0$ and a complex of modules over $\Lambda_{\kzp}^\Gamma$
\beqn
C_{\zp}(H^{(p)}) = \left( \bigoplus_{x \in {\mc O}(H)} HF^\loc_\zp(H^{(p)}, x^{(p)}), d_{\zp} \right)
\eeqn
satisfying the following conditions.
\begin{enumerate}
    \item The complex $C_{\zp}(H^{(p)})$ is chain homotopic to $CF_\zp(G^{(p)})$ for any $G$ with $G^{(p)}$ being nondegenerate. In particular, the cohomology of $C_{\zp}(H^{(p)})$ is canonically isomorphic to $HF_\zp(M)$.

    \item The differential $d_{\zp}$ increases the energy filtration by $\epsilon_H$.
\end{enumerate}
\end{prop}

\begin{proof}
See Appendix \ref{subsectionb3}.
\end{proof}

One can also transform the equivariant pair-of-pants product to the homologically perturbed complexes constructed above.

\begin{prop}\label{prop35}
Let $(M, \omega)$ be a closed symplectic manifold and $H$ be a 1-periodic Hamiltonian whose time-1 map $\phi$ has  finitely many fixed points such that ${\rm Fix}_c(\phi) = {\rm Fix}_c(\phi^p)$. Let $C(H)$ be the complex provided by Proposition \ref{prop31} and $C_{\zp}(H^{(p)})$ be the complex provided by Proposition \ref{prop33b}. 
Then there exists a cochain map
\beqn
\wt P: C(H)^{\otimes p}_{\zp} \to C_{\zp}(H^{(p)}) 
\eeqn
and a decomposition 
\beqn
\wt P = \wt P^{\rm loc} + \wt P^{\rm big},
\eeqn
which satisfy the following conditions.
\begin{enumerate}

\item Let $P: H(C(H)) \to HF_{\zp}(H^{(p)})$ be the composition of $\wt P$ and the quasi-Frobenius map on the cohomology level. Then the following diagram commutes.
\beqn
\xymatrix{  H(C(H)) \ar[r]^-{P} \ar[d]_{\cong}  &   H(C_{\zp}(H^{(p)})) \ar[d]^{\cong} \\
            HF(M; \Lambda^\Gamma)  \ar[r]_-{\fst_p}   &      HF_{\zp}(M; \Lambda_{\kzp}^\Gamma )  }%
            \eeqn
Here the vertical arrows are provided by or induced from the ones of Proposition \ref{prop31} and Proposition \ref{prop33b}.

\item For any $x \in {\mc O}(H)$, $\wt P^\loc$ restricts to a $\Lambda_{\kzp}^\Gamma$-linear map
\beqn
\wt P^\loc_x: HF^\loc(H, x )^{\otimes p}_{\zp} \cong HF^\loc(H, x)^{\otimes p} \otimes \rp \to HF_{\zp}^\loc(H^{(p)}, x^{(p)})
\eeqn
which coincides with the local equivariant pants product.

\item For any $a \in C(H)^{\otimes p}$, one has 
\beqn
{\mc A}_{H^{(p)}}(\wt P^\loc (a)) \geq {\mc A}_H^{\otimes p} (a).
\eeqn
Moreover, there exists $\epsilon_H>0$ such that for any $a \in C(H)^{\otimes p}_\zp$, one has 
    \beqn
    {\mc A}_{H^{(p)}}(\wt P^{\rm big}(a)) \geq {\mc A}_H^{\otimes p} (a) + \epsilon_H.
    \eeqn

\end{enumerate}
\end{prop}

\begin{proof}
See Appendix \ref{proof_prop35}.
\end{proof}

\subsection{Tate spectrum and bar-length comparison}

We consider the relation between bar-length spectra of $H$ and $H^{(p)}$ which extends the nondegenerate case discussed in Section \ref{section2}. Consider the cochain complex $C(H)$ given by Proposition \ref{prop31} (which is well-defined up to filtered chain homotopy equivalence) with coefficient field ${\mb F}_p$ for a fixed prime $p$. As a Floer-type complex (in the sense of \cite{usher-zhang}), it has a well-defined bar-length spectrum
\beqn
0 < \beta_1(H) \leq \cdots \leq \beta_N(H).
\eeqn
Similarly, one considers the Tate version of the complex given by Proposition \ref{prop33b}, namely 
\beqn
\wh C_{\zp}(H^{(p)}):= C_\zp(H^{(p)}) \underset{\Lambda_{\kzp}^\Gamma}{\otimes} \Lambda_\kp^\Gamma,
\eeqn
which is a Floer-type complex over $\Lambda_{\kzp}^\Gamma$, hence has the corresponding Tate bar-length spectrum (it contains no zero bar because of (2) of Proposition \ref{prop33b})
\beqn
0 < \wh \beta_1(H^{(p)}) \leq \cdots \leq \wh \beta_{N^{(p)}}(H^{(p)}).
\eeqn

One can also consider the version using the universal Novikov field. Namely
\beqn
C(H; \Lambda^{\rm univ}):= C(H) \underset{\Lambda^\Gamma}{\otimes} \Lambda^{\rm univ}
\eeqn
and 
\beqn
\wh C_\zp(H^{(p)}; \Lambda_{\kp}^{\rm univ}):= \wh C_\zp(H^{(p)}) \underset{ \Lambda_{\kp}^\Gamma}{\otimes} \Lambda_{\kp}^{\rm univ}.
\eeqn
The cochain map $\wt P$ obviously extends to 
\beq\label{eqn21}
\wh P: \wh{ C(H; \Lambda^{\rm univ})^{\otimes p}_\zp} \to \wh C_{\zp}(H^{(p)}; \Lambda_\kp^{\rm univ}).
\eeq

\begin{cor}\label{cor35}
Under the assumption of Proposition \ref{prop33b} and Proposition \ref{prop35}, if $p$ is an odd prime\footnote{When $p=2$, we have $N^{(p)} = N$ and $\wh \beta_{i}(H^{(p)}) = p \beta_i(H)$.}, then one has $N^{(p)} = 2 N$ and for each $i = 1, \ldots, N$, there holds
\beqn
\wh\beta_{2i-1}(H^{(p)}) = \wh \beta_{2i}(H^{(p)}) = p \beta_i(H).
\eeqn
\end{cor}

\begin{proof}
Consider a version of the previous construction in this section over the universal Novikov ring $\Lambda_{0}^{\rm univ}$  as follows. Choose for each $x \in {\mc O}(H)$ a capping $\ov{x}$ and define 
\beqn
HF^{\rm loc}(H, x; \Lambda_0^{\rm univ}):= q^{- {\mc A}_H(\ov{x})} \Big( HF^{\rm loc}(H, \ov{x})  {\otimes} \Lambda_0^{\rm univ} \Big) \subset HF^{\rm loc}(H, x; \Lambda^{\rm univ})
\eeqn
and
\beqn
C(H; \Lambda_0^{\rm univ}):= \bigoplus_{x \in {\mc O}(H)} HF^{\rm loc}(H, x; \Lambda_0^{\rm univ}) \subset C(H; \Lambda^{\rm univ}).
\eeqn
Then $d_H$ canonically induces a $\Lambda_0^{\rm univ}$-linear differential on $C(H; \Lambda_0^{\rm univ})$. Similarly, one can define the Tate version 
\beqn
\wh C_{\zp}(H^{(p)}; \Lambda_{0, \kp }^{\rm univ}).
\eeqn
Then the cochain map $\wh P$ of \eqref{eqn21} restricts to a $\Lambda_{0, \kp}^{\rm univ}$-linear cochain map
\beqn
\wh P_0: \wh{ C(H; \Lambda_0^{\rm univ})^{\otimes p}_{\zp}} \to \wh C_{\zp}(H^{(p)}; \Lambda_{0, \kp}^{\rm univ}).
\eeqn
As the local restrictions are chain homotopy equivalences, $\wh P_0$ induces an isomorphism of $\Lambda_{0, \kp}^{\rm univ}$-modules
\beqn
\wh P_0: H ( \wh{ C (H; \Lambda_0^{\rm univ} )_\zp^{\otimes p}}  ) \to H( \wh C_{\zp}(H^{(p)};\Lambda_{0, {\mc K}_p}^{\rm univ})).
\eeqn
Therefore, the Tate bar-length spectrum  is computed by
\beqn
\wh{\beta}_{2i-1}(H^{(p)}) = \wh{\beta}_{2i}(H^{(p)})  = p \beta_i(H). \qedhere
\eeqn
\end{proof}

As a corollary, in situations relevant to our applications, we can prove that the differential of $C_{\zp}(H^{(p)})$ vanishes. 

\begin{cor}\label{cor36}
Under the hypothesis of Proposition \ref{prop33b}, suppose in addition that the time-1 map $\phi$ of $H$ satisfies
\begin{align*}
    &\ {\rm Fix}_c(\phi) = {\rm Fix}_c(\phi^p),\ &\ N(\phi; {\mb F}_p) = N(\phi^p; {\mb F}_p) = {\rm dim}_{{\mb F}_p} H^*(M; {\mb F}_p).
\end{align*}
Then the differential $d_{\zp}$ of the complex of Proposition \ref{prop33b} vanishes.
\end{cor}

\begin{proof}
Indeed, if $d_{\zp} \neq 0$, then the Tate version $\wh C_{\zp}(H^{(p)})$ has a nontrivial differential. As $d_{\zp}$ (which is also the differential of $\wh C_\zp(H^{(p)})$) strictly increases the filtration, the Tate bar-length spectrum has a nonzero element $\wh \beta_1(H^{(p)})$. By Corollary \ref{cor35}, one has $\wh \beta_1(H^{(p)}) = p \beta_1(H) > 0$. However, this contradicts Corollary \ref{cor32}.
\end{proof}

On the other hand, one has a comparison on the shortest bar-lengths.

\begin{lemma}\label{lemma37}
Let $p$ be an odd prime and $\phi \in {\rm Ham}(M, \omega)$. Suppose ${\rm Fix}_c(\phi) = {\rm Fix}_c(\phi^p)$ and all $x \in {\rm Fix}_c(\phi)$ are $p$-admissible. If the barcode of the complex $C(H^{(p)})$ constructed by Proposition \ref{prop31} has a nonzero bar, so does the barcode of the Tate complex $\wh C_\zp(H^{(p)})$. In addition,
\beqn
\beta_1(H^{(p)}) \geq \wh \beta_1(H^{(p)}).
\eeqn
\end{lemma}

\begin{proof}
See Appendix \ref{appendixb}. 
\end{proof}

\section{Proofs of the main results}\label{sec-proofsappl}

Using the technical results presented above, we can prove our main results on Hamiltonian dynamics.

\subsection{The divisibility lemma}

For our applications, we require the following key lemma. We first recall the homological perturbation construction in the context of Proposition \ref{prop33b} and Proposition \ref{prop35}. For an isolated $x \in {\mc O}(H)$ with $x^{(p)} \in {\mc O}(H^{(p)})$ also being isolated, choosing a capping $\ov{x}$, one has the local equivariant pants product 
\beqn
\wt P_{\ov{x}}^\loc: HF^\loc(H, \ov{x})^{\otimes p} \to HF_\zp^\loc(H^{(p)}, \ov{x}{}^{(p)}).
\eeqn

\begin{lemma}[The divisibility lemma]\label{lma:div}
Under the assumption of Proposition \ref{prop31}, \ref{prop33b}, \ref{prop35}, suppose in addition that $x$ is $p$-admissible. Hence by Proposition \ref{prop_local_eq_structure}, one has 
\beqn
HF_\zp^\loc(H^{(p)}, \ov{x}^{(p)}) \cong HF^\loc(H^{(p)}, \ov{x}^{(p)}) \otimes \rp.
\eeqn
Then the restriction of the local equivariant pants product $\wt P^{\rm loc}$ at $\ov{x}$ given by Proposition \ref{prop35} has a local decomposition written as 
\beqn
\wt P_{\ov{x}}^{\rm loc} = \sum_{2k \geq 0} t^k \wt P^\loc_{2k} + \sum_{2k+1 \geq 1} t^k \theta \wt P^\loc_{2k+1}
\eeqn
where 
\beqn
\wt P_i^\loc: HF^{\rm loc}(H, \ov{x})^{\otimes p} \to HF^{\rm loc}(H^{(p)}, \ov{x}{}^{(p)})
\eeqn
are the components. Then for any $y \in HF^{\rm loc}(H, \ov{x})$ with $\wt P^\loc(y\otimes \cdots \otimes y)\neq 0$, one has
\beqn
\inf \Big\{ i\ |\ \wt P_i^\loc (y \otimes \cdots \otimes y) \neq 0 \Big\} \leq 2n (p-1).
\eeqn
\end{lemma}

Since the proof is rather long and technical, we first outline the strategy. The key idea is to reduce the consideration to a local normal form of the fixed point also used in the proof of Proposition \ref{prop_local_eq_structure}. For different reasons, the divisibility condition is true for both the nondegenerate direction and the totally degenerate direction. One then uses the Kunneth property of the local equivariant pants product established in Section \ref{section5} to obtain the product case. Another key property, which is easily neglected, is the fact that the local equivariant Kunneth map is bilinear over the ring $\rp$ (in particular, linear over the odd variable $\theta$) proved as Proposition \ref{prop_local_eq_Kunneth}.

\begin{proof}[Proof of Lemma \ref{lma:div}]
We first reduce the consideration to the case with a local normal form as in the proof of Proposition \ref{prop_local_eq_structure}. 
Recall that the local (equivariant) Floer cohomology is invariant under deformation as long as the fixed point remains uniformly isolated through the deformation. Moreover, the identifications between different moments through the deformation are given by certain continuation maps. As the local equivariant pants product is compatible with continuation maps, one can assume that $H$ has a product type. In other words, let $\phi$ be the time-1 map of $H$. Then we may assume that $M = X_1 \times X_2$ and $\phi = \phi_1 \times \phi_2$ are products. If $x = (x_1, x_2)$, then $x_1$ is a nondegenerate fixed point of $\phi_1$ and $x_2$ is a totally degenerate and isolated fixed point of $\phi_2$. 

To save notations, abbreviate $HF^\loc(\phi, x)$ by $HF^\loc(\phi)$ (and in other similar occasions) as the critical points in the consideration are clear from the context. Let
\beqn
\wt P_{x_i}^\loc:= \wt{\fst}{}_{p, x_i}^\loc: HF^\loc(\phi_i)^{\otimes p} \to HF_\zp^\loc(\phi_i^p) \cong HF^\loc(\phi_i^p) \otimes \rp
\eeqn
be the local equivariant pants product and let 
\beqn
P_{x_i}^\loc: HF^\loc(\phi_i) \to HF_\zp^\loc(\phi_i^p)
\eeqn
be their compositions with the quasi-Frobenius map $y \mapsto y^{\otimes p}$.

%
%
%

The basic strategy in the rest of the proof is to use the Kunneth property of the local equivariant pants product (Proposition \ref{prop55}) and the leading order bounds of the equivariant pants product on the two factors. 
Then by Proposition \ref{prop_local_eq_Kunneth} and Proposition \ref{prop55} proved below, the following diagram commutes up to the Koszul sign.
\beqn
\xymatrix{   HF^\loc (\phi)  \ar[rr]^{P^\loc }     &     &  HF_\zp^\loc (\phi^p)  \\
             HF^\loc (\phi_1) \otimes HF^\loc (\phi_2) \ar[rr]_-{P_{x_1}^\loc \otimes P_{x_2}^\loc } \ar[u]^{\kappa^{\rm loc}}  & &   HF_\zp^\loc  (\phi_1^p)  \underset{\rp}{\otimes} HF_\zp^\loc  (\phi_2^p) \ar[u]_{\kappa_\zp^\loc }}
\eeqn
Notice that the non-equivariant Kunneth map $\kappa^{\rm loc}$ is an isomorphism. On the other hand, from the proof of Proposition \ref{prop_local_eq_structure} one knows that $\kappa_\zp^\loc$ is also an isomorphism, although it wouldn't be true if we drop the $p$-admissibility condition.

Now we prove that both $P_{x_1}^\loc$ and $P_{x_2}^\loc$ satisfy the bound we would like to prove. We first analyze the totally degenerate piece $P_{x_2}^\loc$. As $x_2$ is also $p$-admissible, by Proposition \ref{prop_local_eq_structure}, one has an isomorphism of $\rp$-modules
\beqn
HF_\zp^\loc(\phi_2^p) \cong HF^\loc(\phi_2^p) \otimes \rp.
\eeqn

\noindent {\it Claim.} If we decompose
\beqn
P_{x_2}^\loc = \sum_{2k\geq 0} t^k P_{x_2, 2k}^\loc + \sum_{2k+1\geq 1} t^k \theta P_{x_2, 2k+1}^\loc,
\eeqn
then 
\beqn
P_{x_2, i}^\loc = 0 \ {\rm when}\ i > 2(p-1)n_2 = (p-1) {\rm dim}_{\mb R} X_2.
\eeqn

\noindent {\it Proof of the claim.} The basic idea is to reduce the consideration to local Steenrod operations on local Morse cohomology. By the argument of Proposition \ref{prop_local_eq_structure}, the local equivariant cohomology can be computed using an autonomous Hamiltonian $h$ and the isomorphism is provided by a continuation map. Therefore, the cochain level map $\fst_{p, x}^\loc$ can also be defined using the autonomous $h$, denoted by 
\beqn
P^\loc_{\rm Morse}: HM^\loc(h, 0) \to HM^\loc(h, 0) \otimes \rp.
\eeqn
We know that the local Morse cohomology is identified with a relative cohomology $H^*(B, N)$ with $B$ being a $2n_2$-dimensional ball while $N \subset \partial B$ is a codimension zero submanifold with boundary. 
Using an argument analogous to \cite[Section 3.2]{wilk} for general primes $p$, one can show that $P_{\rm Morse}^\loc$ is identified with the classical Steenrod operation 
\beqn
{\it St}_p: H^*(B,N) \to H^*(B,N) \otimes \rp
\eeqn
given by \cite[Equation (1.6)]{Seidel-formal}. By the classical theory of Steenrod operations, the coefficient of $t^\alpha \theta^\beta$ of $St_p(z)$ for $|z|=k$ vanishes when $2\alpha + \beta > k (p-1)$. As $H^*(B, N)$ has non-vanishing degrees up to ${\rm dim} B = 2n_2$, the assertion of the claim follows. (The required conclusion also follows directly from an analysis of the Morse moduli spaces defining $P_{\rm Morse}^{\loc}$, as in the proof of \cite[Proposition 3.3]{wilk}, specifically ``Axiom 5".) \hfill {\it End of the proof of the claim.}

On the other hand, the nondegenerate piece is an easy consequence of Theorem \ref{thm_local_pants}. If $e_1 \in HF^{\rm loc}(\phi_1)$ is a generator (which induces a corresponding generator $e_1^{(p)}$ of $HF_\zp^\loc  (\phi_1^p)$), then there exists $a_1 \leq n_1 (p-1)$ and $c_1 \neq 0$ such that
\beq\label{eqn25}
P^\loc_{x_1}  ( e_1 ) = c_1 t^{a_1} e_1^{(p)}.
\eeq



Now we can use the Kunneth property of the local equivariant pants product (Proposition \ref{prop55}). Let $e_1 \in HF^{\rm loc}(\phi_1; {\mb F}_p)$ be a generator and $y_2 \in HF^{\rm loc}(\phi_2; {\mb F}_p)$. Suppose $P^{\rm loc}( e_1 \otimes y_2 ) \neq 0$. Then by \eqref{eqn25} and the Kunneth property, one has
\beqn
\begin{split}
&\ \pm P^{\rm loc}( \kappa^{{\rm loc}}(e_1 \otimes y_2)) \\
= &\ \kappa_\zp^\loc    \Big( P^{\rm loc}_{x_1}(e_1) \otimes P^{\rm loc}_{x_2}(y_2) \Big) \\
= &\ \kappa_\zp^\loc \Big( c_1 t^{a_1} e_1^{(p)} \otimes P^{\rm loc}_{x_2} (y_2) \Big) \\
= &\ c_1 t^{a_1} \kappa_\zp^\loc  \Big( e_1^{(p)} \otimes P^{\rm loc}_{x_2} (y_2) \Big)\\
= &\ c_1 t^{a_1} \kappa_\zp^\loc \Big( \sum_{2k} t^k e_1^{(p)}\otimes P^\loc_{x_2, 2k}(y_2) + \sum_{2k+1} t^k \theta e_1^{(p)} \otimes P^\loc_{x_2, 2k+1}(y_2) \Big)\\
= &\ c_1 \Big( \sum_{2k} t^{a_1 + k} \kappa_{\zp}^\loc( e_1^{(p)} \otimes P_{x_2, 2k}^\loc(y_2) ) + \sum_{2k+1} t^{a_1 + k} \theta \kappa_{\zp}^\loc( e_1^{(p)} \otimes P_{x_2, 2k+1}^\loc(y_2)) \Big).
\end{split}
\eeqn
Here we used the property that $\kappa_\zp^\loc$ is $\rp$-bilinear.

Now let $P^\loc_{x_2, i_2}(y_2)$ be the leading order term of $P_{x_2}^\loc (y_2)$. Notice that $\kappa_\zp^\loc$ does not decrease the $(t, \theta)$-grading and its leading order term is the non-equivariant Kunneth map $\kappa^{\rm loc}$, which is an isomorphism. Then the leading term of $P^{\rm loc}(\kappa^{\rm loc}(e_1 \otimes y_2))$ is a nonzero multiple of 
\beqn
\kappa^{\rm loc}(e_1^{(p)} \otimes P^\loc_{x_2, i_2}(y_2))
\eeqn
which has $(t, \theta)$-degree no greater than $2n(p-1)$. 
\end{proof}

\subsection{Technical results on symplectically degenerate maxima}

We need a small digression on filtered Floer cohomology of iterates of a Hamiltonian near a symplectically degenerate maximum, which will be used towards the proof of Theorem \ref{thm: pr}, Theorem \ref{thm: infinitely many}, and Theorem \ref{thm: sdm}. Notice that the statement of the following lemma depends on Proposition \ref{prop31}.

\begin{lemma}\label{lemma:SDM}
Let $H$ be a 1-periodic Hamiltonian such that $\ov{x} \in \tilde {\mc O}(H)$ is a symplectically degenerate maximum (whose mean index, denoted by $\Delta$, is necessarily an integer). Denote $c = {\mc A}_H(\ov{x})$. Then for every sufficiently small $\eps > 0$, there exists a sufficiently large natural number $r_{\eps} > 0$ such that for all $r \geq r_{\eps}$ there exists $\delta = \delta_{\eps,r} \in (0,\eps)$ such that the following connecting homomorphism 
\begin{equation}\label{eq: connecting}
     HF^{r\Delta+n}(H^{(r)})^{(rc+\delta, rc+\eps)} \to HF^{r\Delta+n+1}(H^{(r)})^{(rc-\delta,rc+\delta)},\end{equation}
     associated to the short exact sequence
     \beqn
     \xymatrix{ 0 \ar[r] & C(H^{(r)})^{(rc-\delta, rc+\delta)} \ar[r] & C(H^{(r)})^{(rc-\delta, rc+\epsilon)} \ar[r] & C(H^{(r)})^{(rc + \delta, rc + \epsilon)} \ar[r] & 0 } 
     \eeqn
     is nontrivial. 
\end{lemma}

\begin{proof}
    The assertion follows from the methods of proof of \cite[Proposition 4.7]{Ginzburg-CC}, \cite[Theorem 1.18]{GG-ai}, and \cite[Theorem 1.3]{Hein-CCCY}. Indeed, \cite{Hein-CCCY} adapts the methods of \cite{Ginzburg-CC} to show that on an arbitrary closed symplectically Calabi--Yau manifold, there holds
    \beqn
    HF^{r\Delta+n+1}(H^{(r)})^{(rc+\delta, rc+\eps)} \neq 0.
    \eeqn  
The argument in \cite{Hein-CCCY} applies equally well for an arbitrary closed symplectic manifold, as it relies on a general decomposition theorem for filtered Floer homology and Ginzburg's arguments \cite{Ginzburg-CC} applied to the ``local summand" of the decomposition.

Furthermore, Ginzburg proves that the above connecting map is non-trivial in the aspherical case, and this extends, by applying Hein's decomposition, to the case of an arbitrary closed symplectic manifold.
\end{proof}

To continue, we also need the following lemma for our applications. When $(M,\om)$ is rational in the sense that $\brat{[\om],\pi_2(M)} = \rho \Z$ for a constant $\rho \geq 0,$ Lemma \ref{lma:sdm} follows immediately from \cite[Theorem 1.18]{GG-ai}. When $(M,\om)$ is not rational, additional arguments, which use more of the techniques of symplectically degenerate maxima, are required. We present the proof below.

\begin{lemma}\label{lma:sdm}
Let $\phi$ be an $\F_p$ Hamiltonian pseudo-rotation generated by a Hamiltonian $H$. Then for all $k\geq 0$, $\tilde {\mc O}(H^{(p^k)})$ contains no symplectically degenerate maximum.
\end{lemma}

\begin{proof}
Suppose on the contrary, for some $k\geq 0$, $\tilde {\mc O}(H^{(p^k)})$ contains some symplectically degenerate maximum. As $\phi$ is an $\fp$ pseudo-rotation, $\phi^{p^k}$ is also an $\fp$ pseudo-rotation. Hence we may assume that $k = 0$. By Lemma \ref{lemma:SDM}, if $\ov{x} \in \tilde{\mathcal{O}}(H)$ is a symplectically degenerate maximum, then for every sufficiently small $\epsilon > 0$, there exists $r_{\epsilon} \in \mathbb{Z}_{\geq 0}$ such that for all $r \geq r_{\epsilon}$ we can find $\delta \in (0,\epsilon)$ such that the connecting homomorphism \eqref{eq: connecting} is nonzero. However, it follows from Corollary \ref{cor32} that when $r = p^k$, the differential of $C(H^{(p^k)})$, as well as that of $C(H^{(p^k)})^{(a,b)}$ for an arbitrary window $(a,b)$ vanishes. Hence the connecting homomorphism \eqref{eq: connecting} also vanishes, which is a contradiction. This finishes the proof.
\end{proof}

\begin{remark}\label{rmk:sdm for primes}
The same argument applies in the setting of Theorem \ref{thm: infinitely many}.     
\end{remark}

\begin{remark}\label{rmk:sdm for perfect}
The proof by contradiction in the above argument generalizes to the setting of a Hamiltonian diffeomorphism $\phi$ with finitely many periodic points. The argument itself does not use the interplay between quantum Steenrod operations and equivariant pants product; we therefore defer the proof of this more general version, which leads to the proof of Theorem \ref{thm: sdm}, to the end of this Section.
\end{remark}

\subsection{Proof of Theorem \ref{thm: pr}}\label{subsection_proof_uniruled}

Let $\phi$ be an $\fp$ Hamiltonian pseudo-rotation. Choose a Hamiltonian $H$ that generates $\phi$. Theorem \ref{thm: pr} follows from the following proposition, whose proof resembles that in \cite{CGG2}.

\begin{prop}\label{prop_carrier_iteration}
Under the assumption of Theorem \ref{thm: pr}, let $\phi\in {\rm Ham}(M)$ be an $\fp$ Hamiltonian pseudo-rotation generated by a Hamiltonian $H$. Suppose $\qst_p(\mu) = St_p(\mu)$. Then there exists $\ov{x} \in \tilde {\mc O}(H)$ such that for all sufficiently large $k\geq 0$, 
\beqn
HF^{\loc, 2n} (H^{(p^k)}, \ov{x}^{(p^k)}; \fp) \neq 0.
\eeqn
\end{prop}

\begin{proof}[Proof of Theorem \ref{thm: pr}]
Assume on the contrary that $\qst_p(\mu) = St_p(\mu)$. Let $\phi$ be an $\fp$ Hamiltonian pseudorotation generated by a Hamiltonian $H$. Let $\ov{x} \in \tilde {\mc O}(H)$ be provided by Proposition \ref{prop_carrier_iteration}. Then by Lemma \ref{lma:sdm}, for any  $k\geq 0$, the iteration $\ov{x}^{(p^k)}$ cannot be a symplectically degenerate maximum. Then for all $k\geq 0$, $\Delta(\ov{x}^{(p^k)}) \neq 0$. 

By Proposition \ref{prop_carrier_iteration},  when $k$ is sufficiently large, $HF^{\loc, 2n} (H^{(p^k)}, \ov{x}^{(p^k)}) \neq 0$. It follows that the mean index $\Delta(\ov{x}^{(p^k)})$ is negative because $HF^{\rm loc}(H^{(p^k)}, \ov{x}^{(p^k)})$ is supported over degrees in the interval $[-\Delta(\ov{x}^{(p^k)}), -\Delta(\ov{x}^{(p^k)}) + 2n]$ using our grading convention. By homogeneity of the mean index, we see that
\beqn
\Delta(\ov{x}^{(p^k)}) = p^k \Delta(\ov{x}) \xrightarrow{k \to \infty} -\infty. 
\eeqn
However, just as above, note that $HF^{\rm loc} (H^{(p^k)}, \ov{x}^{(p^k)})$ is supported over degrees in the interval $[-\Delta(\ov{x}^{(p^k)}), -\Delta(\ov{x}^{(p^k)}) + 2n]$, which contradicts the conclusion of Proposition \ref{prop_carrier_iteration} when $k$ is sufficiently large. Therefore, $\qst_p(\mu) \neq St_p(\mu)$.
\end{proof}

Now we start to prove Proposition \ref{prop_carrier_iteration}. 

\begin{proof}[Proof of Proposition \ref{prop_carrier_iteration}]
We modify the pseudo-rotation to satisfy the condition of Lemma \ref{lma:div}. As $\phi$ is an $\fp$ Hamiltonian pseudo-rotation, we may replace $\phi$ by $\phi^{p^k}$ for any $k\geq 0$. Hence we may assume that for all $x \in {\rm Fix}_c(\phi)$, 
\beqn
{\rm ker}(d\phi_x - {\rm Id}) = {\rm ker}(d\phi_x^p - {\rm Id}).
\eeqn

Now we translate the condition on the quantum Steenrod operation $\qst_p(\mu) = t^{n(p-1)} \mu$ to one about the equivariant pants product. Using Theorem \ref{thm214}, one has 
\beqn
\fst_p( \pss(\mu)) = t^{n(p-1)} \ssp_{\zp}^{-1}(\mu) \in HF_\zp(M).
\eeqn
By abuse of notations, via $\pss$, we view $\mu$ as its image $\pss(\mu) \in HF(M)$. On the other hand, denote $\mu_\zp:= \ssp_\zp^{-1}(\mu) \in HF_\zp(M)$. Then one has 
\beq\label{eqn27}
\fst_p(\mu) = t^{n(p-1)} \mu_{\zp} \in HF_{\zp}^{2n}(M; \Lambda^\Gamma) \otimes \rp.
\eeq

Now one can localize the above equation to periodic orbits. Denote
\beqn
{\mc O}(H) = \{ x_1, \ldots, x_m\}.
\eeqn
By Proposition \ref{prop31}, as $\phi$ is an $\fp$ Hamiltonian pseudo-rotation, for all $k\geq 0$, one has
\beqn
HF(M; \Lambda^\Gamma) \cong HF(H^{(p^k)}) = \bigoplus_{i = 1}^m  HF^{\rm loc}(H^{(p^k)}, x_i^{(p^k)}; \Lambda^\Gamma).
\eeqn
With respect to this decomposition, we can write 
\beqn
\mu = \sum_{i=1}^m \mu_i^{(p^k)},\ {\rm where}\ \mu_i^{(p^k)} \in HF^{\loc, 2n}  (H^{(p^k)}, x_i^{(p^k)}; \Lambda^\Gamma).
\eeqn
We consider a similar decomposition for the equivariant Floer cohomology. Applying Proposition \ref{prop33b} and Corollary \ref{cor36}, one has 
\beqn
HF_{\zp}(M; \Lambda_{\kzp}^\Gamma) \cong HF_{\zp} (H^{(p^k)}; \Lambda_{\kzp}^\Gamma) = \bigoplus_{i=1}^m HF^{\rm loc}(H^{(p^k)}, x_i^{(p^k)}; \Lambda_{\fp}^\Gamma) \otimes \rp.
\eeqn
So the class $\mu_{\zp}$ can be decomposed as 
\beq\label{equivariant_decomposition}
\mu_{\zp} = \sum_{i=1}^m \mu_{\zp,i}^{(p^k)}(t, \theta),\ {\rm where}\ \mu_{\zp,i}^{(p^k)}(t, \theta) \in HF^{\loc, 2n} (H^{(p^k)}, x_i^{(p^k)}; \Lambda^\Gamma) \otimes \rp.
\eeq

{\it Claim A.} For all $k\geq 0$ and $i$, the leading order term $\mu_{\zp,i}^{(p^k)}(0,0)$ coincides with $\mu_i^{(p^k)}$ up to higher energy terms. 

{\it Proof of Claim A.} On the cochain level, $\mu_{\zp,i}^{(p^k)}(t, \theta)$ is realized by the inverse of the cochain-level equivariant SSP map. The involved moduli spaces contain the leading one corresponding to the non-equivariant SSP map, which gives $\mu_i^{(p^k)}$ up to high energy terms by the results in \cite[Section 13.7]{Bai_Xu_2025}. \hfill {\it End of the proof of Claim A.}

One can also view $\mu_{\zp,i}^{(p^k)}(t, \theta)$ as a class in the Tate cohomology
\beqn
\wh{HF}{}_\zp^\loc (H^{(p^k)}, x_i^{(p^k)}; \Lambda_{{\mc K}_p}^\Gamma) \cong HF^{\loc}(H^{(p^k)}, x_i^{(p^k)}; \Lambda^\Gamma) \otimes \Lambda_{{\mc K}_p} \otimes E(\theta).
\eeqn
What we do is to consider the ordinary and Tate spectral invariants and their ``carriers'' under the $p^k$-th iterations. 

{\it Claim B.} For all $k \geq 1$, there holds
\beqn
c(\mu, H^{(p^k)}) \geq \wh c(\mu_{\zp}, H^{(p^k)}) = p c(\mu, H^{(p^{k-1})}) \geq p^k c(\mu, H)
\eeqn

{\it Proof of Claim B.}
The decomposition of $\mu$ implies that 
\beqn
c(\mu, H^{(p^k)}) = \min_i {\mc A}_{H^{(p^k)}}(\mu_i^{(p^k)}).
\eeqn
The Tate version implies
\beqn
\wh c(\mu_{{\mb Z}/p}, H^{(p^k)}) = \min_i \wh{\mc A}_{H^{(p^k)}}( \mu_{\zp, i}^{(p^k)}).
\eeqn
Claim A implies that 
\beqn
{\mc A}_{H^{(p^k)}}(\mu_i^{(p^k)}) \geq \wh{\mc A}_{H^{(p^k)}}(\mu_{\zp, i}^{(p^k)}) \Longrightarrow c(\mu, H^{(p^k)}) \geq \wh c(\mu_{\zp}, H^{(p^k)}).
\eeqn
Moreover, as the equivariant pants product rescales the spectral invariants by $p$, one has 
\beqn
\wh c(\mu_{\zp}, H^{(p^{k+1})}) = \wh c(t^{n(p-1)} \mu_{\zp}, H^{(p^{k+1})}) = p c(\mu, H^{(p^k)}).
\eeqn
Then inductively, $c(\mu, H^{(p^k)}) \geq p^k c(\mu, H)$. \hfill {\it End of proof of Claim B.}

Claim B has a local version. Suppose for some $i$, ${\mc A}_{H^{(p^k)}}(\mu_i^{(p^k)}) = c(\mu, H^{(p^k)})$.\footnote{One can roughly regard $x_i$, or its $p^k$-th iteration, as a spectral carrier.} Then  
\beq\label{eqn_local_growth}
{\mc A}_{H^{(p^{k+1})}}(\mu_i^{(p^{k+1})}) \geq c(\mu, H^{(p^{k+1})}) \geq p c(\mu, H^{(p^k)}) = p {\mc A}_{H^{(p^k)}}(\mu_i^{(p^k)}).
\eeq

The proposition will follow from a specific quantitative statement and induction. For each $k\geq 0$, define
\beqn
I^{(p^k)}:= \Big\{ i \in \{1, \ldots, m\}\ |\ p^k c(\mu, H) = {\mc A}_{H^{(p^k)}} (\mu_i^{(p^k)}) \Big\}.
\eeqn

{\it Claim C.} There holds 
\beq\label{carrier_iteration}
\bigcap_{k \geq 0} I^{(p^k)} \neq \emptyset.
\eeq

We prove this assertion later. If it is true, then it follows that $c(\mu, H^{(p^k)}) = p^kc (\mu, H)$ for all $k$. Choose $i$ from the nonempty intersection. As all classes have a fixed degree $2n$, by the definition of the group $\Gamma$ and the Novikov ring, there exists a unique capping $\ov{x}_i$ of $x_i$ and a class
\beqn
0 \neq \ov{\mu}{}_i^{(1)} \in HF^{\loc, 2n} (H, \ov{x}_i; \fp) \subset HF^{\loc, 2n} (H, x_i; \Lambda^\Gamma)
\eeqn
such that 
\beqn
\mu_i^{(1)} = a_i^{(1)} \ov\mu_i^{(1)},\ a_i^{(1)} \in \Lambda^\Gamma,\ \nu(a_i^{(1)}) = 0.
\eeqn
Notice that the $p^k$-th iterations of $\ov{x}_i$ still have degree $2n$ and action being $p^k$ times the action of $\ov{x}_i$. It follows that one can write
\beqn
0 \neq \mu_i^{(p^k)} = a_i^{(p^k)} \ov\mu_i^{(p^k)}\ {\rm where}\ \nu(a_i^{(p^k)}) = 0,\ \ov\mu_i^{(p^k)} \in HF^{\loc, 2n} (H^{(p^k)}, \ov{x}_i^{(p^k)}; \fp).
\eeqn
Therefore, the assertion of Proposition \ref{prop_carrier_iteration} holds.



Now we prepare to prove Claim C. By induction, it is equivalent to 
\beqn
\emptyset \neq I^{(p)} \subset I^{(1)}.
\eeqn
The inclusion $I^{(p)} \subset I^{(1)}$ follows from \eqref{eqn_local_growth}. To prove the nonemptiness of $I^{(p)}$, one has to consider the Tate spectral numbers and use the decomposition of the equivariant pants product given by Proposition \ref{prop35}. Indeed, one can write
\beqn
t^{n(p-1)} \mu_{\zp} = t^{n(p-1)} \sum_{i=1}^m \mu_{\zp,i}^{(p)}(t, \theta) = P^{\rm loc}(\mu) + P^{\rm big}(\mu) = \sum_{i=1}^m P^{\rm loc}(\mu_i^{(1)}) + P^{\rm big}(\mu).
\eeqn
We know that $P^{\rm big}$ strictly increases the energy filtration and is additive. Hence 
\begin{multline*}
\wh{\mc A}_{H^{(p)}} \left( P^{\rm big}(\mu) \right) = \wh{\mc A}_{H^{(p)}} \Big( \sum_{i=1}^m P^{\rm big}(\mu_i^{(1)}) \Big) \geq \inf_{1\leq i \leq m} \wh{\mc A}_{H^{(p)}} \Big( P^{\rm big}(\mu_i^{(1)}) \Big) \\
> \inf_{1 \leq i \leq m} p {\mc A}_H (\mu_i^{(1)}) = p c(\mu, H).
\end{multline*}
On the other hand, since the (Tate version) operation $\wh{\fst}_p$ precisely rescales the energy filtration by $p$, we have
\begin{multline*}
pc(\mu, H) = \wh{c}( \wh{\fst}_p(\mu)) = \wh{\mc A}_{H^{(p)}}( P^{\rm loc}(\mu) + P^{\rm big}(\mu))\\
= \wh{\mc A}_{H^{(p)}}( P^{\rm loc}(\mu)) = \wh {\mc A}_{H^{(p)}} \Big( \sum_{i = 1}^m P^{\rm loc}(\mu_i^{(1)})\Big) = \inf_{1\leq i \leq m} \wh {\mc A}_{H^{(p)}} (P^{\rm loc}(\mu_i^{(1)})).
\end{multline*}
As the local equivariant pair-of-pants product rescales the local action by $p$ times,
\beqn
i \in I^{(1)} \Longrightarrow \wh{\mc A}_{H^{(p)}}( P^{\rm loc}(\mu_i^{(1)})) = p {\mc A}_H(\mu_i^{(1)}) = p c(\mu, H).
\eeqn
On the other hand, using \eqref{equivariant_decomposition}, one has
\beq\label{eqn28}
pc(\mu, H) = \wh{c}(t^{n(p-1)} \mu_{\zp}, H^{(p)}) = \inf_{1 \leq i \leq m} \wh {\mc A}_{H^{(p)}} ( \mu_{\zp,i}^{(p)}(t, \theta)).
\eeq

{\it Claim D.} For each $i \notin I^{(1)}$, one has 
\beqn
\wh{\mc A}_{H^{(p)}}(\mu_{\zp,i}^{(p)}(t, \theta)) > p c(\mu, H).
\eeqn

{\it Proof of Claim D.} Choose $i \notin I^{(1)}$. So ${\mc A}_H(\mu_i^{(1)}) > c(\mu, H)$. Notice that $\mu_{\zp,i}^{(p)}(t, \theta)$, if nonzero, either comes from $P^{\rm loc}(\mu_i^{(1)})$ or $P_{\rm big}(\mu)$. Both of them have action strictly greater than $p c(\mu, H)$. \hfill {\it End of proof of Claim D.}

Now we can prove Claim C. First, Claim D and \eqref{eqn28} imply that 
\beqn
pc(\mu, H) = \inf_{i \in I^{(1)}} \wh {\mc A}_
{H^{(p)}}( \mu_{\zp,i}^{(p)}(t, \theta)).
\eeqn
Hence there exists some $i_0 \in I^{(1)}$ such that
\beqn
\wh{\mc A}_{H^{(p)}}(P^{\rm loc}(\mu_{i_0}^{(1)})) = p {\mc A}_H(\mu_{i_0}^{(1)}) =  p c(\mu, H) = \wh{\mc A}_{H^{(p)}}( \mu_{\zp, i_0}^{(p)}(t, \theta)).
\eeqn
Now the $i_0$-th component of \eqref{eqn27} reads
\beqn
P^\loc(\mu_{i_0}^{(1)}) + \Big[ P^{\rm big}(\mu) \Big]_{i_0} = t^{n(p-1)} \mu_{\zp, i_0}^{(p)}(t, \theta) \in HF^{\loc, 2n}(H^{(p)}, x_{i_0}^{(p)}) \otimes \rp.
\eeqn
As the $P^{\rm big}(\mu)$ term has strictly higher energy, the first term on the right (ordered by $(t, \theta)$-degrees) which realizes the action, coincides with the leading order term of $P^{\rm loc}(\mu_{i_0}^{(1)})$. However, by Lemma \ref{lma:div}, the leading degree can be at most $2n(p-1)$. Hence
\beqn
pc(\mu, H) = \wh {\mc A}_{H^{(p)}}(P^\loc(\mu_{i_0}^{(1)})) = \wh{\mc A}_{H^{(p)}}(\mu_{\zp, i_0}^{(p)}(0, 0)).
\eeqn
By Claim A, this also agrees with ${\mc A}_{H^{(p)}}(\mu_{i_0}^{(p)})$. This finishes the proof of Claim C and hence the proposition. 
\end{proof}

\subsection{Proof of Theorem \ref{thm_torsion}}\label{subsection_proof_no_torsion}

The proof is a combination of the proof of Theorem \ref{thm: pr} and the corresponding result that has been established in the monotone setting \cite{AS-torsion}.



\begin{proof}[Proof of Theorem \ref{thm_torsion}] Without loss of generality, we may assume that $G = {\mb Z}/p'$ for a prime $p'$, generated by $\phi \in {\rm Ham}(M, \omega)$. We also assume by contradiction that for some  prime $p \neq p'$, $\qst_p(\mu) = St_p(\mu)$, which can be translated to \eqref{eqn27}, i.e.
\beqn
\fst_p( \pss(\mu)) = t^{n(p-1)} \mu_{\zp} = t^{n(p-1)} \ssp^{-1}_{{\mb Z}/p}(\mu) \in HF_\zp(M).
\eeqn
We focus on the case $p\neq 2$; the argument for the $p=2$ case is similar. Choose a Hamiltonian $H$ generating $\phi$. By Proposition \ref{propc17}, for all $k \geq 0$, one has a canonical decomposition 
\beqn
HF(M; \Lambda^\Gamma) = \bigoplus_{{\mc F} \in \pi_0({\rm Fix}_c(\phi))} HF^{\rm loc}(H^{(p^k)}, {\mc F}^{(p^k)}; \Lambda^\Gamma).
\eeqn
Let ${\mc F}_1, \ldots, {\mc F}_m$ be all the connected components of ${\rm Fix}_c(\phi)$. Then one can decompose
\beqn
\pss(\mu) = \mu_1^{(p^k)} + \cdots + \mu_m^{(p^k)}\ {\rm where}\ \mu_i^{(p^k)} \in HF^\loc (H^{(p^k)}, {\mc F}_i^{(p^k)}; \Lambda^\Gamma).
\eeqn
Proposition \ref{propc17} also provides a decomposition of $\mu_\zp$ written as 
\beqn
\mu_{\zp} = \sum_{i=1}^m \mu_{\zp,i}^{(p^k)}(t, \theta),
\eeqn
where
\beqn
 \mu_{\zp,i}^{(p^k)}(t, \theta) \in  HF^{\rm loc}(H^{(p^k)}, {\mc F}_{i}^{(p^k)}; \Lambda_\rp^\Gamma).
\eeqn
The rest of the proof is completely parallel to that of Theorem \ref{thm: pr}. The key is an analogue of Proposition \ref{prop_carrier_iteration}, which says that there exists a capped component $\ov{\mc F}$ such that for all sufficiently large $k$, 
\beqn
HF^{\loc, 2n}(H^{(p^k)}, \ov{\mc F}{}^{(p^k)}; \fp) \neq 0.
\eeqn
Notice that the divisibility lemma holds in this case as the local equivariant pair-of-pants product coincides with the classical Steenrod operation (Proposition \ref{propc13}) where the roles of Proposition \ref{prop33b} and Proposition \ref{prop35} are replaced by Proposition \ref{propc14}. 
\end{proof}

\subsection{Proof of Theorrem \ref{thm:contractible}}
The proof here follows from the same ideas as in the proof of \cite[Theorem H]{AS-torsion} using the results in Appendix \ref{app:MB}.

\begin{proof}[Proof of Theorrem \ref{thm:contractible}]
    Let's first consider the case when $\phi^{p^k} = \mathrm{Id}$. Using Proposition \ref{homological_perturbation_Morse_Bott} and the PSS isomorphism \cite[Theorem E]{Bai_Xu_2025}, if $H$ is the Hamiltonian generating $\phi$, we see that 
    \beqn
    \sum_{{\mc F} \in \pi_0({\mc O}(H)^{\rm{cont}})} \dim HF^{\rm{loc}}(H,{\mc F}; \Lambda_{{\mb F}_p}^{\Gamma}) \geq \dim H^*(M;\Lambda_{{\mb F}_p}^{\Gamma}),
    \eeqn
    where we only sum over \emph{contractible} components of the fixed points. Using Proposition \ref{propc9} and the identification between local Floer cohomology under different cappings, we see that
    \beqn
    \sum_{{\mc F} \in \pi_0({\mc O}(H)^{\rm{cont}})} \dim H^*({\mc F};{\mb F}_p) \geq \dim H^*(M;{\mb F}_p).
    \eeqn
    On the other hand, by applying the classical Smith inequality iteratively, we have
    \beqn
        \sum_{{\mc F} \in \pi_0({\mc O}(H))} \dim H^*({\mc F};{\mb F}_p) \leq \dim H^*(M;{\mb F}_p).
    \eeqn
    If there were noncontractible components of ${\mc O}(H)$, since $\phi$ is Morse--Bott, the above two inequalities cannot hold simultaneously.

    In general, we have $\phi^{p_1^{k_1} \cdots p_l^{k_l}} = \rm{Id}$ where $p_1, \dots, p_l$ are distinct prime numbers. Then, we can induct on $l$ following \cite[Section 5.10]{AS-torsion}. This finishes the proof.
\end{proof}

\subsection{Proof of Theorem \ref{thm: hyperbolic} and corollaries}\label{subsection_proof_hyperbolic}

First, we may replace $\phi$ by $\psi := \phi^{k_0}$ for some large integer $k_0$ such that 
all periodic points of $\psi$ are fixed points and all capped periodic orbits have Conley--Zehnder index congruent to $n$ mod $2$. Let $H$ be a Hamiltonian generating $\psi$. Since all the fixed points of $\psi$ are hyperbolic, we know
\beqn
{\rm CZ}(\ov{x}^{(k)}) = k {\rm CZ}(\ov{x}),\ \forall \ov{x} \in \tilde{\mc O}(H),\ \forall k \geq 1.
\eeqn
Since the cohomological grading $r=|\ol{x}|$ of $\ol{x}$ satisfies $|\ol{x}| = \CZ(\ol{x})+n,$ this means that for a prime $p$, 
\beqn
p|\ol{x}| - |\ol{x}^{(p)}| = p \CZ(\ol{x}) + pn - \CZ(\ol{x}^{(p)})-n = n(p-1).
\eeqn
By the invertibility of $\fst_{p, x}^\loc$ (see Theorem \ref{thm_local_pants}) one has
\beq\label{eqn31}
\fst_{p, x}^\loc (\ov{x}) = c_x t^{n(p-1)/2} \ov{x}^{(p)} \in HF_\zp^\loc (H^{(p)}, x^{(p)}),\ 0 \neq c_x \in \fp.
\eeq

We prove Theorem \ref{thm: hyperbolic} by contradiction. Suppose
\beqn
\qst_p(u) = St_p(u) = c t^{r(p-1)/2} u
\eeqn
for some $u \in H^r(M; \F_p)$ and a nonzero $c \in \fp$ with $r > n$. Consider the Floer complex $CF(H)$ with coefficients in $\Lambda^{\Gamma}$. It can be translated via Theorem \ref{thm214} to 
\beq\label{eqn32}
\fst_p( \pss(u)) = ct^{r(p-1)/2} \ssp_{\zp}^{-1}(u) \in HF_\zp(M).
\eeq

Now we localize the above equation to fixed points. As the Conley--Zehnder indices of all contractible fixed points have the same parity, the 
differential of $CF(H)$ vanishes. 
Hence for each $x \in {\mc O}(H)$, the class $\pss(u)\in HF^r (M)$ has well-defined local components in $HF^{\loc, r} (H, x; \Lambda^\Gamma)$; when nonzero, this component can be written as $a_x \ov{x}$ for some capping $\ov{x}$ and $a_x \in \Lambda^\Gamma$ with $\nu(a_x) = 0$. More precisely, one has
\beqn
\pss (u) = \sum_{x \in {\mc O}(H)} a_x \ov{x}.
\eeqn

Then by \eqref{eqn31} and \eqref{eqn32}, one has the following equality in $HF_\zp(M)$
\begin{multline}\label{eqn34}
 ct^{r(p-1)/2} \ssp_{\zp}^{-1} (u) = \fst_p( \pss(u))  \\
= P^\loc(\pss(u)) + P^{\rm big} (\pss(u)) = \sum_x \fst_{p, x}^\loc ( a_x \ov{x}) + P^{\rm big} (\pss(u)) \\
= \sum_{x} c_x a_x^p t^{n(p-1)/2} \ov{x}^{(p)} + P^{\rm big} (\pss(u)).
\end{multline}

Now by Theorem \ref{thm211}, $HF_{\zp}(M)\cong H(M) \otimes \Lambda_{\rp}^\Gamma$. Hence one can consider the leading order term with respect to the $(t, \theta)$-grading on both sides of \eqref{eqn34}. The $(t, \theta)$-degree of the leading order term of the left-hand-side is clearly at least $r(p-1)$. 

To derive a contradiction, consider a ``spectral carrier'' of $u$, i.e. some $\ov{x} \in \tilde{\mc O}(H)$ such that
\beqn
c(u, H) = {\mc A}_H ( a_x \ov{x}) = {\mc A}_H(\ov{x}).
\eeqn
Since the local equivariant pants product $\fst_{p, x}^\loc$ rescales the action by $p$ times while $P^{\rm big}$ strictly increases the action, the Tate spectral invariant of the right hand side of the above equation is realized by the term $c_x a_x^p t^{n(p-1)/2} \ov{x}^{(p)}$. Therefore, the $(t, \theta)$-degree of the leading order term of the right-hand-side of \eqref{eqn34} is at most $n(p-1)$. This contradicts the condition that $r>n$. This proves Theorem \ref{thm: hyperbolic}.


\begin{proof}[Proof of Corollary \ref{cor: generic}]
A standard argument (see \cite{GG-generic,Sugimoto-generic}), which is a local argument, shows that for a $C^2$-generic Hamiltonian diffeomorphism either all of its periodic points are hyperbolic or it has infinitely many periodic points by a variant of the Birkhoff-Lewis theorem \cite{Moser}. However, if all the periodic points of $\phi$ are hyperbolic and there are finitely many of them, then an iteration $\psi = \phi^{k_0}$ as above must be a pseudo-rotation with all periodic points hyperbolic. Theorem \ref{thm: hyperbolic} now applies and finishes the proof.    
\end{proof}


\begin{proof}[Proof of Corollary \ref{cor: generic new example}]\label{proof_new_example}\label{proof_cor19}

We work over characteristic $2$ and use quantum Steenrod squares in the monotone setting. Let $\mu \in H^4 (X; {\mb F}_2)$ be the point class. By \cite[Theorem 1.7]{covariant-constant}, one has 
\beqn
\qst_2(\mu) = \theta^4 \mu
\eeqn
where $\theta \in H^1(B {\mb Z}/2; {\mb F}_2)$ is the generator. On the other hand, let 
\beqn
\kappa: H^*(X) \otimes H^*(X) \otimes H^*(\mb{CP}^2) \to H^* (M)
\eeqn
be the usual Kunneth isomorphism. Consider 
\beqn
u = \kappa(\mu \otimes \mu \otimes 1) \in  H^8 (M).
\eeqn
Then by the Kunneth property of the quantum Steenrod squares (Theorem \ref{thm: Kunneth}) 
we have 
\begin{multline*}
\qst_2(u) = (\kappa\otimes m_p) ( \qst_2(\mu) \otimes \qst_2(\mu) \otimes QSt_2(1)) \\
= (\kappa \otimes m_p) ( \theta^4 \mu \otimes \theta^4 \mu \otimes 1) = \kappa( \mu \otimes \mu \otimes 1) \otimes \theta^8.
\end{multline*}
On the other hand, the classical Steenrod square satisfies the same Kunneth property. Hence
\beqn
St_2(u) = \theta^8 u = \qst_2(u).
\eeqn


Now we prove that a $C^2$-generic Hamiltonian diffeomorphism $\phi$ on $M$ has infinitely many periodic points. Recall that a $C^2$-generic $\phi$ either has infinitely many periodic points or all periodic points are hyperbolic. If the claim is not true, then $\phi$ has finitely many periodic points all of which are hyperbolic. Then the calculation $\qst_2(u) = St_2(u)$ and the fact $r = 8 > 6 = n$ contradict Theorem \ref{thm: hyperbolic}, see Remark \ref{rem:even-prime}.
\end{proof}

\subsection{Proof of Theorem \ref{thm: infinitely many}}\label{subsection_proof_anti_HZ}

We first prove a variant of Lemma \ref{lma:sdm}.

\begin{lemma}\label{lemma47}
Under the assumption of Theorem \ref{thm: infinitely many}, assume $\phi$ has only finitely many periodic points. Then no contractible fixed point of $\phi$ is a symplectically degenerate maximum.
\end{lemma}

\begin{proof}
The argument is very similar to that of Lemma \ref{lma:sdm}. 
First, as the local Floer cohomology is defined over integers, 
by the universal coefficient theorem, 
there exists $p_0(\phi)$ such that for all primes $p \geq p_0(\phi)$ and contractible fixed points $x \in {\rm Fix}_c(\phi)$, there holds 
\beqn
HF^{\rm loc}(\phi, x; \fp) \cong HF^{\rm loc}(\phi, x; {\mb Z}) \otimes \fp.
\eeqn
Hence for $p \geq p_0(\phi)$, one has
\beqn
N(\phi;\F_p) = N(\phi;\Q).
\eeqn
Further, by \cite[Theorem 1.1]{GG-local-gap}, for $p$ sufficiently large and all contractible fixed points $x \in {\rm Fix}_c(\phi)$, one has
\beqn
\dim HF^{\loc} (\phi^p, x; \fp) = \dim HF^{\loc} (\phi,x; \fp).
\eeqn
As $\phi$ has only finitely many periodic points, when $p$ is sufficiently large, one has ${\rm Fix}_c(\phi) = {\rm Fix}_c(\phi^p )$. Hence 
\begin{multline}
N(\phi^p; \fp) = \sum_{y \in {\rm Fix}_c(\phi^p)} {\rm dim} HF^{\rm loc}(\phi^p, y; \fp) = \sum_{x \in {\rm Fix}_c(\phi)} {\rm dim} HF^{\rm loc}(\phi, x; \fp) \\
= N(\phi; \fp) = N(\phi; {\mb Q}) = {\rm dim} H^*(M; {\mb Q}) = {\rm dim} H^*(M; \fp).
\end{multline}
If $H$ is a 1-periodic Hamiltonian whose time-1 map is $\phi$, the above identity implies that the differential of the complex $C(H^{(p)})$ constructed in Proposition \ref{prop31} vanishes. Replacing $p^k$ in the proof of Lemma \ref{lma:sdm} by a sufficiently large $p$, one can obtain the same conclusion. 
\end{proof}


We then prove a variant of Proposition \ref{prop_carrier_iteration}. The proof idea is the same but certain details need to be modified.

\begin{lemma}\label{lemma48}
Under the assumption of Theorem \ref{thm: infinitely many}, suppose $\phi$ has only finitely many periodic points, then there exist $\ov{x} \in \tilde {\mc O}(H)$ and infinitely many primes $p$ such that
\begin{align*}
&\ HF^{\loc, 2n} (H, \ov{x}; \fp) \neq 0,\ &\ HF^{\loc, 2n} (H^{(p)}, \ov{x}^{(p)}; \fp) \neq 0.
\end{align*}
\end{lemma}

\begin{proof}
We can run the same argument as proving Proposition \ref{prop_carrier_iteration}. Notice that one no longer has an induction from $p^k$ to $p^{k+1}$. First, since $\phi$ has only finitely many periodic points, by considering large primes $p$, one has 
\beqn
{\rm Fix}_c(\phi) = {\rm Fix}_c(\phi^p).
\eeqn
Then, the equation $\qst_p(\mu) = {\it St}_p(\mu)$ can be translated, via the ordinary PSS map and the equivariant SSP map, to 
\beqn
\fst_p( \pss(\mu)) = t^{n(p-1)} \ssp_{\zp}^{-1}(\mu) \in HF_\zp(M).
\eeqn
Let $x_1, \ldots, x_m$ be all contractible fixed points of $\phi$. Then by Proposition \ref{prop31}, one can decompose $\pss(\mu)$ as 
\beqn
\pss(\mu) = \sum_{i=1}^m \mu_i,\ {\rm where}\ \mu_i \in HF^\loc(H, x_i; \Lambda^\Gamma)
\eeqn
and
\beqn
\pss(\mu) = \sum_{i=1}^m \mu_i^{(p)},\ {\rm where}\ \mu_i^{(p)} \in HF^\loc(H^{(p)}, x_i^{(p)}; \Lambda^\Gamma).
\eeqn
For $p$ sufficiently large, one can also assume that each $x_i$ is $p$-admissible. Using Proposition \ref{prop33b} and Proposition \ref{prop_local_eq_structure}, one can decompose $\mu_{\zp}$ as 
\beqn
\mu_\zp = \sum_{i=1}^m \mu_{\zp, i}^{(p)}(t, \theta),\ {\rm where}\ \mu_{\zp, i}^{(p)} \in HF^\loc(H^{(p)}, x_i^{(p)}; \Lambda_\rp^\Gamma).
\eeqn

Then one has the analogue of Claim A in the proof of Proposition \ref{prop_carrier_iteration}, i.e., 

{\it Claim A.} For infinitely many primes $p$, the leading order term $\mu_{\zp, i}^{(p)}(0,0)$ coincides with $\mu_i^{(p)}$ up to higher energy terms.

The proof is similar. Then as a consequence, for infinitely many primes $p$, there holds
\beqn
c(\mu, H^{(p)}) \geq \wh c(\mu_{\zp}, H^{(p)}) = p c(\mu, H).
\eeqn
This is the analogue of Claim B in the proof of Proposition \ref{prop_carrier_iteration}. Moreover, there exists $i$ such that ${\mc A}_H(\mu_i) = c(\mu, H)$ (i.e., $x_i$ is a spectral carrier). Then for such $i$, one has 
\beqn
{\mc A}_{H^{(p)}}(\mu_i^{(p)}) \geq c(\mu, H^{(p)}) \geq p c(\mu, H) = p {\mc A}_H(\mu_i).
\eeqn

Then one can prove that for each $p$, there exists $i_p$ such that 
\beqn
{\mc A}_{H^{(p)}}(\mu_{i_p}^{(p)}) = p {\mc A}_H(\mu_{i_p}) = pc(\mu, H),
\eeqn
using the same argument. As there are infinitely many such $p$, by the pigeon-hole principle, one can choose $i_0$ such that for infinitely many $p$, $i_0 = i_p$. 
\end{proof}

Now we can prove Theorem \ref{thm: infinitely many}. Assume that $\phi$ has only finitely many periodic points. By Lemma \ref{lemma47}, no contractible fixed point of $\phi$ is a symplectically degenerate maximum. Therefore, there exists $\epsilon>0$ such that for sufficiently large prime $p$,
\beqn
HF^{\loc, 2n}(H, \ov{x}; \fp) \neq 0 \Longrightarrow \Delta( \ov{x}, H) \geq \epsilon.
\eeqn
Therefore 
\beqn
HF^{\loc, 2n}(H, \ov{x}; \fp) \neq 0 \Longrightarrow \Delta( \ov{x}^{(p)}, H^{(p)}) \geq p\epsilon.
\eeqn
As the local Floer cohomology $HF^{\rm loc}(H^{(p)}, \ov{x}^{(p)}; \fp)$ is supported over degrees between $-\Delta (\ov{x}^{(p)}, H^{(p)})$ and $-\Delta(\ov{x}^{(p)}, H^{(p)}) + 2n$, it follows that for sufficiently large prime $p$, 
\beqn
HF^{\loc, 2n}(H, \ov{x}; \fp) \neq 0 \Longrightarrow HF^{\loc, 2n}(H^{(p)}, \ov{x}^{(p)}; \fp) = 0.
\eeqn
This contradicts Lemma \ref{lemma48}. Therefore, $\phi$ has infinitely many periodic points.


\subsection{Proof of Theorem \ref{thm_noncontractible}}

In \cite{Bai_Xu_2025} the first and the fourth authors established the foundation of Hamiltonian Floer theory for homologically nontrivial periodic orbits. We briefly recall the setup.

Choose a possibly noncontractible loop $\gamma: S^1 \to X$ which represents a nonzero class
\beqn
[\gamma] \neq 0 \in H_1(X; {\mb Z})_{\rm free}.
\eeqn
One can consider noncontractible periodic orbits of a Hamiltonian $H$ which are oriented cobordant to $\gamma$.\footnote{Note that this is equivalent to being homologous over ${\mb Z}$.} Given a 1-periodic orbit $x: S^1 \to X$ of $H$, a {\bf $\gamma$-capping} of $x$ is an oriented cobordism $u: \Sigma \to X$ from $\gamma$ to $x$ (where $\Sigma$ could be a higher genus surface); two $\gamma$-cappings are equivalent if their difference, which is a class in $H_2(X; {\mb Z})$, has zero symplectic area. In this setup we do not keep the ${\mb Z}$-grading but only consider the ${\mb Z}_2$-grading. A {\bf $\gamma$-capped orbit} of $H$ is a 1-periodic orbit together with an equivalence class of $\gamma$-cappings. Let 
\beqn
\tilde {\mc O}_\gamma(H)
\eeqn
denote the set of $\gamma$-capped orbits of $H$, on which there is a free action by the group
\beqn
\Pi:= H_2(X; {\mb Z})/ {\rm ker} (\omega).
\eeqn
Then each $\gamma$-capped orbit $p$ has a well-defined symplectic action ${\mc A}_{H,\gamma} (p)\in {\mb R}$ and a well-defined ${\mb Z}/2$-grading. 

\begin{thm}\cite[Theorem 1.10]{Bai_Xu_2025}
Let $(X, \omega)$ be a compact symplectic manifold and $\gamma:S^1 \to X$ be a smooth loop with $[\gamma] \neq 0 \in H_1(X; {\mb Z})_{\rm free}$. Let $H$ be a nondegenerate $1$-periodic Hamiltonian. Then for each $1$-periodic family of $\omega$-compatible almost complex structures $J$, there is a differential on the ${\mb Z}/2$-graded abelian group
\beqn
CF_\gamma ( H):= \Big\{ \sum_{i=1}^\infty a_i p_i\ |\ a_i \in {\mb Z},\ p_i \in \tilde {\mc O}_\gamma(H),\ \lim_{i \to \infty} {\mc A}_{H,\gamma} (p_i) = -\infty \Big\}
\eeqn
leading to a ${\mb Z}/2$-graded  complex of abelian groups
\beqn
CF_\gamma(H, J, \Xi).
\eeqn
Moreover, for each field ${\mb K}$, the tensor product
\beqn
CF_\gamma(H) \otimes \Lambda_{\mb K}^\Gamma
\eeqn
is a Floer-type complex over $\Lambda_{\mb K}^\Gamma$ whose filtered isomorphism class only depends on $H$.
\end{thm}

Under the assumption that $[\gamma] \neq 0 \in H_1(X)$, using a continuation map argument we can prove that the resulting cohomology vanishes. On the other hand, one considers the quantitative theory. For any prime $p$, similar to the previous contractible case, one has an equivariant Floer complex
\beqn
CF_{\zp, \gamma}(H^{(p)})
\eeqn
over $\Lambda_{\kzp}^\Gamma$ which is freely generated by $p$-periodic orbits in homology class $p\gamma$. Its filtered isomorphism class only depends on $H$. 

\begin{thm}\cite[Theorem 1.11]{Bai_Xu_2025}
There exists a filtered chain map
\beqn
\wt{\fst}_{p,\gamma}: C_{\zp}(CF_\gamma(H)^{\otimes p}) \to CF_{\zp,\gamma^{(p)}}(H^{(p)})
\eeqn
which extends to a filtered homotopy equivalence for the Tate versions
\beqn
\wh{\fst}_{p,\gamma}: \wh{C}_{\zp}(CF_\gamma(H)^{\otimes p}) \to \wh{CF}_{\zp, \gamma}(H^{(p)}).
\eeqn
\end{thm}

Let $\wh\beta_{i,p\gamma}(H^{(p)})$ be the finite bar-lengths of the Tate complex and let $\beta_{i,\gamma}(H)$ be the finite bar-lengths of the complex $CF_\gamma(H)$. Then as a consequence, for $p$ odd, there holds
\beqn
\wh\beta_{2i-1, p\gamma}(H^{(p)}) = \wh\beta_{2i, p\gamma}(H^{(p)}) = p \beta_{i,\gamma}(H),\ i = 1, \ldots, N.
\eeqn
By the inequality between the minimal bar-lengths of the Tate and the regular complexes for $H^{(p)}$ in class $p\gamma$, there holds
\beq\label{eqn:minimal-bar-scale}
\inf \beta_{i,p\gamma}(H^{(p)}) \geq p \inf \beta_{i,\gamma}(H).
\eeq

Another ingredient in Sugimoto's argument is the persistence of local Floer homology. We provide a proof for the convenience of the reader.

\begin{lemma}[Persistence of local Floer cohomology: noncontractible case]\label{lemma411}
Let $\ov{x}$ be an isolated $\gamma$-capped 1-periodic orbit of $H$. Let $k \geq 1$. Suppose $\ov{x}$ is $k$-admissible. Then 
\beqn
HF^\loc_\gamma (H, \ov{x}; {\mb K}) \neq 0 \Longrightarrow HF^\loc_{k\gamma} (H^{(k)}, \ov{x}^{(k)}; {\mb K}) \neq 0.
\eeqn
\end{lemma}

\begin{proof}
Let $\phi$ be the time-1 map of $H$. Then $x = x(0)$ is an isolated fixed point of $\phi$. Notice that the local Floer cohomology $HF_\gamma^{\rm loc}(H, \ov{x}; {\mb K})$ only depends on the germ of the symplectomorphism $\phi$ near $x$. Moreover, it is well-known that the germ of a symplectomorphism near a fixed point is Hamiltonian (see \cite[Remark 4.1]{GG-local-gap}). Then one can reduce the consideration to a local problem where the fixed point is contractible. Then the lemma follows from the contractible case \cite[Theorem 1.1]{GG-local-gap}.
\end{proof}

Now we can prove Theorem \ref{thm_noncontractible}. The remaining argument does not involve equivariant Floer theory.

\begin{proof}[Proof of Theorem \ref{thm_noncontractible}]
We adapt Sugimoto's argument in \cite{Sugimoto} to the current, more general situation. Let $H$ be a Hamiltonian as in the statement. Suppose, for a sufficiently large $p$, all 1-periodic orbits of $H^{(p)}$ in class $p \gamma$ are iterations of the generators of $CF_\gamma(H)$. We prove that for the next prime $p'>p$, the Hamiltonian $H^{(p')}$ has a 1-periodic orbit in homology class $p\gamma$. Since any such orbit must be simple (for $p$ being sufficiently large relative to $[\gamma]$), this would prove Theorem \ref{thm_noncontractible}.

Now by our assumption, there exists a $\gamma$-capped 1-periodic orbit $x$ of $H$ in class $\gamma$ with $HF_\gamma^{\rm loc}(H, x) \neq 0$. Let $\ov{x}^{(p)}$ be the obvious $\gamma^{(p)}$-capped 1-periodic orbit of $H^{(p)}$. Then for $p$ sufficiently large, by Lemma \ref{lemma411}, there holds
\beqn
HF_{p\gamma}^{\rm loc}(H^{(p)}, \ov{x}^{(p)}) \neq 0.
\eeqn
By adding a constant $C_p$ to the Hamiltonian $H^{(p)}$, one may assume that 
\beqn
{\mc A}_{H^{(p)}, \gamma^{(p)}} (\ov{x}^{(p)} ) = 0.
\eeqn

{\it Claim. Let $\inf \beta_{i,\gamma}(H) \geq 2\epsilon >0$. Then there exists a (small) constant $\delta>0$ and a nondegenerate perturbation $H_p$ of $H^{(p)}$ such that the natural map 
\beqn
HF (H_p, p \gamma)^{[-\delta, p \epsilon]} \to HF(H_p, p\gamma)^{[-p\epsilon, \delta]}
\eeqn
is nonzero.}

Indeed, using the inequality \eqref{eqn:minimal-bar-scale}, the argument of \cite[Proof of Theorem 2, Section 5]{Sugimoto} directly carries over: we have $\inf \beta_{i,p\gamma}(H^{(p)}) \geq 2p \epsilon$. If $\delta > 0$ is sufficiently small, we see that $\ov{x}^{(p)}$ induces a well-defined cohomology class on both sides such that the version in the latter is the image of the version in the former.

Now consider the Hamiltonian $H^{(p')}$ and a small nondegenerate perturbation $H_{p'}$. Notice that one can choose $H_{p'}$ such that the pointwise difference between $H_{p'}$ and $H_p$ is close to $(p'-p) \|H\|$ up to an error at most $2\delta$. Hence by choosing homotopies between $H_p$ and $H_{p'}$, one can build continuation maps
\begin{align*}
&\ \Phi: CF (H_p, p\gamma) \to CF(H_{p'}, p \gamma),\ &\ \Psi: CF (H_{p'}, p \gamma) \to CF (H_p, p\gamma)
\end{align*}
such that they naturally induce
\beqn
\Phi(CF(H_p, p\gamma)^{[a, b]}) \subseteq CF(H_{p'}, p\gamma)^{[a - (p'-p) \| H \| - 2\delta, b - (p'-p)\|H\| - 2\delta ]} 
\eeqn
and 
\beqn
\Psi(CF(H_{p'}, p\gamma)^{[c, d]}) \subseteq  CF(H_p, p\gamma)^{[c - (p'-p) \|H \| - 2\delta, d - (p'-p)\|H\| - 2\delta]}.
\eeqn
However, as $p'-p = o(p)$ as $p\to +\infty$, one has $(p'-p)\| H \| \ll p\epsilon$. Hence one has the following commutative diagram. 
\beqn
\xymatrix{   
HF (H_p, p \gamma)^{ [ - \delta,  p \epsilon]} \ar[rr] \ar[rd] & & HF(H_p, p \gamma)^{[- p \epsilon,  \delta]} \\
& HF(H_{p'}, p \gamma)^{[ - 3\delta - (p' -p)\| H \|,  p \epsilon - (p' - p) \| H \| - 2\delta  ]}  \ar[ru] & 
 }
\eeqn
The above claim implies that the bottom item in the diagram is nonzero. It follows that $H_{p'}$ has some 1-periodic orbit in class $p\gamma$. As $H_{p'}$ can be arbitrarily close to $H^{(p')}$, it follows that $H^{(p')}$ has a 1-periodic orbit in class $p\gamma$.
\end{proof}

\subsection{Proof of Theorem \ref{thm: sdm}}\label{proof_thm_sdm}

Finally, we prove the following result, which generalizes \cite[Theorem 1.18]{GG-ai} to all closed symplectic manifolds.

\begin{lemma}
Under the assumption of Theorem \ref{thm: sdm}, let $H$ be a Hamiltonian generating $\phi$ which has only finitely many periodic points. Then for each odd prime $p$, there exists $k_0>0$ such that the differential $d_{H^{(p^{k_0})}}$ on the complex $C(H^{(p^{k_0} )} )$ given by Proposition \ref{prop31} is nontrivial. 
\end{lemma}

\begin{proof}
Fix an odd prime $p$. Suppose this is not the case. Then for all $k \geq 1$, the differential $d_{H^{(p^k)}}$ vanishes. Recall that by Proposition \ref{prop31}, the cohomology of $C(H^{(p^k)})$ is isomorphic to the Floer cohomology of $(M, \omega)$. Therefore one has 
\beqn
N(\phi^{p^k}; \fp) = \sum_{y \in {\rm Fix}_c(\phi^{p^k})} {\rm dim}_{\fp} HF^{\rm loc}(\phi^{(p^k)}, y; \fp) = {\rm dim}_{\Lambda^\Gamma} HF(M, \omega; \Lambda^\Gamma).
\eeqn
Then by Definition \ref{def: pr}, $\phi$ is an $\fp$ Hamiltonian pseudo-rotation. 
This contradicts Lemma \ref{lma:sdm}.
\end{proof}

Now we turn to the proof of Theorem \ref{thm: sdm}. Let $H$ be a Hamiltonian as in this theorem and assume its time-1 map $\phi$ has only finitely many periodic points. As a consequence of the previous lemma, by replacing $H$ by $H^{(p^{k_0})}$, we may assume that $d_H$ is nontrivial. Hence the barcode of $C(H)$ contains a finite bar with minimal bar-length $\beta_1(H) > 0$. Then by Corollary \ref{cor35} and Lemma \ref{lemma37}, one has
\beqn
\beta_1 (H^{(p)}) \geq \widehat \beta_1 (H^{(p)}) = p \beta_1(H).
\eeqn
Therefore, there exists $k_1$ such that $\beta_1 ( H^{(p^k)} ) \geq 1$ for $k \geq k_1.$ Then for $\eps < 1/3$ and $k \geq \max\{k_1, \log_p(r_{\eps})\},$ the connecting homomorphism \eqref{eq: connecting} must vanish. Indeed, if the connecting homomorphism \eqref{eq: connecting} were non-zero, then there would be a bar in the barcode of $H^{(p^k)}$ which starts in the interval $(rc-\delta,rc+\delta)$ and ends in the interval $(rc+\delta, rc+\eps).$ In particular, its length would be smaller than $2\delta+\eps< 1.$ This creates a contradiction to results of \cite{Ginzburg-CC,Hein-CCCY, GG-ai} as in the proof of Lemma \ref{lma:sdm}.

\section{Kunneth property of the quantum Steenrod operations}\label{section5}
\label{sec:equivariant-kunneth}

In this section we prove that the quantum Steenrod operation respects the ordinary and the equivariant Kunneth maps in the monotone setting. We also prove an analogous version in the Floer-theoretic setting. Throughout this section, all homology is taken with $\F_p$-coefficients unless otherwise stated. 

\subsection{Equivariant Kunneth maps via Morse theory}\label{subsec:equiv-kunneth}

We first need to describe the equivariant Kunneth map on the chain level, using the Morse model. For $i = 1, 2$, let $X_i$ be a smooth $\zp$-manifold and $f_i: X_i \to {\mb R}$ be a $\zp$-invariant Morse function. In Section \ref{subsubsec:borel-construction} we described the construction of the equivariant Morse cochain complex. 
Now we describe how to define the equivariant Kunneth map 
\beqn
\kappa_{\zp}: HM_{\zp}(f_1) \otimes HM_{\zp}(f_2) \to HM_{\zp}(f_1\times f_2)
\eeqn
on the chain level. 

Recall that we used a specific $\zp$-invariant function $g: S^\infty \to {\mb R}$, see \eqref{equation:morse-function-s-inf}, together with a $\zp$-invariant metric to give a cochain model for $B\zp$. Then $CM_{\zp}(f_i)$ is defined by considering the Morse complex associated to a $\zp$-invariant perturbation $f_{i, v}$ of $f_i$. We also choose another family 
\beqn
f_{\infty, v}: X_1 \times X_2 \to {\mb R},\ v \in S^\infty
\eeqn
for which the equivariant Morse complex $CM_{\zp}(f_1\times f_2)$ can be defined.

Recall $T_{2, 1}$ is the tree with two incoming edges identified with $(-\infty, 0]$ and one outgoing edge identified with $[0, +\infty)$. Consider a piecewise smooth family
\beqn
{\ms F}: T_{2,1} \times S^\infty \times X_1 \times X_2 \to {\mb R}
\eeqn
written as ${\ms F}_{s, v}(x_1, x_2)$ satisfying
\begin{enumerate}
    \item ${\ms F}$ is $\zp$-invariant, namely, for each $\gamma \in \zp$,
    \beqn
    {\ms F}_{s, \gamma v}(\gamma x_1, \gamma x_2) = {\ms F}_{s, v}(x_1, x_2).
    \eeqn

    \item When $s\in T_{2, 1}$ is close to $0$, the function ${\ms F}_{s, v}(x_1, x_2)$ is independent of $s$ and $v$.
    
    \item For $i = 1, 2$, when $s \in T_i$ and $s\ll 0$, ${\ms F}_{s, v}(x_1, x_2) = f_{i, v}(x_i)$.
        
    \item When $s \in T_\infty$ and $s\gg 0$, ${\ms F}_{s, v}(x_1, x_2) = f_{\infty, v}(x_1, x_2)$.
\end{enumerate}
Then for any piecewise smooth map 
\beqn
\zeta = (\zeta_1, \zeta_2, \zeta_\infty): T_{2,1} \to S^\infty
\eeqn
there is an associated family
\beqn
{\ms F}_{\zeta, s}: X_1 \times X_2 \to {\mb R},\ s\in {\mb R}
\eeqn
given by 
\beqn
{\ms F}_{\zeta, s} (x_1, x_2) = \left\{ \begin{array}{cc} {\ms F}_{s, \zeta_\infty (s)}(x_1, x_2),\ &\ s \geq 0,\\
{\ms F}_{s, \zeta_1(s)}(x_1, x_2) + {\ms F}_{s, \zeta_2(s)}(x_1, x_2),\  & s \leq 0. \end{array} \right.
\eeqn
Note that ${\ms F}_{\zeta, s}$ is smooth in $s$ because near $s = 0$, ${\ms F}_{s, v}$ is independent of $s$ and $v$.

We also fix the family of vector fields $Y_s$ on $S^\infty$ parametrized by $s \in T_{2,1}$ (which was used to define the multiplicative structure on $HM_{\zp}(g)$). Then for critical points $x_{i, \pm} \in {\rm crit}(f_i)$, $w_1, w_2, w_\infty \in {\rm crit}(g)$, consider the moduli space
\beqn
{\mc M}_{w_1, w_2, w_\infty}(x_{1, -}, x_{2, -}; x_{1, +}, x_{2, +})
\eeqn
of pairs $(\eta, \zeta)$ satisfying 
\begin{enumerate}

\item $\zeta: T_{2,1}\to S^\infty$ is a $Y$-perturbed flow tree with $\zeta_i(\infty) = w_i$ for $i = 1, 2, \infty$.

\item $\eta: {\mb R} \to X_1 \times X_2$ is a solution to the equation
\beqn
\eta'(s) = \nabla {\ms F}_{\zeta, s}(\eta(s))
\eeqn
such that $\eta$ converges to $(x_{1, -}, x_{2, -})$ at $-\infty$ and to $(x_{1, +}, x_{2, +})$ at $+\infty$.
\end{enumerate}
The union of moduli spaces over all possible asymptotic limits has a free $\zp$-action. Hence by choosing ${\ms F}_{s, v}$ generically, one can achieve equivariant transversality. The moduli space is shown in Figure \ref{Figure_Kunneth_pfold}.

\begin{figure}[ht]
    \centering
    \includegraphics[width=0.5\linewidth]{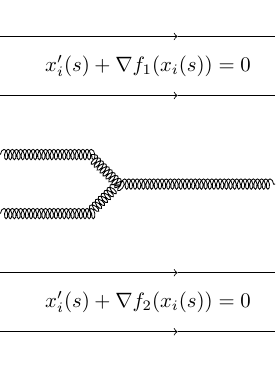}
    \caption{Moduli space defining the equivariant Kunneth map for the tensor products. The coiled tree is a flow tree in $S^\infty$.}
    \label{Figure_Kunneth_pfold}
\end{figure}

On the other hand, the moduli space ${\mc M}_{w_1, w_2, w_\infty}(x_{1,-}, x_{2, -}; x_{1, +}, x_{2, +})$ can be compactified by suitably adding broken configurations. More precisely, the breaking along the leg $T_1$ resp. $T_2$ is regarded as a trajectory in $X_1\times S^\infty$ resp. $X_2 \times S^\infty$, while the breaking along the leg $T_\infty$ is regarded as a map into $X_1\times X_2 \times S^\infty$. Then the counts of the moduli spaces ${\mc M}_{w_1, w_2, w_\infty}(x_{1, -}, x_{2, -}; x_{1, +}, x_{2, +})$ define a chain map 
\beq\label{Morse_Kunneth}
\kappa_{\zp}: CM_{\zp}(f_1) \otimes_{\fp} CM_{\zp}(f_2) \to CM_{\zp}(f_1\times f_2)
\eeq
which descends to the equivariant Kunneth map. One can also prove that on the cohomology level, the equivariant Kunneth map is $\rp$-bilinear, where $HM_{\zp}(f_1)$, $HM_{\zp}(f_2)$, and $HM_{\zp}(f_1\times f_2)$ are $\rp$-modules. We only summarize the results and omit the details.

\begin{prop}
The chain map \eqref{Morse_Kunneth} descends to an $\rp$-bilinear map
\beqn
\kappa_{\zp}: HM_{\zp}(f_1) \otimes_{\rp} HM_{\zp}(f_2) \to HM_{\zp}(f_1 \times f_2).
\eeqn
\end{prop}

\begin{remark}\label{remark52}
When one can achieve equivariant transversality on $X_i$, i.e., one can choose a $\zp$-invariant Morse--Smale pair $(f_i, h_i)$ on $X_i$, then one can also achieve equivariant transversality on $X_1 \times X_2$. In particular, one can take $f_{\infty, v} = f_1 \times f_2$ and choose the product metric. Then one can also choose ${\ms F}$ to be basically $f_1 \times f_2$. Therefore, the equations defining the moduli spaces are decoupled. 
\end{remark}

\subsubsection{Two special cases}

There are two special cases (in the finite-dimensional setting) considered in this paper. The first case is when $\zp$ acts trivially on $X_1$ and $X_2$. In this case, given any Morse functions $f_i: X_i \to {\mb R}$ with an associated Morse complex $CM(f_i)$, as $\rp$-modules one has 
\begin{align*}
&\ HM_{\zp}(f_i) \cong HM(f_i) \otimes \rp,\ &\ HM_{\zp}(f_1 \times f_2) \cong HM(f_1\times f_2) \otimes \rp.
\end{align*}
The equivariant Kunneth map reduces to the ordinary one. Namely,
\beqn
\kappa_{\zp} = \kappa \otimes m_p
\eeqn
where $\kappa: HM(f_1) \otimes HM(f_2) \to HM(f_1 \times f_2)$ is the ordinary Kunneth isomorphism and $m_p: \rp \otimes \rp \to \rp$ is the ring multiplication.

Another special case is the $p$-fold tensor power. Suppose $\zp$ acts trivially on $X_i$, $i = 1, 2$. Let $f_i: X_i\to {\mb R}$ be Morse functions. We have described the Morse model for the equivariant cochain complex $CM(f_i)^{\otimes p}_{\zp}$, which is the $\zp$-invariant part of the product Morse chain complex $CM(f_i^{\times p} \times g)$. Then the cochain-level equivariant Kunneth map is a special case of the general situation.

\subsubsection{Kunneth property of quasi-Frobenius map}

Now consider two Morse cochain complexes $CM(f_1)$, $CM(f_2)$ on $X_1$ and $X_2$ respectively. 

\begin{lemma}\label{Lemma_Kunneth_qF}
The following diagram commutes up to Koszul signs.
\beqn
\xymatrix{ HM(f_1) \otimes HM (f_2) \ar[rr]^-{qF \otimes qF} \ar[d]_{\kappa}    &   & H(CM(f_1)^{\otimes p}_{\zp}) \otimes H (CM(f_2)^{\otimes p}_{\zp})\ar[d]^{\kappa_{\zp}}  \\
         HM(f_1 \times f_2) \ar[rr]_-{qF}   &  &  H( CM (f_1\times f_2)^{\otimes p}_{\zp})   }
\eeqn
\end{lemma}

\begin{proof}[Sketch of proof]
Choose classes $a_i \in HM(f_i)$ and choose representing cocycles $x_i \in CM(f_i)$. Recall that $qF(a_i) \in H(CM(f_i)^{\otimes p}_{\zp})$ is represented by the equivariant cocycle
\beqn
x_i^{\otimes p} \otimes 1 \in CM(f_i) \otimes \cdots \otimes CM(f_i) \otimes \rp.
\eeqn
Using the Morse model on $S^\infty$, let $w_1, \ldots, w_p \in {\rm crit}(g) \subset  S^\infty$ be the $p$ critical points of $g$ of index $0$. Then the above cocycle is equivalent to the Morse cocycle
\beqn
x_i^{\otimes p} \otimes ( w_1 + \cdots + w_p) \in \left( CM(f_i)^{\otimes p} \otimes CM(g) \right)^\inv.
\eeqn
We claim that the equivariant Kunneth map
\beqn
\wt \kappa_\zp:\Big( CM(f_1)^{\otimes p} \otimes CM(g) \Big) \otimes \Big( CM(f_2)^{\otimes p} \otimes CM(g) \Big) \to CM(f_1 \times f_2)^{\otimes p}\otimes CM(g)
\eeqn
gives 
\beqn
\wt\kappa_\zp \Big( (x_1^{\otimes p} \otimes (w_1 + \cdots + w_p) ) \otimes (x_2^{\otimes p} \otimes (w_1 + \cdots + w_p)) \Big) = \pm (\kappa(x_1 \otimes x_2))^{\otimes p} \otimes (w_1 + \cdots + w_p).
\eeqn
Here the sign comes from switching certain copies of $x_1$ and $x_2$. This claim essentially follows from the fact that in this situation equivariant transversality holds. Hence the flow line equation in $X_1 \times X_2$ and the Y-shaped gradient tree equation in $S^\infty$ are decoupled (see Remark \ref{remark52}). This verifies the commutativity (up to signs) of the diagram.
\end{proof}

\subsubsection{Equivariant Kunneth for equivariant local Floer cohomology}\label{subsubsec:local-equiv-kunneth}

In the proof of the divisibility lemma (Lemma \ref{lma:div}), one needs to consider the Kunneth map for local equivariant Floer cohomology. Let $X_i$ be a symplectic manifold, $H_i$ being a Hamiltonian such that $x_i \in {\mc O}(H_i)$ is isolated and $x_i^{(p)} \in {\mc O}(H_i^{(p)})$ is also isolated. We will construct a Kunneth map
\beqn
\kappa_\zp^\loc: HF_\zp^\loc (H_1^{(p)}, x_1^{(p)}; \kzp) \underset{\kzp}{\otimes} HF^{\rm loc}_{\zp}( H_2^{(p)}, x_2^{(p)}; \kzp) \to HF^{\rm loc}_{\zp}( H_1^{(p)} \times H_2^{(p)}, x_1^{(p)}\times x_2^{(p)}; \kzp).
\eeqn

Recall the construction of the local equivariant Floer cohomology given in Section \ref{subsubsec:local-equiv}. Let $G_i$ be nondegenerate on $X_i$ which is sufficiently close to $H_i$. Choose a family 
\beqn
{\ms G}_i^{eq}(v, t, \cdot): X_i \to {\mb R}, v \in S^\infty, t \in S^1
\eeqn
of perturbations of $G_i^{(p)}$ to define the complex $CF_\zp^\loc (H_i^{(p)}, x_i^{(p)})$. On the other hand, $G_1 \times G_2$ is a nondegenerate perturbation of $H_1 \times H_2$. Choose a family 
\beqn
{\ms G}_\infty^{eq}(v, t, \cdot): X_1\times X_2 \to {\mb R}
\eeqn
as a perturbation of $G_1^{(p)}\times G_2^{(p)}$ in order to define the complex $CF_\zp^\loc (H_1^{(p)} \times H_2^{(p)}, x_1^{(p)}\times x_2^{(p)})$. 

To continue, choose the perturbation data for the $Y$-shaped flow trees in $S^\infty$. Namely, a family of vector fields $Y_s$ on $S^\infty$ parametrized by $s\in T_{2, 1}$ which vanishes when $|s|\gg 0$. Then extend ${\ms G}_1^{eq}, {\ms G}_2^{eq}$ and ${\ms G}_{\infty}^{eq}$ to three families: ${\ms G}_{1, s}^{eq}$ on $X_1$ which depends on $s \in T_1$, ${\ms G}_{2, s}^{eq}$ on $X_2$ which depends on $s \in T_2$, and ${\ms G}_{\infty, s}^{eq}$ on $X_1 \times X_2$ which depends on $s \in T_\infty$, such that 1) when $|s|\gg 0$, ${\ms G}_{i, s}^{eq} = {\ms G}_i^{eq}$ and 2) when $s$ is near $0$, ${\ms G}_{1, s} = G_1^{(p)}$, ${\ms G}_{2, s} = G_2^{(p)}$, and ${\ms G}_{\infty, s} = G_1^{(p)}\times G_2^{(p)}$. 

Now consider the following equation for triples $(u_1, u_2, \zeta)$ where $u_1: {\mb R}\times S^1 \to X_1$, $u_2: {\mb R}\times S^1 \to X_2$, $\zeta: T_{2,1} \to S^\infty$ are smooth maps such that 
\begin{enumerate}
    \item $\zeta$ is a $Y$-perturbed flow tree in $S^\infty$.

    \item When $s \leq 0$, for $i = 1, 2$,
    \beqn
    \frac{\partial u_i}{\partial s} + J_i(u_i) \left( \frac{\partial u_i}{\partial t} - X_{{\ms G}_{i, s, \zeta_i(s)}}(u_i) \right) = 0.
    \eeqn

    \item When $s \geq 0$, view $u = (u_1, u_2)$ as a map into $X_1 \times X_2$, for the product almost complex structure $J = J_1 \times J_2$, one has 
    \beqn
    \frac{\partial u}{\partial s} + J (u) \left( \frac{\partial u}{\partial t} - X_{{\mc G}_{\infty, s, \zeta_\infty(s)}}(u) \right) = 0.
    \eeqn
\end{enumerate}
Notice that when $s$ is near $0$, $u_i$ is a solution to the Floer equation for $G_i^{(p)}$. Similar to the finite-dimensional Morse case, one can choose ${\ms G}$ to achieve transversality. Therefore, one obtains a chain level map
\beqn
\kappa_{\zp}^{\rm loc}: CF_\zp^\loc (H_1^{(p)}, x_1^{(p)}) \otimes_{\fp} CF_\zp^\loc (H_2^{(p)}, x_2^{(p)}) \to CF_\zp^\loc (H_1^{(p)}\times H_2^{(p)}, x_1^{(p)} \times x_2^{(p)})
\eeqn
leading to the map on the homology level
\beqn
\kappa_{\zp}^{\rm loc}: HF_{\zp}^{\rm loc}(H_1^{(p)}, x_1^{(p)}) \otimes_{\fp} HF_{\zp}^{\rm loc}(H_2^{(p)}, x_2^{(p)}) \to HF_{\zp}^{\rm loc}(H_1^{(p)} \times H_2^{(p)}, x_1^{(p)} \times x_2^{(p)}).
\eeqn
We summarize our result here.

\begin{prop}\label{prop_local_eq_Kunneth}
The local equivariant Kunneth map $\kappa_\zp^\loc$ on local equivariant Floer cohomology is bilinear over $\rp$.
\end{prop}

\subsection{Quantum Steenrod operation via inhomogeneous term perturbations}\label{subsection52}


The definition of quantum Steenrod operation given in \cite{Bai_Xu_2025} and recalled in Subsection \ref{subsec:quantumStpower} is based on abstract perturbations. In the monotone setting, one can use an alternate approach by geometrically perturbing the equation, as done in \cite{SZhao-pants} and \cite{covariant-constant}. Here we provide such a definition under the monotonicity assumption using the Morse model. Notice that \cite{covariant-constant} uses a particular cellular model for the equivariant cohomology. The comparison between the cellular and the Morse-theoretic approaches is carried out in Appendix \ref{appendixa}.

Let $(M, \omega)$ be a positively monotone closed symplectic manifold. Choose a Morse function $f:M \to {\mb R}$ together with a Morse--Smale metric. Let $g: S^\infty \to {\mb R}$ be the $\zp$-invariant Morse function used throughout this paper. 

We choose the inhomogeneous term as follows. Recall that $\Sigma_p \to S^2$ is the $p$-fold cover branched over $0$ and $\infty$. Consider a smooth family of complex-antilinear maps
\beqn
\nu^{eq}_{v, z, x}: T_z \Sigma_p \to T_x M
\eeqn
parametrized by $(v, z, x) \in S^\infty \times \Sigma_p \times M$ which satisfies the two equivariance conditions:
\begin{enumerate}

\item The $\zp$-equivariance: for the generator $\iota \in \zp$ which acts on both $S^\infty$ and $\Sigma_p$
\begin{equation} \label{eq:nu-equivariance}
\nu^{eq}_{\iota(v),z,x} = \nu^{eq}_{v,\iota(z),x} \circ D \iota_z: T_z \Sigma_p \to  T_x M.
\end{equation}

\item The ${\mb Z}_{\geq 0}$-equivariance: for the generator $\tau\in {\mb Z}_{\geq 0}$ which acts on $S^\infty$
\beqn
\nu_{\tau(v), z, x}^{eq} = \nu_{v, z, x}^{eq}.
\eeqn
\end{enumerate}

We also choose a generic $\omega$-compatible $J$ such that \cite[Assumption 3.2]{covariant-constant} is satisfied, namely, the moduli spaces of certain configurations of simple $J$-holomorphic spheres (chains) are regular and certain evaluation maps on such moduli spaces are transverse to stable and unstable manifolds of the chosen Morse function $f: M \to {\mb R}$. 

Now we describe the moduli spaces; the moduli spaces considered in Section \ref{subsec:quantumStpower} are the special cases for $\nu_{v, z, x}^{eq} \equiv 0$. For $A \in H_2(M; \Z)$, $w_-, w_+ \in {\rm crit}(g)$, and $x_1, \ldots, x_p, y \in {\rm crit}(f)$, let 
\beqn
{\mc M}_{w_-, w_+}(A; x_1, \ldots, x_p, y)
\eeqn
be the set of pairs $(u, \zeta)$ where
\begin{itemize}

\item $\zeta: {\mb R} \to S^\infty$ is a solution to $\zeta'(s) = \nabla g(\zeta(s))$ which converges to $w_\pm$ at $\pm\infty$.

\item $u: \Sigma_p \to M$ is a solution to 
\beq\label{eq:u-equivariant-condition} 
\ov\partial_J u(z) = \nu_{\zeta(0), z,u(z)}^{eq}\,
\eeq

\item $u(\infty) \in W^u(y,f)$,

\item $u(z_l) \in W^s(x_l, f)$ for $l = 1, \ldots, p$.

\item $u_*([\Sigma_p]) = A$.
\end{itemize}

This moduli space has the expected dimension
\begin{equation} \label{equation:morse-moduli-dimension}
|y| + |w_-| + 2 c_1(A) - \sum_{l=1}^{p} |x_l| - |w_+|.
\end{equation}
Just as \cite[Lemma 4.1]{covariant-constant}, for a generic $\nu^{eq}$, one can ensure the required transversality. In particular, when the expected dimension is zero, the moduli space is finite. Moreover, upon choosing orientations on the unstable manifolds of $f$, this moduli space has an induced orientation. Hence one can define the mod $p$ count
\beqn
\#_p {\mc M}_{w_-, w_+}(A; x_1, \ldots, x_p, y)  \in \fp
\eeqn
(which is defined to be zero if the expected dimension is nonzero). 

		\begin{figure}[H]
	\def\svgwidth{0.5\linewidth}
\begingroup%
  \makeatletter%
  \providecommand\color[2][]{%
    \errmessage{(Inkscape) Color is used for the text in Inkscape, but the package 'color.sty' is not loaded}%
    \renewcommand\color[2][]{}%
  }%
  \providecommand\transparent[1]{%
    \errmessage{(Inkscape) Transparency is used (non-zero) for the text in Inkscape, but the package 'transparent.sty' is not loaded}%
    \renewcommand\transparent[1]{}%
  }%
  \providecommand\rotatebox[2]{#2}%
  \newcommand*\fsize{\dimexpr\f@size pt\relax}%
  \newcommand*\lineheight[1]{\fontsize{\fsize}{#1\fsize}\selectfont}%
  \ifx\svgwidth\undefined%
    \setlength{\unitlength}{767.51240888bp}%
    \ifx\svgscale\undefined%
      \relax%
    \else%
      \setlength{\unitlength}{\unitlength * \real{\svgscale}}%
    \fi%
  \else%
    \setlength{\unitlength}{\svgwidth}%
  \fi%
  \global\let\svgwidth\undefined%
  \global\let\svgscale\undefined%
  \makeatother%
  \begin{picture}(1,0.74901083)%
    \lineheight{1}%
    \setlength\tabcolsep{0pt}%
    \put(0,0){\includegraphics[width=\unitlength,page=1]{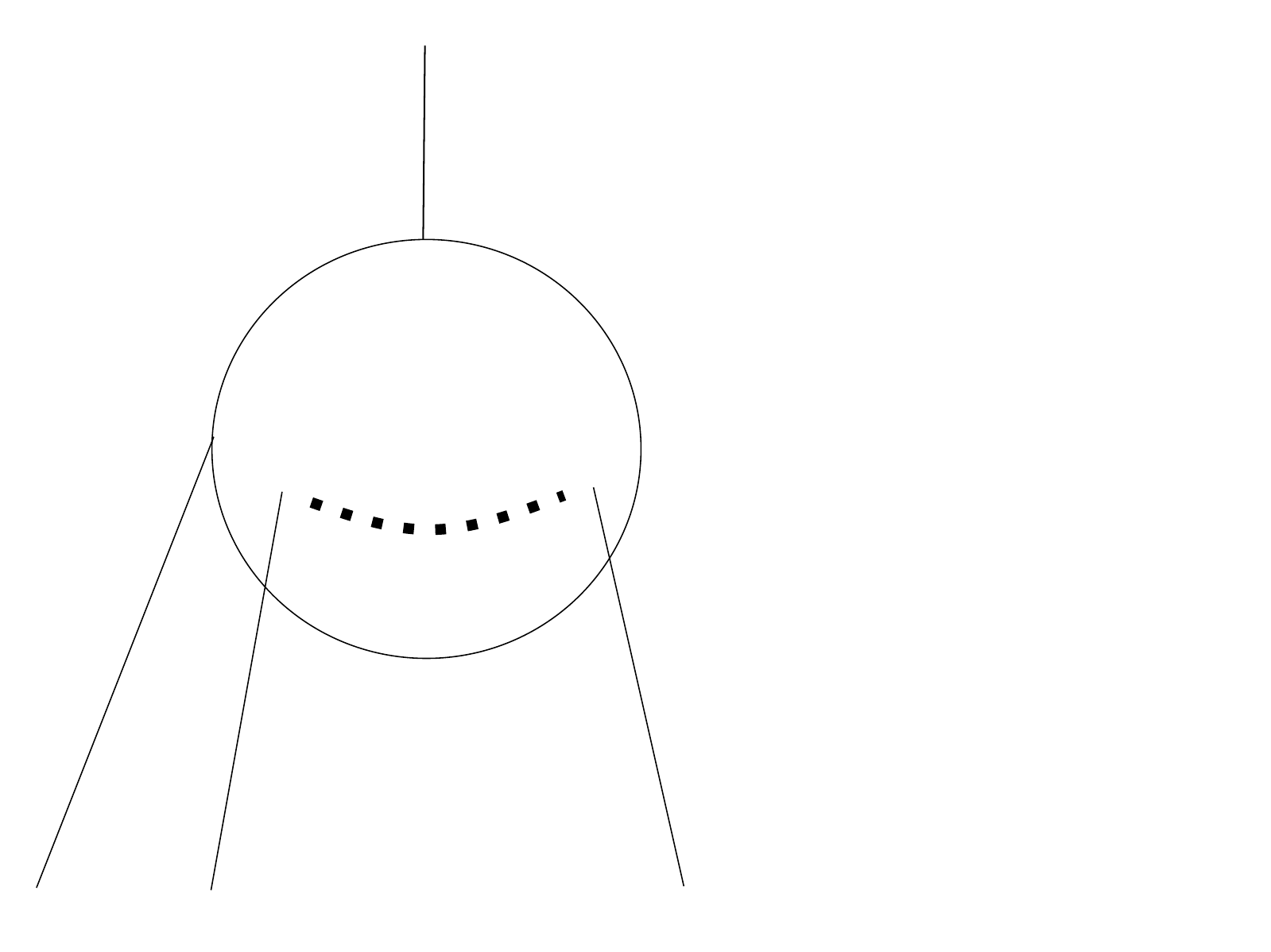}}%
    \put(0.64499171,0.0086038){\makebox(0,0)[lt]{\lineheight{1.25}\smash{\begin{tabular}[t]{l}$Z^0_{j,\delta}$\end{tabular}}}}%
    \put(0.65091441,0.70695664){\makebox(0,0)[lt]{\lineheight{1.25}\smash{\begin{tabular}[t]{l}$Z^k_{i,\epsilon}$\end{tabular}}}}%
    \put(0,0){\includegraphics[width=\unitlength,page=2]{qst.pdf}}%
    \put(0.14108895,0.0141128){\makebox(0,0)[lt]{\lineheight{1.25}\smash{\begin{tabular}[t]{l}$x_0$\end{tabular}}}}%
    \put(-0.00302825,0.01806115){\makebox(0,0)[lt]{\lineheight{1.25}\smash{\begin{tabular}[t]{l}$x_1$\end{tabular}}}}%
    \put(0.49447212,0.01207894){\makebox(0,0)[lt]{\lineheight{1.25}\smash{\begin{tabular}[t]{l}$x_j$\end{tabular}}}}%
    \put(0,0){\includegraphics[width=\unitlength,page=3]{qst.pdf}}%
    \put(0.319544,0.72112933){\makebox(0,0)[lt]{\lineheight{1.25}\smash{\begin{tabular}[t]{l}$y$\end{tabular}}}}%
  \end{picture}%
\endgroup%

			\caption{Elements of ${\mc M}_{w_-, w_+}(A; x_1, \ldots, x_p, y)$, where $w_- = Z^k_{i,\epsilon}$ and $w_+ = Z_{j,\delta}^0$.}
			\label{fig:eq-pop}
		\end{figure}

As the inhomogeneous perturbation $\nu_{v, z, x}^{eq}$ is $\zp$-equivariant, it follows that there is an orientation-preserving diffeomorphism
\beqn
\iota: {\mc M}_{w_-, w_+}(A; x_1, \ldots, x_p, y) \cong {\mc M}_{\iota(w_-), \iota(w_+)} (A; x_2, \ldots, x_p, x_1, y)
\eeqn
by $\iota(u, \zeta) = (u \circ \iota^{-1}, \iota \circ \zeta)$ where $\iota$ is the generator of $\zp$. Hence one has the equality
\beq\label{QST_equivariance}
\#_p {\mc M}_{w_-, w_+}( A; x_1, \ldots, x_p, y) = \#_p {\mc M}_{\iota(w_-), \iota(w_+)} (A; x_2, \ldots, x_p, x_1, y).
\eeq

Now one defines 
\beqn
\wt{\wt{QSt}}{}_{p,A}: CM(f)^{\otimes p}\otimes CM(g) \to CM(f) \otimes CM(g)
\eeqn
by linearly extending
\beq\label{equation:qst-tilde} 
\wt{\wt{QSt}}{}_{p, A}(x_1 \otimes \cdots \otimes x_p \otimes w_-) = \sum_{y, w_+} \#_p {\mc M}_{w_-, w_+}(A; x_1, \ldots, x_p, y) \cdot y \otimes w_+.
\eeq
As $M$ is monotone and $J$ is chosen to make simple holomorphic spheres regular, 1-dimensional moduli spaces can be compactified by adding only configurations with smooth spheres and possibly broken Morse trajectories. Such a structure of the boundary of 1-dimensional moduli spaces implies that $\wt{\wt{QSt}}{}_{p, A}$ is a chain map. 

Notice that by \eqref{QST_equivariance}, this map is $\zp$-equivariant. Namely
\beqn
\wt{\wt{QSt}}{}_{p, A}(x_2 \otimes \cdots \otimes x_p \otimes x_1 \otimes \iota(w_-)) = \iota (\wt{\wt{QSt}}{}_{p, A}( x_1 \otimes \cdots \otimes x_p \otimes w_-) ).
\eeqn
Therefore, its restriction to the $\zp$-invariant part of the domain gives a chain map
\beqn
\wt{QSt}{}_{p, A}: CM(f)^{\otimes p}_{\zp} \cong  ( CM(f)^{\otimes p} \otimes CM(g) )^{\zp} \to CM(f) \otimes CM(g)^{\zp} \cong CM(f) \otimes \rp.
\eeqn
By abuse of notation, the induced map on homology is denoted by
\beqn
\wt{QSt}{}_{p, A}: H_{\zp}(CM(f)^{\otimes p}) \to HM(f) \otimes \rp.
\eeqn

Recall we also have the quasi-Frobenius map in $\fp$-coefficients
\beqn
qF: HM(f) \to H_{\zp}(CM(f)^{\otimes p}).
\eeqn
Then we can define the degree $A$ part of the quantum Steenrod operation as 
\beqn
QSt_{p, A}:= \wt{QSt}_{p, A}\circ qF: HM(f) \to HM(f) \otimes \rp.
\eeqn
A standard argument using continuation maps shows that it induces a well-defined map
\beq
QSt_{p, A}: H(M) \to H(M) \otimes \rp
\eeq
which is  independent of the choices of $f$, $J$ and $\nu^{eq}_{v, z, x}$. The total quantum Steenrod operation is then
\beq
QSt_p:= \sum_{A \in H_2(M; {\mb Z})} q^A QSt_{p, A}: H(M) \to H(M; \Lambda^\Gamma) \otimes \rp.
\eeq
Here notice that the monotonicity condition guarantees it lands in the above target instead of $\Lambda^{\Gamma}_{\mathcal{K}_{0,p}} \otimes E(\theta)$.

\begin{thm}\cite[Theorem H]{Bai_Xu_2025}\label{thm52}
The above definition agrees with the definition using abstract perturbations for monotone targets.
\end{thm}

\subsection{Proof of Theorem \ref{thm: Kunneth}}
\label{subsec:equiv-kunneth-steenrod}

The proof is based on the TQFT/cobordism argument which is described in Figure \ref{figure_QST_Kunneth}.

\begin{figure}[h]
    \centering
    \includegraphics[width=1\linewidth]{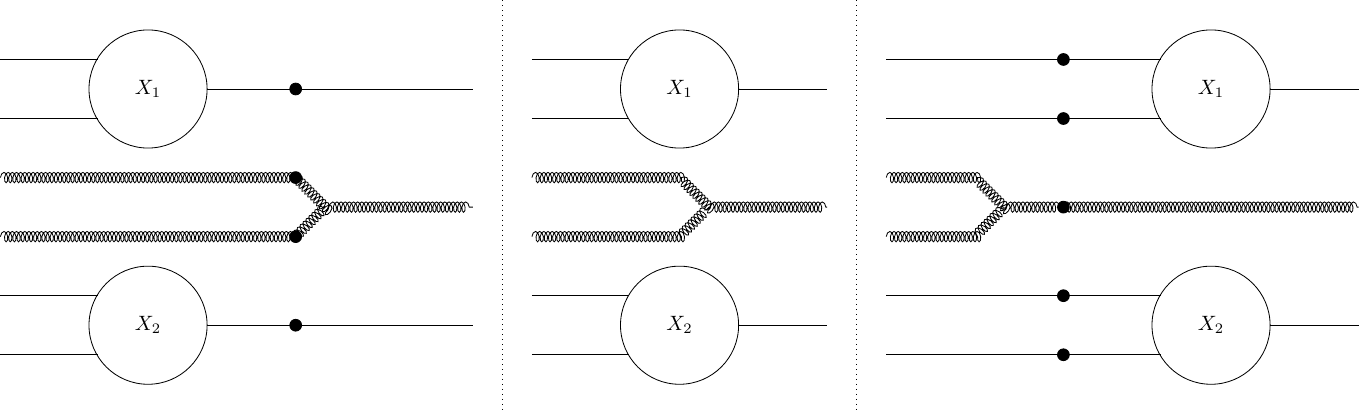}
    \caption{The cobordism leading to the Kunneth property of the quantum Steenrod operation. The coiled trees are flow trees in $S^\infty$ while the straight segments are flow lines in $X_1$ or $X_2$. The left picture corresponds to the composition $\kappa^{\zp} \circ (QSt_p^{X_1} \otimes QSt_p^{X_2})$ while the right picture corresponds to the composition $QSt_p^{X_1 \times X_2} \circ \kappa$.}
    \label{figure_QST_Kunneth}
\end{figure}


We will make the assumption that $X_1,X_2$ are closed monotone symplectic manifolds, with $X_1 \times X_2$ also monotone. Then quantum cohomology and Floer cohomology for $X_1$, $X_2$ and the product $X_1 \times X_2$ can all be defined over the Novikov ring of power series in $q$:
\beqn
\Lambda = \Big\{ \sum_{i=0}^\infty a_i q^i\ |\ a_i \in {\mb K} \Big\}
\eeqn
for any base ring ${\mb K}$.

Following the discussion in the previous subsection, we use the Morse-theoretic model to define $QSt_p$. To set things up, we fix the following data.
\begin{enumerate}
    \item For $X_i$, choose a Morse--Smale pair $(f_i, h_i)$ which gives a Morse complex $CM (f_i)$ over $\fp$. Then the product $(f_1 \times f_2, h_1 \times h_2)$ is Morse--Smale on the product $X_1\times X_2$.

    \item Fix an $\omega_i$-compatible almost complex structure $J_i$.

    \item Fix the Morse data for Y-shaped flow trees in $S^\infty$, which was used to define the ring structure on $\rp$.

    \item For $i = 1, 2$, for $X_i$, $f_i, h_i$, $J_i$, fix generic inhomogeneous terms $\nu_{X_i}^{eq}$ on $S^\infty \times \tilde{S} \times X_i$ for which we can define the quantum Steenrod operation on $X_i$ using the $\nu_{X_i}^{eq}$-perturbed $J_i$-holomorphic maps. Here $\tilde{S} = \Sigma_p$ from above.

    \item For the product manifold $X_1 \times X_2$, for the Morse--Smale pair $(f_1\times f_2, h_1 \times h_2)$, the product almost complex structure $J_1 \times J_2$, the Morse function $g_{\theta_\infty}$ on $S^\infty$, fix an inhomogeneous term $\nu_{X_1\times X_2}^{eq}$ on $S^\infty \times \tilde S \times (X_1\times X_2)$.

\end{enumerate}
These data allow us to define the following chain-level maps.
\begin{enumerate}

\item One has the chain map 
\beqn
\wt\kappa_{\zp}: \Big( CM (f_1)^{\otimes p} \otimes CM(g) \Big) \otimes \Big( CM(f_2)^{\otimes p} \otimes CM(g) \Big) \to CM (f_1\times f_2)^{\otimes p} \otimes CM(g).
\eeqn
Its restriction to the invariant part is the chain-level equivariant Kunneth map
\beqn
\kappa_{\zp}: CM (f_1)^{\otimes p}_{\zp} \otimes CM(f_2)^{\otimes p}_{\zp} \to CM (f_1 \times f_2)^{\otimes p}_{\zp}.
\eeqn

\item One has the chain map 
\beqn
\wt\kappa \otimes m_{\rp}: \Big( CM (f_1) \otimes CM (g) \Big) \otimes \Big( CM (f_2) \otimes CM( g ) \Big) \to CM (f_1 \times f_2) \otimes CM(g).
\eeqn
After restricting to the $(\zp\times \zp)$-invariant part, it induces the map
\beqn
\kappa \otimes m_{\rp}: \Big( H(X_1) \otimes \rp \Big) \otimes \Big( H(X_2) \otimes \rp \Big) \to H(X_1 \times X_2) \otimes \rp.
\eeqn

\item The chain level quantum Steenrod operations for $X_1$, $X_2$ and $X_1 \times X_2$.
\end{enumerate}

Now we consider moduli spaces of configurations described by Figure \ref{figure_QST_Kunneth}. Let $\rho \in {\mb R}$ be an additional parameter. Choose a 1-parameter family of inhomogeneous terms 
\beqn
\nu_{X_1\times X_2, \rho}^{eq},\ \rho \in {\mb R}
\eeqn
on $(S^\infty \times S^\infty) \times \tilde S \times (X_1 \times X_2)$ such that 
\begin{enumerate}
    \item When $\rho \geq 0$, $\nu_{X_1\times X_2, \rho}^{eq}$ does not depend on the second $S^\infty$-factor.
    
    \item For $\rho \gg 0$, $\nu_{X_1 \times X_2, \rho}^{eq}$ coincides with the inhomogeneous term $\nu_{X_1 \times X_2}^{eq}$ used to define the quantum Steenrod operation on $X_1\times X_2$.

    \item For $\rho \ll 0$, $\nu_{X_1 \times X_2, \rho}^{eq}$ coincides with the product $\nu_{X_1}^{eq} \times \nu_{X_2}^{eq}$.
\end{enumerate}

Given critical points
\beqn
x_{i, 1}, \ldots, x_{i, p}, y_i \in {\rm crit}(f_i), w_1 \in {\rm crit}(g_{\theta_1}), w_2\in {\rm crit}(g_{\theta_2}), w_\infty \in {\rm crit} (g_{\theta_\infty})
\eeqn
and $A_i \in H_2(X_i; {\mb Z})$, consider a moduli space ${\mc M}_\rho$ parametrized by $\rho \in {\mb R}$ described as follows. 
\begin{enumerate}

\item When $\rho \leq 0$, notice that we can write $\nu_{X_1\times X_2, \rho}^{eq}$ as the product of a pair of inhomogeneous terms $\nu_{X_1, \rho}^{eq}$ and $\nu_{X_2, \rho}^{eq}$. Then define
\beqn
{\mc M}_\rho = \left\{ \left. \begin{array}{c} x_{i, l}: (-\infty, 0] \to X_i,\ i = 1, 2, 1 \leq l \leq p,\\ y_i: [0, +\infty) \to X_i,\ i = 1, 2,\\
w_1, w_2: (-\infty, 0] \to S^\infty,\\ 
w_\infty: [0, +\infty) \to S^\infty,\\
u_i: \tilde S \to X_i,\ i = 1, 2 \end{array} \right| \begin{array}{c} \ov\partial u_i = \nu_{X_i, w_i(\rho)}^{eq} (u_i), \\ x_{i, l}'(s) = \nabla f_i , y_i'(s) = \nabla f_i,\\ \displaystyle \lim_{s \to -\infty} x_{i, l}(s) = x_{i, l}, \displaystyle \lim_{s \to +\infty} y_i(s) = y_i,\\
\displaystyle \lim_{s \to -\infty} w_i(s) = v_i, \displaystyle \lim_{s \to +\infty} w_\infty(s) = v_\infty,\\
x_{i, l}(0) = u_i(z_l),\ y_i(0) = u_i(z_\infty),\\
w_1(0) = w_2(0) = w_\infty(0)
\end{array} 
 \right\},
\eeqn

\item When $\rho \geq 0$, define
\beqn
{\mc M}_\rho = \left\{ \begin{array}{c} x_{i, l}: (-\infty, 0] \to X_i,\ i = 1, 2, 1\leq l \leq p,\\
y_i: [0, +\infty) \to X_i,\ i = 1, 2,\\
w_1, w_2: (-\infty, 0] \to S^\infty,\\
w_\infty: [0, +\infty) \to S^\infty,\\
u_i: \tilde S \to X_i,\ i = 1, 2 \end{array} \left| \begin{array}{c} \ov\partial (u_1 \times u_2) = \nu_{X_1 \times X_2, w_\infty(\rho)}^{eq}(u_i),\\
 x_{i, l}'(s) = \nabla f_i, y_i'(s) = \nabla f_i,\\ \displaystyle \lim_{s \to -\infty} x_{i, l}(s) = x_{i, l}, \displaystyle \lim_{s \to +\infty} y_i(s) = y_i,\\
\displaystyle \lim_{s \to -\infty} w_i(s) = v_i, \displaystyle \lim_{s \to +\infty} w_\infty(s) = v_\infty,\\
x_{i, l}(0) = u_i(z_l),\ y_i(0) = u_i(z_\infty),\\
w_1(0) = w_2(0) = w_\infty(0),\\
\end{array} \right.\right\}
\eeqn
\end{enumerate}

We only consider moduli spaces of expected dimension zero (without deforming $\rho$). Then including the deformation of $\rho$, define
\beqn
{\mc M}_\pm:= \bigsqcup_{\rho \in {\mb R}_\pm} {\mc M}_\rho.
\eeqn
Then one can see that both ${\mc M}_\pm$ are the zero loci of some Fredholm section of index $1$ defined over a certain Banach manifold. The union of all such moduli spaces admits a natural free $\zp$-action. One can perturb the family of inhomogeneous terms $\nu_{X_1\times X_2, \rho}^{eq}$ to achieve $\zp$-equivariant transversality. In addition, one can make the moduli space ${\mc M}_0$ transverse. 

One can compactify ${\mc M}_\pm$ by allowing sphere bubbles, trajectory breakings, and the parameter $\rho \to \pm\infty$. As we only consider the index $1$ case, the monotonicity condition and the transversality property of simple holomorphic spheres imply that the compactification contains, besides ${\mc M}_\pm$, either configurations having at most one Morse trajectory breaking with finite $\rho$, or configurations at $\rho = \pm\infty$ shown in Figure \ref{figure_QST_Kunneth}. 

Therefore, one obtains a $\zp$-equivariant chain homotopy between the two chain maps
\begin{align*}
&\ \wt\kappa{}^{\zp} \circ ( \wt{QSt}{}_p^{X_1} \otimes \wt{QSt}{}_p^{X_2}),\ &\ \wt{QSt}{}_p^{X_1\times X_2} \circ (\wt\kappa \otimes m_{{\mc R}_p}).
\end{align*}
from
\beqn
\Big( C(f_1)^{\otimes p} \otimes C(g_{\theta_1})  \Big) \otimes \Big( C(f_2)^{\otimes p} \otimes C(g_{\theta_2}) \Big)
\eeqn
to 
\beqn
C(f_1\times f_2)^{\otimes p} \otimes C(g_{\theta_\infty}).
\eeqn
Restricting to the $\zp \times \zp$-invariant part of the domain of the chain maps, one also obtains a chain homotopy. Hence over homology, one has
\beqn
\kappa^{\zp} \circ ( QSt_p^{X_1}\otimes QSt_p^{X_2}) = QSt_p^{X_1\times X_2} \circ (\kappa \otimes m_{{\mc R}_p}).
\eeqn

\subsection{Kunneth property for the local equivariant pair-of-pants product}\label{subsec: local-kunneth}

We now prove an analogue of Theorem \ref{thm: Kunneth} in the scenario of the local equivariant Floer cohomology, which is used in the proof of the divisibility lemma (Lemma \ref{lma:div}). We use the framework based on inhomogeneous perturbations.

\begin{prop}\label{prop55}
Given symplectic manifolds $X_1, X_2$ equipped with Hamiltonians $H_1, H_2$, and an isolated $\ov{x}_i \in \tilde {\mc O}(H_i)$ such that $\ov{x}_i^{(p)} \in \tilde {\mc O}(H_i^{(p)})$ is also isolated, the following diagram commutes up to the Koszul sign. 
\beqn
\xymatrix{ HF^\loc (H_1, \ov{x}_1) \underset{\fp}{\otimes} HF^\loc (H_2, \ov{x}_2) \ar[rr]^-{\fst_p^{\rm loc} \otimes \fst_p^{\rm loc}} \ar[d]_{\kappa^\loc}   &     &HF^\loc_\zp (H_1^{(p)}, \ov{x}_1^{(p)}) \otimes HF_\zp^\loc (H_2^{(p)}, \ov{x}_2^{(p)}) \ar[d]^{\kappa^{\rm loc}_{\zp}}\\
     HF^\loc (H_1 \times H_2, \ov{x}_1\times \ov{x}_2 ) \ar[rr]_{\fst_p^{\rm loc}} & &   HF_\zp^\loc (H_1^{(p)}\times H_2^{(p)}, \ov{x}_1^{(p)}\times \ov{x}_2^{(p)} )  }.
\eeqn
\end{prop}

\begin{proof}
The proof is similar to that of Theorem \ref{thm: Kunneth}. The idea is illustrated in Figure \ref{figure_Kunneth_fst}. 
\begin{figure}[h]
    \centering
    \includegraphics[width=1.1\linewidth]{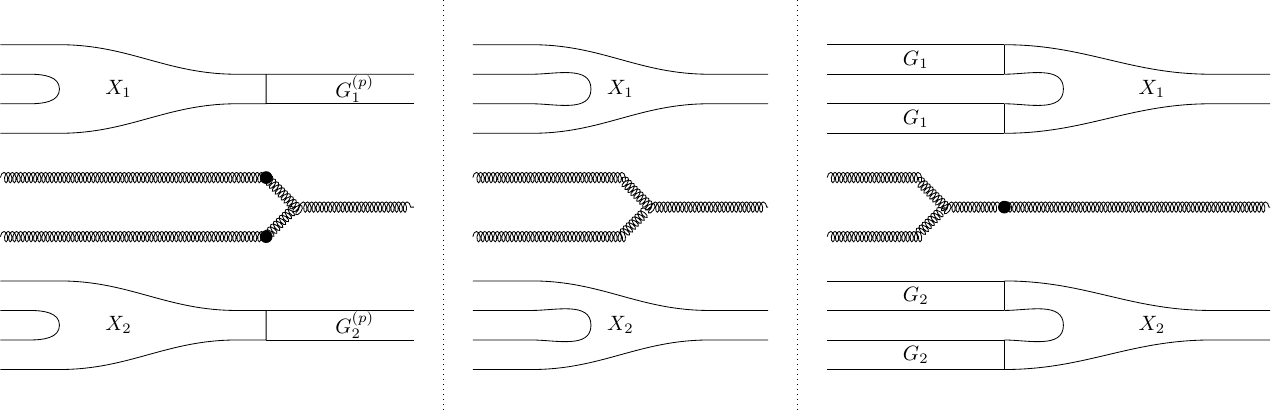}
    \caption{Kunneth property of the local equivariant Floer cohomology. The figure on the left corresponds to the composition $\kappa_{\zp}^{\rm loc} \circ ({FSt}_p^{\rm loc} \otimes {FSt}_p^{\rm loc})$ while the figure on the right corresponds to the composition ${FSt}_p^{\rm loc} \circ \kappa^{\rm loc}$.}
    \label{figure_Kunneth_fst}
\end{figure}
We first choose data so that all objects and arrows appearing in the commutative diagram can be defined on the chain level.

\begin{enumerate}

\item Choose nondegenerate perturbations $G_i$ of $H_i$ and an almost complex structure $J_i$ for which the local Floer cochain complexes $CF^{\rm loc}(H_i, G_i, \ov{x}_i)$ are defined. In this way, the local Floer cochain complex $CF^{\rm loc}(H_1 \times H_2, G_1 \times G_2, \ov{x}_1\times \ov{x}_2)$ is also defined using the product data, giving
\beqn
CF^{\rm loc}(H_1\times H_2, G_1 \times G_2, \ov{x}_1 \times \ov{x}_2) \cong CF^{\rm loc}(H_1, G_1, \ov{x}_1) \otimes CF^{\rm loc}(H_2, G_2, \ov{x}_2).
\eeqn
This isomorphism coincides with the chain-level Kunneth map $\kappa^{\rm loc}$, which can be defined following the line in Section \ref{subsec:equiv-kunneth}. From now on, the almost complex structures will be fixed and we omit the discussion on the dependence on them.

\item To define the local equivariant complexes $CF_{\zp}^{\rm loc}(H_i^{(p)}, G_i^{(p)}, \ov{x}_i^{(p)})$, we also choose a family of perturbations
\beqn
{\ms G}_{i,v}^{eq}: S^1 \times X_i \to {\mb R}, \ v \in S^\infty
\eeqn
of $G_i^{(p)}$ so the complex is defined. 

\item From now on, we drop from the notations the dependence on $H_1$ and $H_2$ as all further constructions only depend on $G_1$ and $G_2$. To define the local equivariant complex for the product $G_1^{(p)} \times G_2^{(p)}$, we choose another perturbation
\beqn
{\ms G}_{\infty,v}^{eq}: S^1 \times X_1 \times X_2 \to {\mb R}, \ v \in S^\infty
\eeqn
so the complex $CF_{\zp}^{\rm loc}(G_1^{(p)} \times G_2^{(p)}, \ov{x}_1^{(p)} \times \ov{x}_2^{(p)})$ is defined.

\item Choose a family of vector fields $Y_s$ on $S^\infty$ parametrized by $s \in T_{2, 1}$ so that the multiplicative structure $\rp \otimes \rp \to \rp$ is defined. Then choose a family 
\beqn
\tilde {\ms G}_{s, v}^{eq}: \tilde S \times X_1 \times X_2 \to {\mb R},\ s \in T_{2, 1},\ v \in S^\infty
\eeqn
which connects ${\ms G}_{1,v}^{eq} \times {\ms G}_{2, v}^{eq}$ and ${\ms G}_{\infty, v}^{eq}$ so that the chain-level equivariant Kunneth map
\beqn
\kappa^{\rm loc}_{\zp}: CF_{\zp}^{\rm loc} (G_1^{(p)}, \ov{x}_1^{(p)}) \otimes CF_{\zp}^{\rm loc} (G_2^{(p)}, \ov{x}_2^{(p)} ) \to CF_{\zp}^{\rm loc} (G_1^{(p)}\times G_2^{(p)}, \ov{x}_1^{(p)} \times \ov{x}_2^{(p)})
\eeqn
is defined.

\item Choose $\zp$-equivariant family of connections on the trivial ${\rm Ham}(X_i)$-bundle on $\tilde S$ denoted by
\beqn
\sigma_{i, v}^{eq},\ v \in S^\infty
\eeqn
which is equal to $G_i dt$ on each of the incoming cylindrical ends and which is equal to ${\ms G}_i^{eq}$ on the positive cylindrical end, to define the chain level equivariant pair-of-pants product
\beqn
\wt{FSt}{}_p^{\loc}: CF^{\rm loc}(G_i, \ov{x}_i)^{\otimes p}_{\zp} \to CF^{\rm loc}_{\zp}(G_i^{(p)}, \ov{x}_i^{(p)}).
\eeqn

\item Choose a $\zp$-equivariant family of connections on the trivial ${\rm Ham}(X_1\times X_2)$-bundle over $\tilde S$ denoted by 
\beqn
\sigma_{\infty, v}^{eq},\ v \in S^\infty
\eeqn
which is equal to $(G_1 \times G_2) dt$ on each of the incoming cylindrical ends and is equal to ${\ms G}_\infty^{eq} dt$ on the positive cylindrical end, to define the chain-level equivariant pair-of-pants product
\beqn
\wt{FSt}{}_p^{X_1\times X_2}
\eeqn

\end{enumerate}

Now we would like to choose a family of connections on the trivial ${\rm Ham}(X_1)\times {\rm Ham}(X_2)$-bundle over $\tilde S$, denoted by 
\beqn
\tilde \sigma_{\rho, v_1, v_2}^{eq},\ \rho \in {\mb R}, v_1, v_2 \in S^\infty
\eeqn
satisfying
\begin{enumerate}
    \item $\tilde \sigma_{\rho, v_1, v_2}^{eq}$ is $\zp$-equivariant with respect to the diagonal action on $S^\infty \times S^\infty \times \tilde S$.

    \item For all $\rho$, on each of the negative cylindrical ends of $\tilde S$, $\tilde \sigma_{\rho, v_1, v_2}^{eq}$ coincides with $(G_1 \times G_2) dt$. 

    \item For all $\rho$, on the positive cylindrical end of $\tilde S$, $\tilde \sigma_{\rho, v_1, v_2}^{eq}$, when restricted to the diagonal of $S^\infty$, coincides with ${\ms G}_{\infty, v}^{eq} dt$.

    \item When $\rho \ll 0$, $\tilde \sigma_{\rho, v_1, v_2}^{eq}$ coincides with $\sigma_{1, v_1}^{eq}\times \sigma_{2, v_2}^{eq}$.

    \item When $\rho \gg 0$, $\tilde \sigma_{\rho, v_1, v_2}^{eq}$, when restricted to the diagonal $v = v_1 = v_2$ of $S^\infty$, coincides with $\sigma_{\infty, v}^{eq}$.

\end{enumerate}

We still use the canonical map $T_{2,1} \to {\mb R}$. For $i = 1, 2$, choose $\ov{x}_{i, 1}, \ldots, \ov{x}_{i, p} \in \tilde {\mc O}(G_i, \ov{x}_i)$, $\ov{y}_i \in \tilde {\mc O}(G_i^{(p)}, \ov{x}_i^{(p)})$. Choose $w_1, w_2, w_\infty \in {\rm crit}(g)$. Then for each fixed $\rho$, define the moduli space ${\mc M}_\rho$ as follows. When $\rho \leq 0$ resp. $\rho \geq 0$, ${\mc M}_\rho$ consists of pairs $(u, \zeta)$ where $u: \tilde S \to X_1 \times X_2$ is a smooth map, $\zeta = (\zeta_1, \zeta_2, \zeta_\infty): T_{2,1} \to S^\infty$ is a $Y$-perturbed gradient tree, such that 
    \beqn
    (\nabla^{\tilde \sigma_{\rho, \zeta_1(\rho), \zeta_2(\rho)}^{eq}} u)^{0,1} = 0\ {\rm resp.}\ (\nabla^{\tilde \sigma_{\rho, \zeta_\infty(\rho), \zeta_\infty(\rho)}^{eq}} u)^{0,1} = 0
    \eeqn
    and such that $u$ is asymptotic to $\ov{x}_{1, k} \times \ov{x}_{2, k}$ along the $k$-th negative end and asymptotic to $\ov{y}_1 \times \ov{y}_2$ along the positive end, while $\zeta$ is asymptotic to $w_1, w_2, w_\infty$ respectively along the corresponding edges. Define
    \beqn
    {\mc M}_\pm:= \bigsqcup_{\pm \rho \geq 0} {\mc M}_\rho.
    \eeqn
    One can see that each ${\mc M}_\rho$ is the zero locus of a smooth Fredholm section of a certain Banach vector bundle. Moreover, by choosing $\tilde \sigma_{\rho, v_1, v_2}^{eq}$ generically, one can make ${\mc M}_\pm$ together with ${\mc M}_0$ regular. One can also compactify the moduli space ${\mc M}_\pm$ by suitably adding broken configurations (notice that in this local situation, sphere bubbling won't happen). Then when the index of an individual moduli space ${\mc M}_\rho$ is zero, the compactifications $\ov{\mc M}_\pm$ are compact 1-dimensional topological manifolds with boundary. Then the union of $\ov{\mc M}_\pm$ along ${\mc M}_0$ defines a chain homotopy between the two cochain maps. The first one, corresponding to the boundary $\rho = -\infty$, is the composition 
    \begin{multline*}
    \wt\kappa_{\zp}^{\rm loc} \circ (\wt{\wt{FSt}}{}_p^{\rm loc} \otimes \wt{\wt{FSt}}{}_p^{\rm loc}): \\
    \Big( CF^{\rm loc}(H_1, \ov{x}_1)^{\otimes p} \otimes CM(g) \Big) \otimes  \Big( CF^{\rm loc}(H_2, \ov{x}_2)^{\otimes p} \otimes CM(g) \Big) \\
    \to CF^{\rm loc}( H_1^{(p)}\times H_2^{(p)}, \ov{x}_1^{(p)}\times \ov{x}_2^{(p)}) \otimes CM(g).
    \end{multline*}
    The second one, corresponding to the boundary $\rho = +\infty$, is the composition
    \beqn
    \wt{\wt{FSt}}{}_p^{\rm loc} \circ \wt\kappa^{\otimes p}
    \eeqn
    between the same pair of cochain groups. The chain homotopy is $\zp$-equivariant. Hence we can restrict to the $\zp\times \zp$-fixed subspace, which gives a chain homotopy between
    \beqn
    \wt{\kappa}{}_{\zp}^{\rm loc} \circ ( \wt{FSt}{}_p^{\rm loc}\otimes \wt{FSt}{}_p^{\rm loc}): CF^{\rm loc}(H_1, \ov{x}_1)^{\otimes p}_{\zp} \otimes CF^{\rm loc}(H_2, \ov{x}_2)^{\otimes p}_{\zp} \to CF^{\rm loc}_{\zp}(H_1^{(p)}\times H_2^{(p)}, \ov{x}_1^{(p)}\times \ov{x}_2^{(p)}).
    \eeqn
    and $\wt{FSt}{}_p^{\rm loc} \circ \wt\kappa{}_{\zp}^{\otimes p}$ between the same pair of groups. Hence as maps between cohomology, one has the equality
    \beqn
    \kappa_{\zp}^{\rm loc} \circ ({FSt}_p^{\rm loc} \otimes {FSt}_p^{\rm loc}) = {FSt}_p^{\rm loc} \circ \kappa_{\zp}^{\otimes p}.
    \eeqn
Lastly, one has the following commutative diagram (which can be proved exactly the same as Lemma \ref{Lemma_Kunneth_qF}).
\beqn
\xymatrix{  HF^{\rm loc}(H_1, \ov{x}_1) \otimes HF^{\rm loc}(H_2, \ov{x}_2) \ar[rr]^-{qF \otimes qF} \ar[d]_{\kappa^{\rm loc}} & & H ( CF^{\rm loc}(H_1, \ov{x}_1)^{\otimes p}_{\zp}) \otimes H( CF^{\rm loc}(H_2, \ov{x}_2)^{\otimes p}_{\zp}) \ar[d]^{\kappa_{\zp}}\\
HF^{\rm loc}(H_1 \times H_2, \ov{x}_1 \times \ov{x}_2) \ar[rr]_-{qF}  &   &     H(CF^{\rm loc}(H_1\times H_2, \ov{x}_1\times \ov{x}_2)^{\otimes p}_{\zp} )   }
\eeqn
Then the proof this proposition is finished. 
\end{proof}

\appendix


\section{Technical discussions on the quantum Steenrod operation} \label{appendixa}
\label{sec:alternative-qst}


In this appendix, we compare the definition of the quantum Steenrod power operation $\qst_p$ in \cite{covariant-constant} and the one recalled in Section \ref{subsec:quantumStpower} restricted to the monotone situation. Throughout this section, all cohomology is taken with $\F_p$-coefficients unless otherwise stated. We first notice that by Theorem \ref{thm52}, the latter agrees with the one used in Section \ref{section5}. Hence we will not use abstract perturbation but only the inhomogeneous perturbation incorporated with the Morse model. 

The idea is as follows. To define each individual count which leads to the definition of the quantum Steenrod operation, one considers the moduli spaces of $J$-holomorphic spheres perturbed by an inhomogeneous term $\nu_{v, z, x}^{eq}$ parametrized by $v$ contained in a certain chain in $S^\infty$. In Subsection \ref{subsection52}, the parameter $\nu$ is contained in some Morse chain associated to the $\zp$-invariant Morse function $g: S^\infty \to {\mb R}$. In \cite{covariant-constant}, the parameter $\nu$ is chosen to be contained in a singular chain. The comparison is basically reduced to the comparison between corresponding chains inside $S^\infty$.

\subsection{Setup}

Now we start the comparison. 
Let $M$ be a monotone symplectic manifold. In the version of \cite{covariant-constant}, it was not required that the inhomogeneous term $\nu^{eq}$ is invariant under the ${\mb Z}_{\geq 0}$-action; however, imposing this condition in the setup of \cite{covariant-constant} does not affect that construction. Hence we can choose such a symmetric inhomogeneous term $\nu^{eq}$. We also choose a generic almost complex structure $J$ to ensure that the counts are well-defined. To temporarily distinguish the two definitions, let $\qst^{\rm Morse}_p$ be the quantum Steenrod operation defined in Subsection \ref{subsection52} and let $\qst^{\rm cell}_p$ be the one defined in \cite{covariant-constant}.


The comparison is done on the level of moduli spaces. Choose a Morse function $f$ together with a Morse--Smale metric. For any subset $Q \subset S^\infty$, denote by
\beqn
{\mc M}_Q(A; x_1, \ldots, x_p, y)
\eeqn
the set of pairs $(u, v)$ where $v \in Q$ and $u: \Sigma_p \to M$ solves the equation
\beqn
\ov\partial_J u(z) = \nu_{v, z, u(z)}^{eq}
\eeqn
and satisfies the constraints at the $p+1$ marked points imposed by the unstable/stable manifolds of $x_1, \ldots, x_p, y$.

The cochain groups used in the two approaches are also different. We choose the following identification. Define 
\beqn
\begin{split}
    \Psi: CM(f) \otimes \rp \to & CM(f) \otimes CM(g)\\
    x \otimes t^k \mapsto &\ \sum_{l=1}^p x  \otimes w_{2k}^l,\\
    x \otimes t^k \theta \mapsto &\ \sum_{l=1}^p x \otimes w_{2k+1}^l,
\end{split}
\eeqn
cf. Equations \eqref{critical_points_1} and \eqref{critical_point_2}. Then the image of $\Psi$ is contained in the $\zp$-invariant part and defines an isomorphism onto its image.

\subsection{The comparison}

We first briefly recall the cells in $S^{\infty}$ used in the paper \cite{covariant-constant}. In each nonnegative degree $i$, we specify a cell $\Delta_i \subset S^\infty$ which projects down to a singular cycle in $B\zp$ representing a generator of $H_i (B\zp; \fp)$. 
For an even degree $i = 2k$, it is
\beqn
\Delta_{2k} = \Big\{ v = (v_j)_{j \geq 0} \in S^\infty \ |\ v_k \geq 0,\ v_{k+1} = v_{k+2} = \cdots = 0\Big\};
\eeqn
for an odd degree $i = 2k+1$, it is 
\beqn
\Delta_{2k+1} = \Big\{ v = (v_j)_{j \geq 0} \in S^\infty\ |\ {\rm arg}(v_k) \in  [0,2\pi/p],  v_{k+1} = v_{k+2} = \cdots = 0\Big\}.
\eeqn
We have 
\begin{align}\label{equation:d-delta} 
&\ \partial \Delta_{2k} = \Delta_{2k-1} + \tau \Delta_{2k-1} + \dots + \tau^{p-1} \Delta_{2k-1},\ &\ \partial \Delta_{2k+1} = \tau \Delta_{2k} - \Delta_{2k}.
\end{align}

We use the notations in \cite{covariant-constant} to define the chain-level maps. To this end, for each $i\geq 0$, $A \in H_2(M;{\mb Z})$, and critical points $x_1, \ldots, x_p, y$ of $f$, one has a well-defined mod $p$ count
\beqn
\#_p {\mc M}_{\Delta_i}(A; x_1, \ldots, x_p, y) \in \fp
\eeqn
(which is zero by definition if the expected dimension is nonzero). These counts define linear maps
\beqn
\qst^{\rm cell}_{p,i,A}: CM(f)^{\otimes p} \to CM(f).
\eeqn
Maps of this form assemble to define the quantum Steenrod operation $\qst^{\rm cell}_p$. If $\alpha \in CM(f)$ is a Morse cocycle, we define
\beqn
\qst^{\rm cell}_p([\alpha]) = \sum_{A \in H_2(M; {\mb Z})} q^A QSt^{\rm cell}_{p,A}([\alpha])
\eeqn
while 
\beqn
\qst^{\rm cell}_{p,A} ([\alpha]) = \sum_{k=0}^\infty \Big( \qst^{\rm cell}_{p,2k, A} (\alpha \otimes \cdots \otimes \alpha ) \otimes t^k + \qst^{\rm cell}_{p,2k+1, A}(\alpha \otimes \cdots \otimes \alpha) \otimes t^k \theta \Big)
\eeqn
which is, by \cite[Lemma 4.4]{covariant-constant}, a cocycle whose cohomology class only depends on the class $[\alpha]$.

Now we describe the cellular decomposition of $S^{\infty}$ used in this paper, which arises from Morse theory. Recall the function $g: S^\infty \to {\mb R}$ has critical points given in \eqref{critical_points_1} and \eqref{critical_point_2}. Let $W^u(w_i^l)$ resp. $W^s(w_i^l)$ be the unstable resp. stable manifolds of the critical point $w_i^l$ with respect to the ascending flow of $g$. Then define
\beqn
P_i^l:= W^u(w_0^l) \cap W^s(w_i^1) \subset S^\infty
\eeqn
which is an $i$-dimensional cell. 
Define
\beqn
P_i:=\bigcup_l P_i^l \subset S^\infty.
\eeqn
By counting zero-dimensional moduli spaces, one obtains a set of mod $p$ counts
\beqn
\#_p {\mc M}_{P_i}(A; x_1, \ldots, x_p, y).
\eeqn
Then define a linear map
\beqn
\qst^{\rm Morse}_{p,i,A}: CM(f)^{\otimes p} \to CM(f)
\eeqn
using these as coefficients.

\begin{lemma}
The quantum Steenrod operation $\qst^{\rm Morse}_p$ defined in Subsection \ref{subsection52} is determined by $\qst^{\rm Morse}_{p,i,A}$ in the same way as $QSt^{\rm cell}_{p,i,A}$ determines $QSt^{\rm cell}_p$. More precisely, for any Morse cocycle $\alpha\in CM(f)$, one has 
\beqn
QSt^{\rm Morse}_{p, A} ([\alpha]) = \big[\Psi \left(  \sum_{k=0}^\infty \Big( QSt^{\rm Morse}_{p,2k, A}  (\alpha \otimes \cdots \otimes \alpha ) \otimes t^k + QSt^{\rm Morse}_{p,2k+1, A}  (\alpha \otimes \cdots \otimes \alpha) \otimes t^k \theta \Big) \right)\big].
\eeqn
\end{lemma}

\begin{proof}
To simplify notations, we omit the superscript ``Morse" in this proof. We may assume that $\alpha$ is a single critical point; the case when $\alpha$ is a linear combination of critical points only adds notational complexities. Consider the chain
\beqn
\tilde \alpha^p:=\sum_{l=1}^p \alpha \otimes \cdots \otimes \alpha \otimes w_0^l\in CM(f)^{\otimes p} \otimes CM(g).
\eeqn
As $\alpha$ is a cocycle, one can see that $\tilde \alpha^p$ is $d_{\zp}$-closed. Now by the definition of the chain map $\wt{\wt{QSt}}_{p, A}$ and the $\zp$-equivariance of the counts, one has
\beqn
\begin{split}
\wt{\wt{QSt}}{}_{p, A}(\tilde \alpha^p) = &\ \sum_{y \in {\rm crit}(f)} \sum_{l=1}^p \sum_{w \in {\rm crit}(g)} \#_p {\mc M}_{w_0^l, w}(A; \alpha, \ldots, \alpha, y) \cdot y \otimes w\\
= &\ \sum_{y \in {\rm crit}(f)} \sum_{i\geq 0} \left( \sum_{l=1}^p  \#_p {\mc M}_{w_0^l, w_i^1} (A; \alpha, \ldots, \alpha, y) \cdot y \right) \otimes \sum_{l=1}^p \tau^{l-1}(w_i^1)\\
= &\  \sum_{y \in {\rm crit} (f)} \sum_{i\geq 0} \left( \#_p {\mc M}{}_{P_i}(A; \alpha, \ldots, \alpha, y) \cdot y  \right) \otimes \sum_{l=1}^p w_i^l\\
= &\ \Psi \left( \sum_{k=0}^\infty QSt_{p,2k,A} (\alpha \otimes \cdots \otimes \alpha) \otimes t^k + \sum_{k=0}^\infty QSt_{p,2k+1,A} (\alpha \otimes \cdots \otimes \alpha) \otimes t^k \theta \right).
\end{split}
\eeqn
As the cohomology classes of both sides are well-defined, the lemma follows. 
\end{proof}

As a consequence, the comparison is reduced to the following situation.

\begin{lemma}\label{lemma_QST_comparison}
Fix $A \in H_2(M; {\mb Z})$ and $i \geq 0$. For any generic inhomogeneous term $\nu_{v, z, x}^{eq}$, for any Morse cocycle $\alpha\in CM(f)$, the two Morse cocycles 
\beqn
QSt^{\rm Morse}_{p,i,A} (\alpha \otimes \cdots \otimes \alpha),\ QSt^{\rm cell}_{p,i,A}(\alpha \otimes \cdots \otimes \alpha)
\eeqn
are cohomologous.
\end{lemma}

This lemma will be proved in a moment. By combining the two above lemmas, one obtains the following theorem.

\begin{thm}
For a monotone closed symplectic manifold $M$, the definition of the quantum Steenrod operation given in Subsection \ref{subsection52} agrees with the definition of $QSt^{\rm cell}_p$. 
\end{thm}

\subsection{Proof of Lemma \ref{lemma_QST_comparison}}

The proof is based on explicit comparisons of the involved chains in $S^\infty$.

\begin{lemma}\label{lemmaa4}
There exists a sequence of piecewise smooth chains $B_i \in C_i(S^i; {\mb Z}) \subset C_i(S^\infty; {\mb Z})$ for $i \geq 1$ satisfying
\beqn
P_i - \Delta_i = \left\{ \begin{array}{cc}  \partial B_{2k+1},\ &\ i = 2k,\\
                                            \tau(B_{2k+1}) - B_{2k+1} + \partial B_{2k+2},\ &\ i = 2k+1.\end{array} \right.
                                            \eeqn
\end{lemma}

\begin{proof}
$\Delta_0 = P_0$ is obvious and $\Delta_1 = P_1$ because they are both the unique positively oriented segment of $S^1\subset S^\infty$ between $\Delta_0$ and $\tau(\Delta_0)$. Hence we can choose $B_1 = 0$, $B_2 = 0$. 

For any positive even degree $i = 2k$, by the differential of the Morse cochain complex $CM(g)$, one has
\beqn
\partial P_{2k} = \sum_{l=1}^p \tau^{l-1}( P_{2k-1}).
\eeqn
Together with \eqref{equation:d-delta}, one obtains
\beqn
\partial( P_{2k} - \Delta_{2k}) = \sum_{l=1}^p \big( \tau^{l-1}(P_{2k-1}) - \tau^{l-1}(\Delta_{2k-1}) \big).
\eeqn
Notice that as chains, the sphere $S^{2k-1}$ can be written as
\beqn
S^{2k-1} = \sum_{l=1}^p \tau^{l-1}(\Delta_{2k-1}) = \sum_{l=1}^p \tau^{l-1}(P_{2k-1}).
\eeqn
It follows that $\partial (P_{2k} - \Delta_{2k}) = 0$. As $S^\infty$ is contractible, there hence exists $B_{2k+1} \in C_{2k+1}(S^{2k+1};{\mb Z})$ such that
\beqn
P_{2k} - \Delta_{2k} = \partial B_{2k+1}.
\eeqn
For any positive odd degree $i = 2k+1$, by the differential of the Morse cochain complex $CM(g)$, one has 
\beqn
\partial P_{2k+1} = \tau(P_{2k}) - P_{2k}.
\eeqn
Then together with \eqref{equation:d-delta}, one has 
\beqn
\partial( P_{2k+1} - \Delta_{2k+1}) = (\tau - {\rm Id}) ( P_{2k} - \Delta_{2k}) = (\tau - {\rm Id}) \partial B_{2k+1}.
\eeqn
Again, by contractibility of $S^\infty$, there exists $B_{2k+2} \in C_{2k+2}(S^{2k+2}, {\mb Z})$ such that 
\beqn
P_{2k+1} - \Delta_{2k+1} = (\tau-1) B_{2k+1} + \partial B_{2k+2}.\qedhere
\eeqn
\end{proof}

\begin{proof}[Proof of Lemma \ref{lemma_QST_comparison}]
Again, we assume that $\alpha$ is a single critical point; the case when $\alpha$ is a linear combination of critical points follows from the special case with only adding more notations. We can choose the inhomogeneous term $\nu^{eq}$ such that the moduli spaces ${\mc M}_Q(A; x_1, \ldots, x_p, y)$ are regular when $Q$ is equal to faces of $P_i$, $\Delta_i$, or $B_i$. Then by using 1-dimensional moduli spaces of the form ${\mc M}_{B_{i+1}}(A; \alpha, \ldots, \alpha, y)$, by Lemma \ref{lemmaa4}, one can see the following holds.
\begin{enumerate}

\item When $i = 2k$ is even, one has 
\beqn
\begin{split}
&\ QSt^{\rm Morse}_{p,i,A}(\alpha \otimes \cdots \otimes \alpha) - QSt^{\rm cell}_{p,i,A}(\alpha \otimes \cdots \otimes \alpha) \\
= &\ \sum_{y \in {\rm crit}(f)} \left( \#_p {\mc M}_{P_i}(A; \alpha, \ldots, \alpha, y) - \#_p {\mc M}_{\Delta_i}(A; \alpha, \ldots, \alpha, y) \right) \cdot y\\
= &\ \pm d \Big( \sum_{z \in {\rm crit} (f)} \#_p {\mc M}_{B_{2k+1}}(A; \alpha, \ldots, \alpha, z) \cdot z  \Big).
\end{split}
\eeqn 

\item When $i = 2k+1$ is odd, one has 
\beqn
\begin{split}
&\ QSt^{\rm Morse}_{p,i,A}(\alpha \otimes \cdots \otimes \alpha) - QSt^{\rm cell}_{p,i,A}(\alpha \otimes \cdots \otimes \alpha) \\
= &\ \sum_{y \in {\rm crit}(f)} \Big( \#_p {\mc M}_{P_i}(A; \alpha, \ldots, \alpha, y) - \#_p {\mc M}_{\Delta_i}(A; \alpha, \ldots, \alpha, y) \Big) \cdot y\\
= &\ \pm d \Big( \sum_{z \in {\rm crit} (f)} \#_p {\mc M}_{B_{2k+2}}(A; \alpha, \ldots, \alpha, z) \cdot z  \Big) \\
&\ + \sum_{y \in {\rm crit}(f)} \Big( \#_p {\mc M}_{\tau(B_{2k+1})}(A; \alpha, \ldots, \alpha, y) - \#_p {\mc M}_{B_{2k+1}} (A; \alpha, \ldots, \alpha, y) \Big) \cdot y.
\end{split}
\eeqn 
The last line vanishes by the $\zp$-invariance of the counts. 
\end{enumerate}
Therefore, the difference we are considering is a Morse coboundary.
\end{proof}

\begin{remark}
    We would like to point out that the definition of $QSt_p^{\rm cell}$ as recalled above differs from the sign convention in \cite[Equation (2.12)]{covariant-constant} because the generators $t^k \theta$ and $t^k$ of $\rp$ are chosen such that \begin{equation} \label{equation:choice} \langle t^k \theta, [\Delta_{2k+1}] \rangle = \langle t^k, [\Delta_{2k}] \rangle = (-1)^k  \end{equation}
    in \emph{loc.cit}. Our convention agrees with the one in \cite{jae-hee-lee}. For our application Corollary \ref{cor: generic new example}, because we work over $\mathbb{F}_2$, the two conventions agree.
\end{remark}




\section{Homological perturbation arguments}\label{appendixb}
\label{app: hp}

In this appendix we explain how to use homological perturbation arguments in order to prove Proposition \ref{prop31} and other useful results. For example we show that the $\Z/p$-equivariant Tate bar-length spectrum of $H^{(p)}$ coincides with the bar-length spectrum of $H$ stretched $p$-times, assuming that $(\phi^1_H)^p$ has finitely many contractible fixed points. (Note that we work in the class of contractible orbits. In the non-contractible case, a similar result can be proven, see \cite{Sugimoto}.)

\subsection{Crossing energy estimates}

\subsubsection{Crossing energy estimate for fixed data}

We state two lemmas about crossing energy estimates for both cylinders and pair-of-pants. Both are special cases of \cite[Proposition 7.1]{SZhao-pants} after translating to the fixed-point setting, and hence no longer need proofs. For an alternative treatment, see \cite[Section 2.3]{Sugimoto}.

\begin{lemma}\label{lemmab1}
Let $H_0$ be a Hamiltonian on $(M, \omega)$ with isolated 1-periodic orbits. Let $J_0$ be a compatible almost complex structure. Then there exists $\epsilon_0>0$ such that all nonconstant Floer cylinders for $(H_0, J_0)$ have energy no less than $\epsilon_0$. 
\end{lemma}

To state the case for  pair-of-pants, we need to recall the meaning of constant solutions. Let $H_0$ be a Hamiltonian on $(M, \omega)$ with an isolated set of 1-periodic orbits. For each $p$, recall that $\Sigma_p^\circ$ is the sphere punctured at $\infty$ as well as all $p$-th roots of unity. Then there is a $p$-fold branched cover $\pi_p: \Sigma_p^\circ \to {\mb R}\times S^1$. For each $x \in {\mc O}(H_0)$, let $u_x^p: \Sigma_p^\circ \to M$ be the pullback of the map 
\beqn
u_x: {\mb R}\times S^1 \to M,\ u_x(s, t) = x(t).
\eeqn
Then $\{ u_x^p\ |\ x \in {\mc O}(H_0)\}$ is a finite subset of $C^0 (\Sigma_p^\circ, M)$. Moreover, $u_x^p$ is a zero energy solution to the equation 
\beqn
(\nabla^{\sigma_0} u)^{0,1} = 0
\eeqn
where $\sigma_0$ is the flat Hamiltonian connection on $\Sigma_p^\circ$. 

\begin{lemma}\label{lemmab2}
Let $H_0$ be a Hamiltonian on $(M, \omega)$ such that both $H_0$ and $H_0^{(p)}$ have isolated fixed 1-periodic orbits. Let $J_0$ be a compatible almost complex structure. Then there exists $\epsilon_p>0$ such that for any solution $u: \Sigma_p^\circ \to M$ to $(\nabla^{\sigma_0} u)^{0,1} = 0$ other than $u_x^p$, $E(u) \geq \epsilon_p$.
\end{lemma}

\subsubsection{Crossing energy estimate for perturbed Floer data}

\begin{lemma}\label{lemmab3}
Let $(H_0, J_0)$ be as in Lemma \ref{lemmab1}. For each $x \in {\mc O}(H_0)$, let $u_x: {\mb R}\times S^1 \to M$ be the constant Floer cylinder at $x$. Then for any collection of disjoint open neighborhoods $U(u_{x_1}), \ldots, U(u_{x_m})$ for all 1-periodic orbits of $H_0$, the following is true.

Let $\sigma$ be a Hamiltonian connection on ${\mb R}\times S^1$ which is equal to $H_\pm dt$ near $\pm \infty$ where $H_\pm$ are two nondegenerate Hamiltonians such that $\sigma$ is sufficiently close to $H_0 dt$ in $C^\infty$ topology on ${\mb R}\times S^1 \times M$ and such that $\int_{{\mb R}\times S^1} | F_\sigma| ds dt$ is sufficiently small. Let $J_{s, t}$ be a domain-dependent almost complex structure which is sufficiently close to $J_0$ in $C^\infty$ topology. Let $u: {\mb R} \times S^1 \to M$ be a solution to 
\beqn
(\nabla^\sigma u)^{0,1} = 0.
\eeqn
Then 
\begin{enumerate}

\item Suppose $u$ is asymptotic to $x_\pm \in {\mc O}(H_0, x)$ for some $x \in {\mc O}(H)$. Moreover, suppose there exists a capping $\ov{x}$ (which induces cappings $\ov{x}_\pm$ such that the concatenation of $\ov{x}_-$ with $u$ gives the capping $\ov{x}_+$). Then $u$ is contained in the neighborhood $U(u_x)$.

\item If the above assumption is not true, then $E(u) \geq \frac{1}{2}\epsilon_0$ where $\epsilon_0$ is from Lemma \ref{lemmab1}.
\end{enumerate}
\end{lemma}

\begin{proof}
The proof is based on a compactness argument and Lemma \ref{lemmab1}. Suppose the statement is not true. Then one can find a sequence $\sigma_i$ of Hamiltonian connection converging to $H_0 dt$ and a sequence of domain-dependent almost complex structures $J_i$ converging to $J_0$, both in $C^\infty$-topology, as well as a sequence of solutions $u_i$ to the equation $(\nabla^{\sigma_i} u_i)^{0,1} = 0$ which do not satisfy the claims. By choosing a subsequence (still indexed by $i$), one can argue in the following two situations.

\begin{enumerate}
    \item Suppose there exists $x \in {\mc O}(H_0)$ such that for all $i$, $u_i$ is asymptotic to $x_{i, \pm} \in {\mc O}(H_{i, \pm}, x)$. Moreover, for a fixed capping $\ov{x}$ of $x$ with induced ones $\ov{x}_{i, \pm}$, the concatenation of $\ov{x}_{i, -}$ with $u_i$ coincides with $\ov{x}_{i, +}$. The condition on the integral of $|F_{\sigma_i}|$ implies that $E(u_i) \to 0$. Then Lemma \ref{lemmab1} and the usual compactness argument imply that $u_i$ converges uniformly to the constant Floer cylinder $u_x$. Hence for $i$ sufficiently large, $u_i$ is contained in a $C^0$-neighborhood of $u_x$.

    \item Suppose for all $i$ we are not in the above situation for any $x \in {\mc O}(H_0)$. By contradiction, assume that $E(u_i) \leq \frac{1}{2} \epsilon_0$. Again one can apply the compactness argument. Lemma \ref{lemmab1} implies that any subsequential limit must be a constant Floer cylinder $u_x$ and the convergence is uniform. This implies that for $i$ sufficiently large, one is indeed in the first situation. \qedhere
\end{enumerate}
\end{proof}

We have the following version for pair-of-pants. Since we do not treat continuation maps, one only needs to consider flat Hamiltonian connections.

\begin{lemma}\label{lemmab4}
Let $H_0, J_0$ and $p$ be as in Lemma \ref{lemmab2}. Then for any collection of disjoint neighborhoods $U(u_{x_1}^p), \ldots, U(u_{x_m}^p) \subset C^0(\Sigma_p^\circ, M)$ with respect to the $C^0$-topology, the following is true.

Let $G$ be a nondegenerate Hamiltonian which is sufficiently close to $H_0$ in $C^\infty$-topology. Let $J$ be a compatible almost complex structure which is sufficiently close to $J_0$ in $C^\infty$-topology. Let $\sigma$ be the flat Hamiltonian connection on $\Sigma_p^\circ$ which is the pullback of $G dt$ on ${\mb R}\times S^1$. Let $u: \Sigma_p^\circ \to M$ be a solution to $(\nabla^\sigma u)^{0,1} = 0$. 

\begin{enumerate}
    \item Suppose there is $x \in {\mc O}(H_0)$ such that at the $j$-th negative puncture, $u$ is asymptotic to some $y_j \in {\mc O}(G, x)$ and at the positive puncture $u$ is asymptotic to $y_\infty \in {\mc O}(G^{(p)}, x^{(p)})$. Moreover, assume that for a  capping $\ov{x}$ of $x$ (which induces cappings $\ov{y}_j$ and $\ov{y}_\infty$), the concatenation of $\ov{y}_1\sqcup \cdots \sqcup \ov{y}_p$ with $u$ gives the capping $\ov{y}_\infty$. Then the map $u$ is contained in $U(u_x^p) \subset C^0(\Sigma_p^\circ, M)$. 

    \item If the above assumption is not true, then one has 
    \beqn
    E(u) \geq \frac{1}{2} \min \{ \epsilon_p, \epsilon_0(H_0^{(p)})\}
    \eeqn
    where $\epsilon_p$ is from Lemma \ref{lemmab2} and $\epsilon_0(H_0^{(p)})$ is from Lemma \ref{lemmab1} while replacing $H_0$ by $H_0^{(p)}$.
\end{enumerate}
\end{lemma}

\begin{proof}
We prove by contradiction. For simplicity assume $J_i = J_0$; the general case can be treated with only notational complexity. Let $G_i$ be a sequence of nondegenerate Hamiltonians converging to $H_0$ in $C^\infty$-topology. Let $u_i: \Sigma_p^\circ \to M$ be a sequence of solutions to the Floer equation $(\nabla^{\sigma_i} u)^{0,1} = 0$ for $\sigma_i$ being the Hamiltonian connection induced from $G_i$. We assume that we are in the situation of either (1) or (2) for all $i$. 

Suppose we are in situation (1). By choosing a subsequence, we may assume that for the same $x \in {\mc O}(H_0)$, $u_i$ is asymptotic to an orbit $y_{i, j} \in {\mc O}(G_i, x)$ for the $j$-th negative puncture and asymptotic to $y_{i, \infty} \in {\mc O}(G_i^{(p)}, x^{(p)})$ for the positive puncture. Then the condition together with the energy identity for solutions implies that $E(u_i) \to 0$. The Gromov compactness implies that a subsequence of $u_i$ (still indexed by $i$) converges to a solution to $(\nabla^{\sigma_0} u)^{0,1} = 0$ on $\Sigma_p^\circ$, up to breakings at cylindrical ends. By Lemma \ref{lemmab2}, the limit is one of $u_{x}^p$. Lemma \ref{lemmab1} (applied to both $H_0$ and $H_0^{(p)}$) also excludes breakings. Hence the convergence $u_i \to u_x^p$ is uniform, implying that for $i$ sufficiently large, $u_i\in U(u_x^p)$.

Suppose we are in situation (2) for the sequence. If the conclusion is false, then for a subsequence of $u_i$ (still indexed by $i$) one has 
\beqn
E(u_i) \leq \frac{1}{2} \min \{\epsilon_p,\ \epsilon_0(H_0^{(p)})\}.
\eeqn
Then by the same compactness argument, a subsequence of $u_i$ converges up to breaking to a $u_x^p$ on $\Sigma_p^\circ$. Lemma \ref{lemmab1} and the energy bound exclude breakings at both positive and negative ends. The convergence is uniform. Hence for $i$ being sufficiently large one is indeed in situation (1), which is a contradiction. 
\end{proof}

\subsection{Proof of Proposition \ref{prop31}}

Choose a sequence of nondegenerate Hamiltonians $H_i$ converging in $C^\infty$ topology to $H$. Then for $i$ sufficiently large, there is a canonical decomposition
\beqn
{\mc O}(H_i) = \bigsqcup_{x \in {\mc O}(H)} {\mc O}(H_i, x)
\eeqn
where each ${\mc O}(H_i, x)$ is the ``cluster'' of contractible orbits of $H_i$ which are close to $x$. One has a similar decomposition for capped orbits
\beqn
\tilde {\mc O}(H_i) = \bigsqcup_{\ov{x} \in \tilde {\mc O}(H)} \tilde {\mc O}(H_i, \ov{x}).
\eeqn

Fix an $\omega$-compatible almost complex structure $J$. Consider the Floer complex $(CF(H_i),d_{H_i})$ defined using $J$ and an  auxiliary datum (see Theorem \ref{thm:Floer-cochain}). While $\ol{x}_{ij}$ has action $\cl A_{H_i}(\ol{x}_{ij})$ in $CF(H_i)$, we define a new action map on $CF(H_i)$ by setting
$$
\cl A(\ol{x}_{ij}) = \cl A_H(\ol{x}).
$$ 
This modification does not change the convergence condition required for Floer cochains in $CF(H_i)$. Moreover, for each $x \in {\mc O}(H)$ resp. $\ov{x} \in \tilde {\mc O}(H)$ let
\beqn
CF(H_i, x) \subset CF(H_i)\ {\rm resp.}\ CF(H_i, \ov{x}) \subset CF(H_i)
\eeqn
be the $\Lambda_{\mb Z}^\Gamma$-submodule resp. subgroup generated by orbits in the cluster ${\mc O}(H_i, x)$ resp. capped orbits in $\tilde {\mc O}(H_i, \ov{x})$. Then one has the direct sum decomposition
\beqn
CF(H_i) = \bigoplus_{x \in {\mc O}(H)} CF(H_i, x) = \widehat{\bigoplus_{\ov{x}\in \tilde {\mc O}(H)}} CF(H_i, \ov{x}).
\eeqn

\begin{lemma}\label{lemmab5}
There exists $\epsilon_H >0$ such that for $i$ sufficiently large, there is a decomposition
\begin{equation}\label{eq: d split} d_{H_i} = d_i^\loc + D_i,
\end{equation} 
satisfying the following conditions.
\begin{enumerate}
    \item For each capped orbit $\ov{x} \in \tilde {\mc O}(H)$, $d_i^\loc$ restricts to a differential $d_{i, \ov{x}}^\loc:  CF(H_i, \ov{x}) \to CF(H_i, \ov{x})$ whose homology is (canonically) isomorphic to the local Floer cohomology $HF^\loc(H, \ov{x}; {\mb Z})$.

\item $d_i^\loc$ preserves the filtration ${\mc A}_H$; on the other hand, for all $z \in CF(H_i)$, 
\beq\label{eq: cross energy}\cl A_H (D_i(z)) \geq  \cl A_H (z) + \eps_H
\eeq
\end{enumerate}
\end{lemma}

\begin{proof}
This lemma follows from Lemma \ref{lemmab3}. Let $\epsilon_H$ be $\frac{1}{4} \epsilon_0$ where $\epsilon_0$ is the one of Lemma \ref{lemmab1} with $H_0$ replaced by $H$. Choose a collection of sufficiently small $C^0$-neighborhoods $U(u_{x_j})$ for all $x_j \in {\mc O}(H)$. 

Then for each $i$, Lemma \ref{lemmab3} implies the dichotomy that each Floer cylinder of $H_i$ is either contained in $U(u_x)$ for some $x \in {\mc O}(H)$ or has energy at least $2\epsilon_H$. 
Hence one can define $d_i^\loc$ resp. $D_i$ to be the linear map coming from Floer cylinders belonging to the first resp. second category. By the definition of the modified action ${\mc A}_H$, one can see that $d_i^\loc$ preserves it; moreover, as ${\mc A}_H$ can be arbitrarily close to ${\mc A}_{H_i}$, one can see that when $i$ is large, $D_i$ increases the energy filtration by at least $\epsilon_H$. By this definition, one can see each $CF(H_i, \ov{x})$ is invariant under $d_i^\loc$. Applying the same dichotomy to Floer cylinders of index $1$, one can see that $d_i^\loc$ is a differential. Let $d_{i, \ov{x}}^\loc$ be the restriction of $d_i^\loc$ to $CF(H_i, \ov{x})$.

It remains to show that for each $\ov{x}$, the complex $(CF(H_i, \ov{x}), d_{i,\ov{x}}^\loc)$ is chain homotopic to any complex $CF^\loc (G, \ov{x}; {\mb Z})$ computing the local Floer cohomology $HF^\loc(H, \ov{x}; {\mb Z})$. Indeed, one can view $H_i$ as a local nondegenerate perturbation near the orbit $x$. Then the claim follows for the same argument as one shows that the local Floer cohomology is independent of the choices of local perturbations. 
\end{proof}


We proceed to the construction of $d_H$ in Proposition \ref{prop31}. It relies on the homological perturbation lemma. To start, we need a basic algebraic fact.

\begin{lemma}\label{lemmab6}
Let $R$ be a principal ideal domain and let $(V, d)$ be a ${\mb Z}$-graded cochain complex of free $R$-modules such that $H(V)$ is free. Then there exist cochain maps 
\begin{align*}
    &\ \sigma_V: H(V) \to V,\ &\ \pi_V: V \to H(V)
\end{align*}
and a chain homotopy
\beqn
\Theta_V: V \to V
\eeqn
such that 
\beqn
\begin{split}
\pi_V \circ \sigma_V = &\ {\rm Id}_{H(V)},\\
\sigma_V \circ \pi_V = &\ {\rm Id}_V + d \circ \Theta_V + \Theta_V \circ d,\\
\Theta_V^2 = &\ 0,\\
\Theta_V \circ \sigma_V = &\ 0,\\
\pi_V \circ \Theta_V = &\ 0.
\end{split}
\eeqn
\end{lemma}

\begin{proof}
For each degree $n$, let $Z^n \subset V^n$ be the kernel of differential and $B^n = d(V^{n-1})$. As $H(V)$ is free, the exact sequence 
\beqn
\xymatrix{ 0 \ar[r] & B^n \ar[r] & Z^n \ar[r] & H^n \ar[r] & 0 }
\eeqn
splits. Choose a splitting
\beqn
V^n = Z^n \oplus B^{n+1} = B^n \oplus H^n \oplus B^{n+1}.
\eeqn
Then define $\sigma_V$ to be the inclusion of $H^n$ with respect to this splitting and $\pi_V$ to be the corresponding projection. Then define $\Theta_V$ to be the identification of $B^n \subset V^n$ with $B^n \subset V^{n-1}$, extended by $0$ on $H^n \oplus B^{n+1}$. It is then easy to see that all the conditions are satisfied.
\end{proof}

Now change coefficients to $\fp$. By Lemma \ref{lemmab6}, one can choose a collection of chain homotopy equivalences
\beqn
\sigma_{i, \ov{x}}: HF^{\rm loc}(H, \ov{x}) \to CF^{\rm loc}(H_i, \ov{x})
\eeqn
which is compatible with Novikov actions. Namely, for each $A \in \Gamma$, the diagram
\beqn
\xymatrix{  HF^{\rm loc}(H_i, \ov{x}) \ar[r] \ar[d]_{\sigma_{i, \ov{x}}}   &    HF^{\rm loc}(H_i, A \cdot \ov{x}) \ar[d]^{\sigma_{i, A \cdot \ov{x}}} \\
            CF^{\rm loc}(H_i, \ov{x}) \ar[r]   &    CF^{\rm loc}(H_i, A \cdot \ov{x}) }
\eeqn
commutes. Here the horizontal arrows are induced by recapping by class $A$. Then by taking direct sums, we obtain two $\Lambda_{{\mb F}_p}^\Gamma$-linear maps
\begin{align*}
&\ \sigma_i: C(H) \to CF(H_i),\ &\ \pi_i: CF(H_i) \to C(H),
\end{align*}
with $\Theta_i: CF(H_i) \to CF(H_i)$ such that the triple $(\sigma_i, \pi_i, \Theta_i)$ satisfies the relations listed in Lemma \ref{lemmab6}. Denote by $d_{H_i}: CF(H_i) \to CF(H_i)$ the Floer differential. Then following the homological perturbation formula (see \cite{Huebschman_Kadeishvili} and \cite{Markl_perturbation}) we define
\beq\label{hp_formula}
\begin{split}
d_H:= &\ \pi_i \left( \sum_{l=0}^\infty (d_{H_i} \Theta_i)^l \right) d_{H_i} \sigma_i,\\
\pi_H:= &\ \pi_i \left( \sum_{l=0}^\infty ( d_{H_i} \Theta_i)^l\right),\\
\sigma_H:= &\ \left( \sum_{l = 0}^\infty (\Theta_i d_{H_i} )^l \right) \sigma_i,\\
\Theta_H:= &\ \Theta_i \left( \sum_{l=0}^\infty ( d_{H_i} \Theta_i)^l \right).
\end{split}
\eeq
One can check that $d_H^2 = 0$, $\pi_H: CF(H_i) \to C(H)$ and $\sigma_H: C(H) \to CF(H_i)$ are cochain maps, and 
\begin{align*}
&\ \pi_H \circ \sigma_H = {\rm Id},\ &\ \sigma_H \circ \pi_H - {\rm Id} = d_H \circ \Theta_H + \Theta_H \circ d_H.
\end{align*}
Then conditions listed in Proposition \ref{prop31} are readily checked. Indeed, the transferred differential $d_H$ is necessarily induced by Floer differentials from Floer trajectories with energy $\geq \epsilon_H$. Moreover, as $\pi_H$ and $\sigma_H$ are inverses up to chain homotopy, invariance of Floer theory implies that $C(H)$ is chain homotopy equivalent to $CF(G)$ for any nondegenerate $G$. Therefore, (1) and (2) of Proposition \ref{prop31} are proved. 

Before we continue, we need the following Lemma.

\begin{lemma}\label{lemmab7}
Given any two such sequences $H_i$ and $H_i'$, for $i$ sufficiently large, there exists a chain homotopy equivalence between $CF(H_i)$ and $CF(H_i')$ such that the modified action filtration is preserved. 
\end{lemma}

\begin{proof}
When $i$ is sufficiently large, consider a continuation map 
\beqn
\Phi_i: CF(H_i) \to CF(H_i')
\eeqn
induced by a homotopy from $H_i$ to $H_i'$ which remains to be $C^\infty$-close to $H$. Then, using Lemma \ref{lemmab3}, we can decompose
\beqn
\Phi_i = \Phi_{\rm small} + \Phi_{\rm big},
\eeqn
where $\Phi_{\rm small}$ is the contribution from the continuation Floer cylinders whose energy is strictly less than $2 \epsilon_H$. One observes that, for each $\ov{x} \in \tilde {\mc O}(H)$, $\Phi_{\rm small}$ restricts to a chain homotopy equivalence between local Floer chain complexes. It follows from the construction of the modified action that $\Phi_i$ preserves the action filtration.
\end{proof}

Now we prove the last item of Proposition \ref{prop31}, namely, the filtered complex $C(H)^{(a, b)}$ calculates the filtered Floer homology of $HF(H)^{(a, b)}$ when $a, b \notin {\rm Spec}(H)$, defined by \eqref{eq: filtered-irrational}. Notice that it suffices to prove that 
\beqn
HF(C(H)^{>a}) \cong HF(H)^{>a}.
\eeqn
Recall that $HF(H)^{>a}$ is defined as the colimit 
\beqn
\varinjlim  HF(K)^{>a}
\eeqn
where the colimit is taken over all nondegenerate $K$ satisfying $K < H$ and $a \notin {\rm Spec}(K)$ with respect to the partial order defined by pointwise comparison. Denote this partially ordered set by $P_a (H)$. It is easy to see that one can choose a cofinal sequence $H_i \in P_a (H)$ with $H_i \nearrow H$ uniformly. Then for $i$ sufficiently large, the Floer cochain complex $CF(H_i)$ admits a modified energy filtration 
\beqn
{\mc A}_H: CF(H_i) \to {\mb R} \cup \{+\infty\}.
\eeqn
Moreover, there exists a sequence $\delta_i>0$ satisfying
\beqn
{\mc A}_{H_i}(x) - \delta_i \leq {\mc A}_H(x)\ \  \forall x \in CF(H_i).
\eeqn
Then there exists a natural inclusion map
\beqn
CF(H_i)^{{\mc A}_{H_i}>a + \delta_i} \hookrightarrow CF(H_i)^{{\mc A}_H>a}.
\eeqn
Denote $H_i' = H_i - \delta_i$, which still form a cofinal sequence in $P_a (H)$. Then the above map is the same as the map
\beqn
CF(H_i')^{{\mc A}_{H_i'}>a} \to CF(H_i)^{{\mc A}_H>a}
\eeqn
(Notice that $CF(H_i') = CF(H_i)$.) We can choose a subsequence such that 
\beqn
i<j \Longrightarrow H_i' < H_j'.
\eeqn
Then there are well-defined continuation maps induced by any monotone homotopy
\beqn
HF(H_i')^{{\mc A}_{H_i'}>a} \to HF(H_j')^{{\mc A}_{H_j'}>a}
\eeqn
On the other hand, for the modified action ${\mc A}_H$, by Lemma \ref{lemmab7}, for any $i<j$, there is also a well-defined continuation map
\beqn
HF(H_i)^{{\mc A}_H>a} \to HF(H_j)^{{\mc A}_H>a}.
\eeqn

\begin{lemma}
The direct limit map
\beqn
\varinjlim HF(H_i')^{{\mc A}_{H_i'}>a} \to \varinjlim HF(H_i)^{{\mc A}_H>a}
\eeqn
is an isomorphism.
\end{lemma}

\begin{proof}
The chain-level inclusion $CF(H_i')^{{\mc A}_{H_i'} >a} \to CF(H_i)^{{\mc A}_H>a}$ induces the long exact sequence
\beqn
\xymatrix{ HF(H_i)^{{\mc A}_{H_i'} \leq a < {\mc A}_H}   \ar[r] &    HF(H_i)^{{\mc A}_{H_i'} > a} \ar[r] & HF(H_i)^{{\mc A}_H>a} \ar[r] & HF(H_i)^{{\mc A}_{H_i'}\leq a < {\mc A}_H}}. 
\eeqn
By the exactness of the direct limit functor, it suffices to show that 
\beqn
\varinjlim HF(H_i)^{{\mc A}_{H_i'} \leq a < {\mc A}_H} = 0.
\eeqn
Indeed, given any $[c_i] \in HF(H_i)^{{\mc A}_{H_i'} \leq a < {\mc A}_H}$, there exists $\delta>0$ such that 
\beqn
{\mc A}_H([c_i]) > a + \delta.
\eeqn
Moreover $[c_i]$ can be represented by a cocycle 
\beqn
c_i \in CF(H_i)^{{\mc A}_H>a + \delta}.
\eeqn
Then for $j>i$ sufficiently large, the monotone continuation map from $CF(H_i)$ to $CF(H_j)$ sends $c_i$ to a cocycle $c_j$ with ${\mc A}_{H_j}(c_j) > a + \delta_j \Longrightarrow {\mc A}_{H_j'}(c_j) > a$. Hence $[c_i]$ cannot survive the direct limit. 
\end{proof}

However, notice that $H_i'$ is a cofinal sequence in $P_a (H)$, while $HF(H_i)^{{\mc A}_H>a}$ are all isomorphic to $H(C(H)^{>a})$. Hence
\beqn
HF(H)^{>a} \cong HF(C(H)^{>a})
\eeqn
for all $a \notin {\rm Spec}(H)$. The case of $HF(H)^{(a, b)}$ follows from a small modification of the above argument. Hence the third item of Proposition \ref{prop31} is proved. 

\begin{remark}
We note that the complex $(C(H), {\mc A}_H)$ is actually well-defined up to filtered chain homotopy equivalence despite the auxiliary choices made in the proof. Indeed, given two sequences $H_i, H_i'$, for sufficiently large $i$, one has filtered chain homotopy equivalences
\beqn
C(H, H_i) \simeq CF(H_i) \simeq CF(H_i') \simeq C(H, H_i'),
\eeqn
where the homotopy equivalence $CF(H_i) \simeq CF(H_i')$ is defined via a continuation map and they are equipped with the modified action filtration.
\end{remark}

\subsection{Proof of Proposition \ref{prop33b}}\label{subsectionb3}


The construction is similar to that of Proposition \ref{prop31}. Choose a sequence of nondegenerate Hamiltonians $H_i$ which converges to $H$ such that $H_i^{(p)}$ is also nondegenerate. As we know ${\rm Fix}_c(\phi) = {\rm Fix}_c(\phi^p)$, when $H_i$ is sufficiently close to $H$, all capped 1-periodic orbits of $H_i^{(p)}$ are close to $p$-th iterations of elements of $\tilde {\mc O}(H)$. Then one has the $\zp$-equivariant Floer cochain complex
\beqn
CF_{\zp} (H_i^{(p)}) = \Big( CF(H_i^{(p)}) \widehat{\otimes} CM(g) \Big)^\inv
\eeqn
which also depends on certain choices. Recall that this cochain complex has an energy filtration determined by the usual symplectic action ${\mc A}_{H_i^{(p)}}$. As each generator $\ov{y}$ of $CF(H_i^{(p)})$ is close to a unique iteration $\ov{x}^{(p)} \in \tilde{\mc O}(H^{(p)})$, we redefine the filtration by 
\beqn
{\mc A}_{H^{(p)}}( \ov{y} ) = {\mc A}_{H^{(p)}}(\ov{x}^{(p)}) = p{\mc A}_H(\ov{x}).
\eeqn
Once again this modification does not change the convergence requirement for cochains in $CF_{\zp}(H_i^{(p)})$. 

Next, by Lemma \ref{lemmab3} with $H_0$ replaced by $H^{(p)}$, for $i$ sufficiently large, the differential $d_{i, \zp}$ decomposes as 
\beqn
d_{i, \zp} = d_{i, \zp}^{\rm loc} + D_{i,\zp}
\eeqn
where $d_{i, \zp}^{\rm loc}$ preserves the modified action and $D_{i, \zp}$ increases the energy filtration. Then Lemma \ref{lemmab3} applied to index 1 moduli spaces implies that $d_{i, \zp}^\loc$ restricts to a differential on 
\beqn
\Big( CF(H_i^{(p)}, \ov{x}^{(p)}) \otimes CM(g) \Big)^\inv.
\eeqn
It is easy to see that the cohomology defined by $d_{i, \zp}^{\rm loc}$ is the local equivariant Floer cohomology $HF_\zp^{\rm loc}(H^{(p)}, \ov{x}^{(p)})$. Now define the new cochain group
\beqn
C_{\zp} (H^{(p)}):= \bigoplus_{x \in {\mc O}(H)} HF_{\zp}^{\rm loc} ( H^{(p)}, x^{(p)}).
\eeqn

To define the differential on the above cochain group, we need to apply the homological perturbation method. We assume that all fixed points of $\phi$ are $p$-admissible. Then by Proposition \ref{prop_local_eq_structure}, $HF_\zp^{\rm loc}(H^{(p)}, \ov{x}^{(p)})$ is a free $\kzp$-module. Then by Lemma \ref{lemmab6}, the complex $(CF_{\zp}(H_i^{(p)}), d_{i, \zp}^\loc)$ is chain homotopy equivalent to $(\oplus_{x \in {\mc O}(H)} HF_{\zp}^{\loc}(H^{(p)}, \ov{x}{}^{(p)}), 0)$. One can choose the chain homotopies such that they respect recapping of orbits. The collection of such chain homotopies can be labelled as in the following diagram
\beqn
\xymatrix{  & C_{\zp}(H_i^{(p)})   \ar@<.5ex>@{>}[r]^{\pi}  \ar@(ur,ul)[]_{\Theta}   &     C_{\zp}(H^{(p)}) \ar@<.5ex>@{>}[l]^{\sigma}   \ar@(ur,ul)[]_{0}   &  }
\eeqn

Now we apply the homological perturbation construction. Treating $D_{i, \zp}$ as a perturbation of $d_{i, \zp}^\loc$, the complex $CF_{\zp}(H_i^{(p)})$ is then chain homotopic to a complex 
\beqn
C_{\zp} (H^{(p)}) = \left( \bigoplus_{x \in {\mc O}(H)} HF_\zp^{\rm loc} (H^{(p)}, x^{(p)}), d_{\zp} \right)
\eeqn
where $d_{\zp}$ is linear over $\Lambda_{\kzp}^\Gamma$. By the homological perturbation formula \eqref{hp_formula}, $d_{\zp}$ increases the energy filtration by at least a positive constant $\epsilon_H$.

Properties listed in Proposition \ref{prop33b} follow immediately from the construction.

\begin{proof}[Proof of Lemma \ref{lemma37}]
We go through the homological perturbation in a slightly different way. Let $G$ be a nondegenerate perturbation of $H$ which is sufficiently close such that $G^{(p)}$ is also nondegenerate. Consider the equivariant Floer complex $CF_{\zp}(G^{(p)})$. By the discussion prior to Theorem \ref{thm213}, one can identify
\beqn
CF_{\zp}(G^{(p)}) \cong CF(G^{(p)}) \otimes \rp \cong \Big( CF(G^{(p)}) \oplus CF(G^{(p)}) \otimes \theta \Big) \otimes \kzp.
\eeqn
Moreover, by Theorem \ref{thm213}, one can write the equivariant differential as 
\beqn
d_{\zp} = \left[ \begin{array}{cc} d_{G^{(p)}} & 0 \\ S_p & d_{G^{(p)}} \end{array} \right] + t K.
\eeqn
If we ignore the terms $S_p$ and $tK$, then via the homological perturbation construction given in the proof of Proposition \ref{prop31}, with $H$ replaced by $H^{(p)}$, one obtains copies of the complex $C(H^{(p)})$ with a differential $d_{H^{(p)}}$ which strictly increases the energy. One could, in addition, view the nilpotent term $S_p$ and the term  $tK$ having positive $t$ degrees as perturbations. Then one obtains a complex, denoted temporarily by 
\beqn
\wt{C}_{\zp}(H^{(p)}) \cong \Big( C(H^{(p)}) \oplus C(H^{(p)}) \otimes \theta \Big) \otimes \kzp
\eeqn
whose differential has the form 
\beqn
\wt{d}_{\zp} = \left[ \begin{array}{cc} d_{H^{(p)}} & 0 \\ \wt S_p & d_{H^{(p)}}    \end{array} \right] + t \wt K.
\eeqn
Further, it has the local part and ``high energy'' part, while the local part computes the local equivariant Floer cohomology. As $d_{H^{(p)}}$ has vanishing local parts, the local part of $\wt{d}_{\zp}$ reads
\beqn
\wt d_{\zp}^\loc  = \left[ \begin{array}{cc} 0 &  0 \\ \wt S_p^\loc & 0  \end{array} \right] + t\wt K^\loc.
\eeqn
However, one knows that the local equivariant cohomology is isomorphic to $HF^\loc(H^{(p)}, x^{(p)}) \otimes \rp$. Hence $\wt S_p^\loc$ and $\wt K^\loc$ both vanish. This implies that the complex $\wt C_{\zp}(H^{(p)})$ is indeed a case of the complex $C_{\zp}(H^{(p)})$. Moreover, the difference $d_{\zp} - d_{H^{(p)}}$ strictly increases the energy filtration. Therefore, the bar-length spectrum of the complex $C_{\zp}(H^{(p)})$, if not empty, has no zero bars. Then by the forms of the ordinary and equivariant differentials, the assertion follows.
\end{proof}

\subsection{Proof of Proposition \ref{prop35}}\label{proof_prop35}

Choose a sequence of nondegenerate $H_i$ converging to $H$ as in the proofs of Proposition \ref{prop31} and Proposition \ref{prop33b}. For each $i$, by the general construction of \cite{Bai_Xu_2025}, one obtains the chain-level equivariant pair-of-pants product
\beqn
\wt P_i:= \wt\fst_p: CF(H_i)^{\otimes p}_{\zp} \to CF_{\zp}(H_i^{(p)}; \Lambda_{\kzp}^\Gamma)
\eeqn
while the base field is changed to $\fp$. 
Now by the homological perturbation construction for $C(H)$ and $C_{\zp}(H^{(p)})$, one has chain homotopy equivalences
\begin{align*}
&\ \wt \pi_i: C(H_i) \to C(H),\ &\ \wt\sigma_i: C(H) \to C(H_i)
\end{align*}
and
\begin{align*}
&\ \wt \pi_{i, \zp}: CF_{\zp}(H_i^{(p)}) \to C_{\zp}(H^{(p)}),\ &\ \wt\sigma_{i, \zp}: C_{\zp}(H^{(p)}) \to CF_{\zp}(H_i^{(p)}).
\end{align*}
The cochain maps $\wt\pi_i$ and $\wt\sigma_i$ induce cochain maps
\begin{align*}
&\ \wt\pi_i^{\otimes p}: C(H_i)^{\otimes p}_{\zp} \to C(H)^{\otimes p}_{\zp},\ &\ \wt\sigma_i^{\otimes p}: C(H)^{\otimes p}_{\zp} \to C(H_i)^{\otimes p}_{\zp}.
\end{align*}
Then we define $\wt P$ by the diagram
\beqn
\vcenter{ \xymatrix{   C(H)^{\otimes p}_{\zp} \ar[r]^{\wt P} \ar[d]_{\wt\sigma_i^{\otimes p}}   &  C_{\zp}(H^{(p)})   \\
            CF(H_i)^{\otimes p}_{\zp} \ar[r]_{\wt P_i}   &    CF_{\zp}(H_i^{(p)})  \ar[u]_{\wt\pi_{i,{\zp}}} } }.
\eeqn
We claim that for $i$ sufficiently large, $\wt P$ defined in this way satisfies the conditions stated in Proposition \ref{prop35}. 

First, because $\wt P$ is transfered from $\wt P_i$ via the chain homotopy equivalences produced by homological perturbation, and because $\wt P_i$ induces the equivariant pair-of-pants product on the homology level, the first assertion of Proposition \ref{prop35} holds. 

Now we describe the decomposition $\wt P = \wt P^\loc + \wt P^{\rm big}$. For each $\ov{x} \in \tilde {\mc O}(H)$, recall one has an $\fp$-summand  
\beqn
CF^\loc (H_i, \ov{x}) \subset CF(H_i)
\eeqn
generated by $\tilde {\mc O}(H_i, \ov{x})$ and an $\fp$-summand
\beqn
CF_\zp^\loc (H_i^{(p)}, \ov{x}^{(p)}):= \Big( CF(H_i^{(p)}, \ov{x}^{(p)}) \otimes CM(g) \Big)^\inv  \subset \Big( CF(H_i^{(p)}) \otimes CM(g)\Big)^\inv = CF_\zp(H_i^{(p)})
\eeqn
generated by $\tilde {\mc O}(H_i^{(p)}, \ov{x}^{(p)})$. Then by the crossing energy lower bound (Lemma \ref{lemmab4}), when $i$ is sufficiently large, one can decompose the map
\beqn
\wt{P}_i: CF(H_i)^{\otimes p}_\zp \to CF_\zp (H_i^{(p)})
\eeqn
as $\wt P_i = \wt{P}_i^\loc + \wt{P}_i^{\rm big}$ where $\wt P_i^\loc$ restricts to a map
\beqn
\wt P_{i, \ov{x}}^\loc: CF^\loc (H_i, \ov{x})^{\otimes p}_{\zp} \to CF_\zp^\loc (H_i^{(p)}, \ov{x}^{(p)})
\eeqn
which preserves the modified energy filtration, while $\wt P_i^{\rm big}$ increases the energy filtration by at least a certain positive constant which is independent of sufficiently large $i$ and other choices. As $\wt\sigma_i$ and $\wt\pi_{i,{\zp}}$ strictly preserve the energy filtration and are defined near each individual orbit $\ov{x}^{(p)}$, we define
\begin{align*}
&\ \wt P^\loc:= \wt \pi_{i,\zp} \circ \wt P_i^\loc \circ \wt \sigma_i^{\otimes p},\ &\ \wt P^{\rm big}:= \wt \pi_{i,\zp} \circ \wt P_i^{\rm big}  \circ \wt \sigma_i^{\otimes p}.
\end{align*}
Then the second assertion of Proposition \ref{prop35} holds. The third assertion follows from the construction and the fact that $\wt\pi_{i, \zp}$ and $\wt \sigma_i^{\otimes p}$ both preserve the energy filtration.

\section{Floer theory for Morse--Bott Hamiltonians and Hamiltonian torsion}\label{app:MB}

Hamiltonian Floer cohomology for degenerate Hamiltonians of Morse--Bott type has been constructed in \cite{Ruan_Tian_Bott}, while a Lagrangian formulation was provided in \cite{Pozniak}. The paper \cite{AS-torsion} discussed certain features of local Floer cohomology of such Hamiltonians while assuming that the symplectic form is rational. In this paper, in order to establish Theorem \ref{thm_torsion} in full generality, we provide a comprehensive discussion without any topological assumptions. 

\begin{defn}\cite[Definition 1.1]{Ruan_Tian_Bott}
Let $(M, \omega)$ be a closed symplectic manifold. Consider a Hamiltonian diffeomorphism $\phi\in \Ham(M, \omega)$ and a connected component ${\mc F}$ of ${\rm Fix}_c(\phi)$. The connected component ${\mc F}$ is said to be of {\bf Morse--Bott type}\footnote{Notice that this condition is stronger than ${\rm graph}(\phi)$ intersects cleanly with the diagonal. The latter does not imply the equality between the eigenspace and generalized eigenspace of the eigenvalue $1$.} if ${\mc F}$ is a closed  submanifold and for each $x \in {\mc F}$
\beqn
T_x {\mc F} = {\rm Ker} (d\phi_x - {\rm Id}) = {\rm Ker} \big( (d\phi_x - {\rm Id})^{{\rm dim}_{\mb R} M} \big) \subset T_x M.
\eeqn
\end{defn}

Notice that if $\phi \in {\rm Ham}(M, \omega)$ has finite order, then any connected component ${\mc F} \subset {\rm Fix}_c(\phi)$ is necessarily of Morse--Bott type. 

In this appendix, we would like to do the following.
\begin{enumerate}

    \item Construct the local Hamiltonian Floer cohomology $HF(\phi, {\mc F})$.

    \item When $\phi$ is a finite-order element of ${\rm Ham}(M, \omega)$ whose order is not divisible by a prime $p$, construct and compute the local equivariant Floer cohomology  $HF_{\zp}^{\loc}(\phi^{(p)}, {\mc F})$.

    \item Under the above setting, construct the local equivariant pair-of-pants product 
    \beqn
    \fst_{p, {\mc F}}^\loc: HF^{\loc}(\phi, {\mc F}) \to HF^{\loc}_{\zp}(\phi^{(p)}, {\mc F})
    \eeqn
    and identify it with the classical Steenrod operation on ${\mc F}$.

    \item Establish analogues of homological perturbation construction in Section \ref{section3}.
\end{enumerate}

\subsection{Crossing energy estimates in the Morse--Bott case}

We establish a few important estimates which allow us to define corresponding algebraic objects. One first has an exponential decay estimate, which does not hold for general degenerate Hamiltonians.

For simplicity, our estimates are always referring to a fixed compatible almost complex structure $J$. It is not hard to remove the dependence by considering 1-parameter families of almost complex structures. The following statement is standard.

\begin{lemma}\label{lemmac2}
Let $H$ be a Hamiltonian on $(M, \omega)$ generating a diffeomorphism $\phi$ with Morse--Bott fixed points. Then for each $\omega$-compatible almost complex structure $J$, there exists $\epsilon_0 (H, J) >0$ and $\delta_0 >0$ satisfying the following conditions. Let $u: [0, +\infty) \times S^1 \to M$ be a finite energy solution to the Floer equation for $(H, J)$. Suppose $E(u) \leq \epsilon_0(H, J)$, then one has 
\beqn
\limsup_{s \to +\infty} e^{\delta_0 s} | \partial_s u(s, t)| < +\infty.
\eeqn
\end{lemma}

The exponential decay estimate implies the lower bound of energy of nontrivial solutions.

\begin{cor}\label{corc3}
Let $H$ and $J$ be as above. Then there exists $\hbar_0(H, J) >0$ such that for any solution $u: {\mb R}\times S^1 \to M$ to the Floer equation for $(H, J)$, one has 
\beqn
E(u) > 0 \Longrightarrow E(u) \geq \hbar_0(H, J).
\eeqn
\end{cor}

Then one can obtain the energy lower bound for nontrivial pair-of-pants. Recall that for each 1-periodic orbit $x \in {\mc O}(H)$ and a prime $p$, there is a solution $u_x^p: \Sigma_p^\circ \to M$ to the equation 
\beqn
(\nabla^\sigma u)^{0,1} = 0
\eeqn
for any $J$ which has zero energy. 

\begin{lemma}\label{lemmac4}
Let $H, J$ be as above and $p$ be a prime. Assume that $H^{(p)}$ also generates a Hamiltonian diffeomorphism with Morse--Bott fixed points. Let $\sigma$ be the flat Hamiltonian connection on $\Sigma_p^\circ$ which is the pullback of $H dt$ on ${\mb R}\times S^1$. Then there exists $\hbar_p(H, J)>0$ such that for any solution $u: \Sigma_p^\circ \to M$ to $(\nabla^{\sigma} u)^{0,1} = 0$ which is not $u_x^p$ for any $x \in {\mc O}(H)$, then $E(u) \geq \hbar_p(H, J)$.
\end{lemma}

\begin{proof}
We prove by contradiction. Suppose there exists a sequence of solutions $u_i: \Sigma_p^\circ \to M$ with $0 \neq E(u_i) \to 0$, then one can apply the compactness argument. First, the energy bound implies that a subsequence of $u_i$ (still indexed by $i$) converges in $C_\loc^\infty$ to an energy zero solution $u_\infty: \Sigma_p^\circ \to M$, hence must coincide with some $u_x^p$. Corollary \ref{corc3} implies that the convergence is uniform on $\Sigma_p^\circ$ without breaking. Hence for $i$ sufficiently large, $u_i$ is in a $C^0$-neighborhood of $u_x^p$. However, as the Hamiltonian connection is flat, the energy $E(u_i)$ is topological. Hence $E(u_i) = 0$ which is a contradiction.
\end{proof}

Then one can derive analogues of Lemma \ref{lemmab3} and Lemma \ref{lemmab4} in the Morse--Bott case. We first need some notations. Let $H$ generate a Hamiltonian diffeomorphism $\phi$ with Morse--Bott fixed points. Then for any nondegenerate $G$ which is sufficiently $C^2$-close to $H$, one has well-defined partitions
\begin{align*}
    &\ {\mc O}(G) = \bigsqcup_{{\mc F}\in \pi_0({\mc O}(H))} {\mc O}(G, {\mc F}),\ &\ \tilde {\mc O}(G) = \bigsqcup_{\ov{\mc F} \in \pi_0(\tilde {\mc O}(H))} \tilde {\mc O}(G, \ov{\mc F})
\end{align*}
where ${\mc O}(G, {\mc F})$ resp. $\tilde {\mc O}(G, \ov{\mc F})$ is the ``cluster'' of 1-periodic orbits resp. capped 1-periodic orbits of $G$ which is close to ${\mc F}$ resp. $\ov{\mc F}$. Moreover, if $x \in {\mc O}(G, {\mc F})$, then for any capping $\ov{\mc F}$, there is a canonically induced capping $\ov{x} \in \tilde {\mc O}(G, \ov{\mc F})$.

\begin{lemma}\label{lemmac5}
Let $H$ and $J$ be as in Lemma \ref{lemmac2}. Then for any open neighborhoods ${\mc U} \subset C^0({\mb R}\times S^1, M)$ of 
\beqn
\big\{ u_x\ |\ x \in {\mc O}(H)  \big\} \subset C^0( {\mb R}\times S^1, M)
\eeqn
the following is true.

Let $\sigma$ be a Hamiltonian connection on ${\mb R}\times S^1$ which is equal to $G_\pm dt$ near $\pm \infty$ where $G_\pm$ are two nondegenerate Hamiltonians which are sufficiently close to $H$ in the $C^\infty$-topology such that $\int_{{\mb R}\times S^1} |F_\sigma| ds dt$ is sufficiently small. Let $u: {\mb R}\times S^1 \to M$ be a solution to $(\nabla^\sigma u)^{0,1} = 0$ with respect to $J$. Then the following is true.
\begin{enumerate}
    \item Suppose $u$ is asymptotic to $x_\pm \in {\mc O}(G_\pm, {\mc F})$ for some component ${\mc F}$. Moreover, suppose there exists a capping $\ov{\mc F}$ (which induces cappings $\ov{x}_\pm$ such that the concatenation of $\ov{x}_-$ with $u$ gives the capping $\ov{x}_+$). Then $u$ is contained in the neighborhood ${\mc U}$.

    \item If the above assumption is not true, then $E(u_i) \geq \frac{1}{2}\hbar_0(H, J)$ where $\hbar_0(H, J)$ is from Corollary \ref{corc3}.
\end{enumerate}
\end{lemma}

\begin{proof}
The proof can be carried out  verbatim as that of Lemma \ref{lemmab3}.
\end{proof}

Now we state and prove the analogue of Lemma \ref{lemmab4}. Recall that for each $x \in {\mc O}(H)$, for each prime $p$, there is the ``trivial pair-of-pants'' $u_x^p: \Sigma_p^\circ \to M$ at $x$.

\begin{lemma}\label{lemmac6}
Let $H, J$ and $p$ be as in Lemma \ref{lemmac4}. Then for any neighborhood ${\mc U}_p \subset C^0(\Sigma_p^\circ, M)$ of 
\beqn
\{ u_x^p\ |\ x \in {\mc O}(H)\}
\eeqn
with respect to the $C^0$-topology, the following is true. 

Let $G$ be a nondegenerate Hamiltonian which is sufficiently close to $H$ in $C^\infty$-topology. Let $\sigma$ be the flat Hamiltonian connection on $\Sigma_p^\circ$ which is the pullback of $G dt$ on ${\mb R}\times S^1$. Let $u: \Sigma_p^\circ \to M$ be a solution to $(\nabla^\sigma u)^{0,1} = 0$ with respect to $J$. Then
\begin{enumerate}
    \item Suppose there is a component ${\mc F}\subset {\mc O}(H)$ such that at the $j$-th negative puncture $u$ is asymptotic to some $y_j \in {\mc O}(G, {\mc F})$ and at the positive puncture $u$ is asymptotic to some $y_\infty \in {\mc O}(G^{(p)}, {\mc F}^{(p)})$. Moreover, assume that for a capping $\ov{\mc F}$ of ${\mc F}$ (which induces cappings $\ov{y}_j$ and $\ov{y}_\infty$), the concatenation of $\ov{y}_1 \sqcup \cdots \sqcup \ov{y}_p$ with $u$ gives the capping $\ov{y}_\infty$. Then the map $u$ is contained in ${\mc U}_p$.

    \item If the above assumption is not true, then one has 
    \beqn
    E(u) \geq \frac{1}{2} \min \Big\{ \hbar_0(H, J), \hbar_0(H^{(p)}, J), \hbar_p(H, J) \Big\}
    \eeqn
    where $\hbar_0$ is from Corollary \ref{corc3} and $\hbar_p$ is from Lemma \ref{lemmac4}.

\end{enumerate}
\end{lemma}

\begin{proof}
The proof can be carried out  verbatim as that of Lemma \ref{lemmab4}.
\end{proof}

\subsection{Local Floer cohomology in the Morse--Bott case}

Start with $\phi \in {\rm Ham}(M,\omega)$ with a Morse--Bott component ${\mc F} \subset {\rm Fix}_c(\phi)$. Let $H$ be a Hamiltonian generating $\phi$. Then ${\mc F}$ can be identified with a component of ${\mc O}(H)$. Choose a $C^\infty$-close nondegenerate perturbation $G$ of $H$. Choose a capping $\ov{\mc F}$. Then one can define the cochain group
\beqn
CF^\loc(G, \ov{\mc F}; {\mb Z})
\eeqn
freely generated by elements in the cluster $\tilde {\mc O}(G, \ov{\mc F})$. To define the differential, choose an $\omega$-compatible almost complex structure $J$; we will not prove the independence of the local Floer cohomology from the choice of $J$. Then by Theorem \ref{thm:Floer-cochain}, upon making additional choices, one has the complex $CF(G; \Lambda_{\mb Z}^\Gamma)$ with differential $d_G$. Moreover, by the crossing energy estimate for cylinders, i.e. Lemma \ref{lemmac5}, one can decompose
\beqn
d_G = d_G^\loc + d_G^{\rm big}
\eeqn
where $d_G^\loc$ is defined via counting ``local'' Floer cylinders whose energy is arbitrarily small and $d_G^{\rm big}$ is defined via counting ``big'' energy cylinders. By using Lemma \ref{lemmac5} for index one Floer cylinders, one can see that $d_G^\loc$ is a differential and restricts to a differential
\beqn
d_{G,{\ov{\mc F}}}^\loc: CF^\loc (G, \ov{\mc F}; {\mb Z}) \to CF^\loc (G, \ov{\mc F}; {\mb Z}).
\eeqn
As $d_G$ is linear over $\Lambda^\Gamma$, one also obtains the differential
\beqn
d_{G, {\mc F}}^\loc: CF^\loc (G, {\mc F}; \Lambda_{\mb Z}^\Gamma ) \to CF^\loc (G, {\mc F}; \Lambda_{\mb Z}^\Gamma)
\eeqn
where $CF^\loc (G, {\mc F}; \Lambda_{\mb Z}^\Gamma) \subset CF(G; \Lambda_{\mb Z}^\Gamma)$ is the $\Lambda^\Gamma_{\mb Z}$-submodule generated by orbits near ${\mc F}$. 

One can use a continuation map to prove the independence on the nondegenerate perturbation $H$. If $G'$ is another (small) nondegenerate perturbation of $H$ and $CF(G'; \Lambda_{\mb Z}^\Gamma)$ is the Floer cochain complex with differential $d_{G'}$, which contains the ``local'' differential $d_{G'}^\loc$, then choose a Hamiltonian connection $\sigma$ on ${\mb R}\times S^1$ interpolating $G dt$ and $G' dt$. Upon other choices, one obtains the continuation map associated to $\sigma$
\beqn
\Phi: CF(G; \Lambda_{\mb Z}^\Gamma) \to CF(G'; \Lambda_{\mb Z}^\Gamma).
\eeqn
By using the crossing energy estimate, i.e. Lemma \ref{lemmac5}, again, one can decompose
\beqn
\Phi = \Phi^{\rm loc} + \Phi^{\rm big}
\eeqn
where $\Phi^{\rm loc}$ is defined via counting those ``local'' solutions to the continuation map equation. Similar to proving that $d_H^\loc$ is a differential, one can see that $\Phi^{\rm loc}$ is a cochain map from $CF^\loc(G, \ov{\mc F})$ to $CF^\loc(G', \ov{\mc F})$ for each $\ov{\mc F}$. Indeed, all the global constructions of \cite{Bai_Xu_2025}, combined with the crossing energy estimate, can be localized. One then can prove that $CF^\loc(H, \ov{\mc F})$ is well-defined up to chain homotopy. Let the well-defined cohomology groups be $HF^\loc(H, \ov{\mc F})$ resp. $HF^\loc(H, {\mc F})$. If one forgets the ${\mb Z}$-grading, then $HF^\loc(H, \ov{\mc F})$ only depends on the diffeomorphism $\phi$ and the component ${\mc F}$, hence is denoted by $HF^\loc(\phi, {\mc F})$.

\begin{remark}\label{remarkc7}
Alternatively, one can construct the local Floer cohomology via classical perturbation method as done in \cite{AS-torsion}. We omit the details of such construction and the comparison with the current approach. However, to fit into our framework, especially to prove the corresponding homological perturbation results, we need a version which uses the general framework of \cite{Bai_Xu_2025} and extract the local differential from the total differential.
\end{remark}

\subsubsection{Local PSS map}

One uses a local PSS map to calculate the local Floer cohomology defined above. 

Choose a capping $\ov{\mc F}$ of the component ${\mc F} \subset {\mc O}(H_0)$. Let $H$ be a nondegenerate Hamiltonian which is $C^\infty$-close to $H_0$, from which one obtains the complex $CF^\loc(H, \ov{\mc F}; {\mb Z})$ generated by capped 1-periodic orbits of $H$ which are close to ${\mc F}$ whose cappings are induced from the capping $\ov{\mc F}$. On the other hand, choose a Morse--Smale pair $(f_{\mc F}, h_{\mc F})$ on ${\mc F}$ with associated Morse complex $CM({\mc F}; {\mb Z})$. To define the local PSS map, choose a Hamiltonian connection $\sigma$ on ${\mb R}\times S^1$ interpolating from $H_0 dt$ to $H dt$. Then for each $y \in {\rm crit} f_{\mc F}$ and a generator $\ov{x} \in \tilde {\mc O}(H)$ of $CF^\loc(H, \ov{\mc F}; {\mb Z})$, consider moduli spaces of pairs $(y, u)$ where
\begin{enumerate}

\item $y: (-\infty, 0] \to {\mc F}$ is a solution to $y'(s) = \nabla f_{{\mc F}} (y(s))$ such that $y(-\infty) = y$.

\item $u: {\mb R}\times S^1 \to M$ is a solution to $(\nabla^\sigma u)^{0,1} = 0$, such that $u$ converges at $+\infty$ to the orbit of $H$ underlying $\ov{x}$ and $u$ converges at $-\infty$ (necessarily) to an orbit of $H_0$ in ${\mc F}$, such that the capping $\ov{\mc F}$ together with $u$ induces the capping $\ov{x}$.

\item $y(0)\in {\mc F}$ coincides with $u(-\infty, 0) \in {\mc F}$ (which is a fixed point of $\phi$).
\end{enumerate}
Denote this moduli space by ${\mc M}_{y;\ov{x}}^\loc$.

One can compactify this space by allowing (global) breakings and {\it a priori} bubblings. However, by choosing $H$ sufficiently close to $H_0$ and choosing $\sigma$ appropriately, elements of ${\mc M}_{y; \ov{x}}^\loc$ have arbitrarily small energy. Hence bubbling does not appear in the compactification. Moreover, Corollary \ref{corc3} and Lemma \ref{lemmac5} imply that broken trajectories appearing in the compactification can only be the ``local'' ones of $H$.  

From the above features of the compactification of ${\mc M}_{y; \ov{x}}^\loc$, one can use the classical method to achieve transversality, for example, by perturbing $J$. Such perturbations must be compatible with the one used to define the local Floer cohomology via classical perturbation method (see Remark \ref{remarkc7}). We omit the details. 

There is an additional necessity of discussing orientations. In general, the linearized operator in the normal direction of ${\mc F}$ induces a possibly nontrivial local system, denoted by ${\mathfrak o}$. To obtain an isomorphism over ${\mb Z}$, one needs to replace $CM({\mc F}; {\mb Z})$ by the twisted Morse complex $CM({\mc F}; {\mathfrak o})$; when $\phi$ has finite order, ${\mathfrak o}$ is trivial (see the related discussion in \cite{AS-torsion}). Then, by counting solutions, one defines a chain map
\beqn
\pss_{\mc F}^\loc: CM({\mc F}; {\mathfrak o}) \to CF^\loc(H, \ov{\mc F}; {\mb Z}).
\eeqn
One can similarly define the local SSP map
\beqn
\ssp_{\mc F}^\loc: CF^\loc(H, \ov{\mc F}; {\mb Z}) \to CM({\mc F}; {\mathfrak o} ).
\eeqn

\begin{remark}
The local PSS and SSP maps can be extracted in a different way which fits in the general framework of \cite{Bai_Xu_2025}. Indeed, one can define a ``pearly Floer cohomology'' using the Morse--Bott Hamiltonian and Morse complexes on critical submanifolds. The construction for  $H \equiv 0$ is included in \cite{Bai_Xu_2025}. Then a global PSS/SSP map can be defined, while the local one can be extracted by considering the low-energy part.
\end{remark}

\begin{prop}\label{propc9}
The local PSS and local SSP maps are chain homotopy inverses to each other. As a consequence, one has a canonical isomorphism
\beqn
H({\mc F}; {\mathfrak o}) \cong HF^\loc(\phi, \ov{\mc F}; {\mb Z}).
\eeqn
When $\phi$ has finite order, the local system ${\mathfrak o}$ is trivial.
\end{prop}

\begin{proof}
We prove that the two-way compositions of $\pss_{\mc F}^\loc$ and $\ssp_{\mc F}^\loc$ are both identities up to homotopy. The composition $\ssp_{\mc F}^\loc \circ \pss_{\mc F}^\loc$ is about gluing at periodic orbits of the nondegenerate $H$. The smallness of the energy together with the crossing energy lemma for the cylinder (Lemma \ref{lemmac5}) implies that this composition is homotopic to the map on $CM({\mc F}; {\mathfrak o})$ induced by a Morse continuation map. For the composition $\pss_{\mc F}^\loc \circ \ssp_{\mc F}^\loc$, the argument contains gluing at the Morse--Bott periodic orbits of $H_0$. This can also be done by the standard method. The triviality of ${\mathfrak o}$ was established in \cite[Section 5.2]{AS-torsion}.
\end{proof}

\subsection{Local equivariant pair-of-pants product}

\subsubsection{Local equivariant Floer cohomology}

Let $\phi\in {\rm Ham}(M, \omega)$ have finite order which is not divisible by a prime $p$. Then there is a bijection between the set of connected components of ${\rm Fix}_c(\phi)$ and the set of connected components of ${\rm Fix}_c(\phi^p)$. We view ${\mc F} \subset {\rm Fix}_c(\phi)$ as a component of ${\rm Fix}_c(\phi^p)$. We would like to define and identify the local equivariant cohomology
\beqn
HF_{\zp}^{\loc}(\phi^{(p)}, {\mc F}).
\eeqn
Let $H_0$ be a 1-periodic Hamiltonian generating $\phi$. Then for any generic small perturbation $H$, the iteration $H^{(p)}$ is still nondegenerate. Choose a capping $\ov{\mc F}$ of ${\mc F}$, which induces a capping $\ov{\mc F}{}^{(p)}$ of ${\mc F}^{(p)}$. Then there is a well-defined subset 
\beqn
\tilde {\mc O}(H^{(p)}, \ov{\mc F}{}^{(p)}) \subset \tilde {\mc O}(H^{(p)})
\eeqn
which generates a cochain group $CF^\loc( H^{(p)}, \ov{\mc F}{}^{(p)})$. Consider the ${\mb Z}/p$-invariant part of the tensor product
\beqn
CF^\loc_\zp(H^{(p)}, \ov{\mc F}{}^{(p)}):= \big( CF^\loc(H^{(p)}, \ov{\mc F}{}^{(p)}) \otimes CM(g) \big)^{\inv}
\eeqn
where $g: S^\infty \to {\mb R}$ is the Morse function we have been using for the Morse-theoretic Borel construction. Then there is a ${\mb Z}_{\geq 0}\times \zp$-action on this cochain group. We can then define an equivariant differential
\beqn
d_{\zp}^\loc: CF^\loc_\zp(H^{(p)}, \ov{\mc F}{}^{(p)}) \to CF^\loc_\zp(H^{(p)}, \ov{\mc F}{}^{(p)})
\eeqn
following the same approach as the isolated case. Notice that one can still use the same almost complex structure $J_0$, hence the previous crossing energy estimate still applies here. The resulting cohomology, denoted by 
\beqn
HF^\loc_\zp(H_0^{(p)}, \ov{\mc F}{}^{(p)})
\eeqn
can be proved to be well-defined and independent of the perturbation $H$. By forgetting the ${\mb Z}$-grading, this cohomology only depends on $\phi$ and ${\mc F}$, hence denoted by 
\beqn
HF_\zp^\loc(\phi^{(p)}, {\mc F}).
\eeqn

\begin{remark}
Analogous to Remark \ref{remarkc7}, one can define the local Floer cohomology by using domain-dependent perturbations while utilizing the free $\zp$-action on the moduli spaces (see \cite{AS-torsion}). Such a construction provides homotopy equivalent complexes and isomorphic equivariant local Floer cohomology.   
\end{remark}

Now we prove the structural result on the local equivariant Floer cohomology. To be able to compute the local equivariant pants product, one uses the PSS map approach.

\begin{prop}\label{propc11}
Let $\phi$ be a finite-order Hamiltonian diffeomorphism on $(M, \omega)$ generated by a Hamiltonian $H$ and $p$ be a prime which does not divide the order of $\phi$. Let ${\mc F}$ be a connected component of ${\rm Fix}_c(\phi)$. Then up to a global degree shift there is an isomorphism of $\rp$-modules
\beqn
HF_{\zp}^{\loc}(H^{(p)}, \ov{\mc F}{}^{(p)}) \cong H({\mc F}; \fp) \otimes \rp.
\eeqn
\end{prop}

\begin{proof}
The proof is similar to that of Proposition \ref{propc9}. Let $(f_{\mc F}, h_{\mc F})$ be a Morse--Smale pair on ${\mc F}$ which induces a Morse complex $CM({\mc F}; \fp)$. Let $G$ be a $C^\infty$-close nondegenerate perturbation of $H$ such that $G^{(p)}$ is also nondegenerate. Let $\sigma$ be a Hamiltonian connection on ${\mb R}\times S^1$ interpolating from $H$ to $G$, which induces a Hamiltonian connection $\sigma^{(p)}$ interpolating from $H^{(p)}$ to  $G^{(p)}$.  Then for each $y \in {\rm crit}(f_{\mc F})\subset {\mc F}$ and $\ov{x} \in \tilde {\mc O}(G^{(p)})$, one can consider a similar moduli space as the one used in the proof of Proposition \ref{propc9}, with $\sigma$ replaced by $\sigma^{(p)}$.

Moreover, to count solutions equivariantly, couple the equation by parametrized gradient flow lines of $g$ in $S^\infty$. Then for $w_-, w_+ \in {\rm crit}(g)$ and $y \in {\rm crit}(f_{\mc F})$, $\ov{x}\in \tilde {\mc O}(G^{(p)}, \ov{\mc F}{}^{(p)})$, one has a moduli space
\beqn
{\mc M}_{w_-, w_+}(y, \ov{x})
\eeqn
of pairs. As $\sigma^{(p)}$ has small curvature, the usual energy identity together with Lemma \ref{lemmac5} implies that the moduli space can be compactified by only adding local broken configurations without sphere bubbles. Still, one does not have automatic ${\mb Z}_{\geq 0} \times \zp$-equivariant transversality. One could, however, achieve transversality by virtual perturbations which respect the symmetry. As there are no sphere bubbles, one can use single-valued perturbations and obtain integral counts. Then one obtains a $\zp$-equivariant cochain map
\beqn
CM({\mc F};{\mb F}_p) \otimes CM(g) \to CF^{\loc}(G^{(p)}, \ov{\mc F}{}^{(p)}) \otimes CM(g).
\eeqn
Its restriction to the $\zp$-invariant part is 
\beqn
CM({\mc F}; \rp) \to CF_\zp^\loc (G^{(p)}, \ov{\mc F}{}^{(p)})
\eeqn
and hence induces the PSS map
\beqn
\pss_{p, {\mc F}}^\loc: H({\mc F}; \rp) \to HF_\zp^\loc( H^{(p)}, \ov{\mc F}{}^{(p)}).
\eeqn
The standard method proves that this map is $\rp$-linear.

One can also define a similar SSP map
\beqn
\ssp_{p, {\mc F}}^\loc: HF_\zp^\loc( H^{(p)}, \ov{\mc F}{}^{(p)} ) \to H({\mc F}; \rp).
\eeqn
To prove that the PSS and SSP maps have composition in either direction being invertible, one considers the two different gluings. As it is true for the non-equivariant versions, the equivariant version only differs by terms with positive $(t, \theta)$-degrees. Therefore, the claimed isomorphism can be induced by $\pss_{p, {\mc F}}^\loc$.
\end{proof}

\subsubsection{Local equivariant pair-of-pants product}

Previously we have constructed the local Floer cochain complex $CF^{\rm loc}(G, \ov{\mc F})$ for any sufficiently close nondegenerate perturbation $G$ of the Morse--Bott Hamiltonian $H$. Let $p$ be a prime and suppose ${\mc F}^{(p)}$ is still a Morse--Bott fixed submanifold of $\phi^p$. We may choose $G$ such that $G^{(p)}$ is still nondegenerate and can be used to define the local Floer cochain complex $CF^{\rm loc}(G^{(p)}, \ov{\mc F}{}^{(p)})$. Then one can consider the flat Hamiltonian connection on the surface $\Sigma_p^\circ$ which is the pullback of $G dt$ from the cylinder to define the equivariant pair-of-pants product, which firstly provides a $\zp$-equivariant cochain map
\beqn
\wt{\wt{\fst}}_p: CF(G)^{\otimes p} \widehat{\otimes} CM(g) \to CF(G^{(p)}) \widehat{\otimes} CM(g).
\eeqn

Now by the crossing energy estimate for pair-of-pants (Lemma \ref{lemmac6}), when $G$ is sufficiently close to $H$, one has a decomposition
\beqn
\wt{\wt{\fst}}_p = \wt{\wt{\fst}}{}_p^\loc + \wt{\wt{\fst}}{}_p^{\rm big}
\eeqn
corresponding to those ``local'' pair-of-pants and ``big'' energy pair-of-pants. One can see that the local one restricts to maps
\beqn
\wt{\wt{\fst}}{}_{p, \ov{\mc F}}^\loc: CF^\loc(G, \ov{\mc F})^{\otimes p} \otimes CM(g) \to CF^\loc( G^{(p)}, \ov{\mc F}{}^{(p)}) \otimes CM(g).
\eeqn
Using the crossing energy estimates again, one can see that this restriction is a cochain map with respect to the local differentials. Moreover, it is still $\zp$-equivariant. Hence one obtains a cochain map
\beqn
\wt{\fst}{}_{p, \ov{\mc F}}^\loc: CF^\loc(G, \ov{\mc F})^{\otimes p}_\zp \to CF_\zp^\loc(G^{(p)}, \ov{\mc F}{}^{(p)}).
\eeqn
The induced map on cohomology, composed with the quasi-Frobenius map, gives a map
\beqn
\fst_{p, \ov{\mc F}}^\loc: HF^\loc(H, \ov{\mc F}) \to HF^{\loc}_{\zp}(H^{(p)}, \ov{\mc F}{}^{(p)}).
\eeqn
We call this map the {\bf local Floer-theoretic Steenrod operation}.

\begin{remark}
Again, if $HF^\loc(H, \ov{\mc F})$ and $HF_\zp^\loc(H^{(p)}, \ov{\mc F}{}^{(p)})$ are defined via classical perturbation method, then one can define the product $\wt{\fst}{}_{p, \ov{\mc F}}^\loc$ by extending such perturbations. The crossing energy estimate for pair-of-pants still plays a role in such a construction.
\end{remark}

Lastly we prove that the local Floer-theoretic Steenrod operation agrees with the classical Steenrod operation on the component ${\mc F}$.

\begin{prop}\label{propc13}
The following diagram commutes. 
\beqn
\xymatrix{ HF^{\loc}(H, \ov{\mc F}) \ar[rr]^-{\fst_{p, \ov{\mc F}}^\loc}   &  &  HF^{\loc}_{\zp}(H^{(p)}, \ov{\mc F}{}^{(p)}) \ar[d]^{\ssp_{\zp, \ov{\mc F}}^\loc}  \\
        H({\mc F}) \ar[rr]_-{St_p} \ar[u]^{\pss_{\mc F}^\loc} &  &    H({\mc F}; \fp) \otimes \rp  }
\eeqn
\end{prop}

\begin{proof}[Sketch of proof]
The proof is based on the crossing energy estimate and gluing (at nondegenerate Hamiltonian orbits). Let $G$ be a nondegenerate perturbation of $H$. On the cochain level, one needs to verify the commutativity of 
\beq\label{eqn51}
\vcenter{ \xymatrix{
CF(G, \ov{\mc F})^{\otimes p} \otimes CM(g)  \ar[rr]^-{\wt{\wt{\fst}}{}_{p, \ov{\mc F}}^\loc} & &   CF(G^{(p)}, \ov{\mc F}{}^{(p)}) \otimes CM(g) \ar[d]^{\wt{\wt{\ssp}}{}_{\zp, {\mc F}}^\loc}\\
CM({\mc F})^{\otimes p} \otimes CM(g)  \ar[rr] \ar[u]^{(\pss_{{\mc F}}^\loc)^{\otimes p} \otimes {\rm Id}}  & &  CM({\mc F}) \otimes CM(g).
}}
\eeq
Here the unlabelled arrow, which is supposed to induce the classical Steenrod operation on ${\mc F}$, has not been explicitly defined yet.

One uses a typical TQFT+gluing argument to prove the commutativity. The triple composition
\beqn
\wt{\wt{\ssp}}{}_{\zp, {\mc F}}^\loc \circ \wt{\wt{\fst}}^{\loc}_{p, \ov{\mc F}} \circ ( ( \pss_{\mc F}^\loc)^{\otimes p} \otimes {\rm Id})
\eeqn
can be identified, via gluings at $p$ negative ends resp. one positive end and at (nondegenerate) orbits of $G$ resp. $G^{(p)}$ as well as critical points of $g$, to the map defined via the count of configurations where one has $p$ incoming gradient rays and one outgoing gradient ray on ${\mc F}$, between which is a solution to the Floer equation on $\Sigma$ being ${\mb C}$ removing $p$-th roots of unity with flat Hamiltonian connection equal to the pullback of $H dt$ from ${\mb R}\times S^1$. The crossing energy estimate can guarantee the gluing happens ``locally'' and the Floer equation on $\Sigma$ must be a constant one. Therefore, the configurations one finally obtains are those Morse gradient trees responsible for the classical Steenrod operation on ${\mc F}$. The original commutative diagram is then obtained by 1) restricting \eqref{eqn51} to the $\zp$-invariant part, 2) reducing to cohomology, and 3) composing with the quasi-Frobenius maps.
\end{proof}

\subsection{Homological perturbation}

In the situation of Theorem \ref{thm_torsion}, one needs to consider a similar homological perturbation construction for Hamiltonian diffeomorphisms with Morse--Bott fixed point sets. Below is an analogue of Proposition \ref{prop31} for such situations.

\subsubsection{Non-equivariant Floer cohomology}

\begin{prop}\label{homological_perturbation_Morse_Bott}
Let $H$ be a Hamiltonian generating $\phi \in {\rm Ham}(M, \omega)$ with Morse--Bott contractible fixed points. Then there exist $\epsilon(H) > 0$ and, for each ground field ${\mb K}$, a complex of ${\mb Z}$-graded $\Lambda_{\mb K}^\Gamma$-modules
\beqn
C(H )= \left( \bigoplus_{{\mc F} \in \pi_0( {\mc O}(H))} HF^{\rm loc}(H, {\mc F}; \Lambda_{\mb K}^\Gamma), d_H \right)
\eeqn
satisfying the following conditions.

\begin{enumerate}

    \item Let ${\mc A}_H$ be the energy filtration on the local Floer cohomology groups defined by setting it to be the constant ${\mc A}_H(\ov{\mc F})$ on $HF^\loc(H, \ov{\mc F}; {\mb K}) \subset HF^\loc(H, {\mc F}; \Lambda_{\mb K}^\Gamma)$. Then $d_H$ increases the energy filtration by at least $\epsilon(H)$.

    \item $C(H)$ is chain homotopic to the Floer cochain complex associated to any nondegenerate Hamiltonian (given by Theorem \ref{thm:Floer-cochain}). 

    \item For any $a<b$, $a, b \notin {\rm Spec}(H)$, the cohomology of the complex $C(H)^{(a, b)}$ agrees with $HF(H)^{(a, b)}$ which was defined by \eqref{eq: filtered-irrational}.
\end{enumerate}
\end{prop}

\begin{proof}
Choose a nondegenerate Hamiltonian $G$ satisfying the assumption of Lemma \ref{lemmac2}. Then the chain complex $CF(G; \Lambda_{\mb K}^\Gamma)$ can be defined with a differential $d_G$, which has a decomposition 
\beqn
d_G = d_G^\loc + d_G^{\rm big}
\eeqn
into ``local'' part $d_G^\loc$ and the big-energy part $d_G^{\rm big}$. Then the homological perturbation argument provides a construction of the desired complex $C(H)$. The first two properties both follow from the construction directly. For the third one, the proof is exactly the same as the corresponding one of Proposition \ref{prop31}. 
\end{proof}

\subsubsection{Equivariant Floer cohomology}

Now we state and prove the analogue of Proposition \ref{prop33b} and Proposition \ref{prop35}.

\begin{prop}\label{propc14}
Let $H$ be a 1-periodic Hamiltonian on $(M, \omega)$ whose time-1 map $\phi$ generates a nontrivial ${\mb Z}/p'$-action for a certain prime $p'$. Let $p\neq p'$ be a different odd prime. Then there exists a cochain complex 
\beqn
C_{\zp}(H^{(p)}) = \left( \bigoplus_{{\mc F} \in \pi_0({\mc O}(H))} HF^{\rm loc}(H^{(p)}, {\mc F}^{(p)}; \Lambda^\Gamma) \underset{\Lambda^\Gamma}{\otimes} \Lambda_\rp^\Gamma,\  d_\zp \right)
\eeqn
satisfying the following conditions.
\begin{enumerate}
\item There is a chain homotopy equivalence (which is well-defined up to homotopy) from $C_{\zp}(H^{(p)})$ to the $\zp$-equivariant Floer cochain complex for any Hamiltonian $G$ with $G^{(p)}$ nondegenerate. In particular, there is a canonical isomorphism
\beqn
HF_\zp(M) \cong H(C_\zp(H^{(p)})).
\eeqn

\item There exists $\epsilon(H) >0$ such that $d_\zp$ increases the energy filtration by at least $\epsilon(H)$.

\end{enumerate}
Moreover, there exists a cochain map
\beqn
\wt P: C(H)^{\otimes p}_\zp  \to C_{\zp}(H^{(p)})
\eeqn
and a decomposition 
\beqn
\wt P = \wt P^\loc + \wt P^{\rm big}
\eeqn
satisfying the following conditions.
\begin{enumerate}
    \item Let $P: H(C(H)) \to HF_\zp(H^{(p)})$ be the composition of $\wt P$ with the quasi-Frobenius map on the cohomology level. Then the following diagram commutes.
    \beqn
    \xymatrix{  H(C(H)) \ar[r]^-P \ar[d]_{\cong}  &   H( C_{\zp}(H^{(p)})) \ar[d]^{\cong} \\
        HF(M; \Lambda^\Gamma) \ar[r]_-{\fst_p}  & HF_{\zp}(M; \Lambda_{\kzp}^\Gamma)    }
        \eeqn
        Here the vertical isomorphisms are provided by Proposition \ref{homological_perturbation_Morse_Bott} and above.

        \item On the component summand $HF^\loc(H^{(p)}, {\mc F}^{(p)}; \Lambda^\Gamma) \otimes_{\Lambda^\Gamma} \Lambda_\rp^\Gamma$, define the action filtration ${\mc A}_{H^{(p)}}$, which induces via direct sum an action filtration ${\mc A}_{H^{(p)}}$ on $C_\zp(H^{(p)})$. Then $\wt P^{\rm big}$ increases the energy filtration by at least a certain constant $\epsilon > 0$.
        
        \item For each $\ov{\mc F} \in \pi_0(\tilde {\mc O}(H))$, $\wt P^\loc$ restricts to a linear map
        \beqn
        HF^{\rm loc}(H, \ov{\mc F})^{\otimes p}_\zp \to  HF^{\rm loc}(H^{(p)}, \ov{\mc F}{}^{(p)}) \otimes \rp
        \eeqn
        and the following diagram commutes: here $\delta_{{\mb Z}/p}$ is defines similarly as in 
        \beqn
        \xymatrix{  HF^{\rm loc}(H, \ov{\mc F})^{\otimes p}_\zp \ar[d]^{=} \ar[rr]^-{\wt P^\loc} & & {\rm ker}(d_{\zp}^{\loc}) \cap \big( HF^{\rm loc}(H^{(p)}, \ov{\mc F}{}^{(p)}) \otimes \rp \big) \ar[d]\\
            HF^{\loc}(H, \ov{\mc F})^{\otimes p}_\zp  \ar[rr]_-{\wt{\fst}{}_{p, \ov{\mc F}}^\loc} & & HF^{\loc}_{\zp}(H^{(p)}, \ov{\mc F}^{(p)})}
        \eeqn

\end{enumerate}

\end{prop}

\begin{proof}
The construction of the differential $d_{\zp}$ is very similar to the isolated case given in the proof of Proposition \ref{prop33b}. The component-wise decomposition of $C_\zp(H^{(p)})$ follows from Proposition \ref{propc11}. Lemma \ref{lemmab6} allows us to construct the chain homotopy equivalence between the local equivariant Floer complex and the local equivariant Floer cohomology, where the coefficient ring is $\kzp$. Then the first half of this proposition, i.e., the canonical isomorphism $H(C_\zp(H^{(p)})) \cong HF_\zp(M)$ and the claim that $d_{\zp}$ strictly increases the action by a positive constant, follows from the same argument as the isolated case.

\end{proof}

One can also use the Tate version of the above construction to obtain an analogue of Corollary \ref{cor35}. As $C(H)$ is a Floer-type complex, one has a well-defined bar-length spectrum 
\beqn
0 < \beta_1(H) \leq \cdots \leq \beta_N(H).
\eeqn
One also considers the Tate version of the homologically perturbed equivariant Floer complex. Define
\beqn
\wh C_{\zp}(H^{(p)}):= C_\zp(H^{(p)}) \underset{\Lambda_{\kzp}^\Gamma}{\otimes} \Lambda_\kp^\Gamma
\eeqn
which is a Floer-type complex over the Novikov field $\Lambda_{\kp}^\Gamma$. Hence it has the Tate bar-length spectrum
\beqn
0 < \wh \beta_1(H^{(p)}) \leq \cdots \leq \wh \beta_{N^{(p)}}(H^{(p)}).
\eeqn

\begin{cor}\label{corc16}
Under the assumption of Proposition \ref{propc14}, for $p$ being an odd prime, one has $N^{(p)} = 2N$ and for each $i = 1, \ldots, N$, there holds
\beqn
\wh \beta_{2i-1}( \phi^p) = \wh \beta_{2i}(\phi^p) = p \beta_i(\phi).
\eeqn
\end{cor}

\begin{proof}
The proof is very similar to that of Corollary \ref{cor35}. One can change the coefficient to the universal Novikov rings/fields. Then the equivariant pair-of-pants product induces an isomorphism of $\Lambda_{0, \kp}^{\rm univ}$-modules
\beqn
\wh P_0: H( \wh{C(H; \Lambda_0^{\rm univ})^{\otimes p}_\zp}) \to H(\wh C_{\zp}(H^{(p)}; \Lambda_{0, \kp}^{\rm univ}))
\eeqn
where the bar-lengths under consideration label the torsion parts. 
\end{proof}

\begin{prop}\label{propc17}
Let $H$ be a 1-periodic Hamiltonian whose time one map $\phi \in {\rm Ham}(M, \omega)$ generates a nontrivial ${\mb Z}/p'$-action on $M$ for a certain prime $p'$. Then for any odd prime $p \neq p'$, the following are true.
\begin{enumerate}

\item For any $k \geq 0$, the differential $d_{H^{(p^k)}}$ of the complex $C(H^{(p^k)})$ given by Proposition \ref{homological_perturbation_Morse_Bott} vanishes. Hence 
\beqn
HF(M; \Lambda^\Gamma) \cong \bigoplus_{{\mc F} \in \pi_0({\mc O}(H))} HF^\loc( H^{(p^k)}, {\mc F}^{(p^k)}; \Lambda^\Gamma).
\eeqn

\item For any $k \geq 1$, the differential $d_{\zp}$ of the complex $C_{\zp}(H^{(p^k)})$ given by Proposition \ref{propc14} vanishes. Hence
\beqn
HF_\zp(M; \Lambda_{\kzp}^\Gamma) \cong \bigoplus_{{\mc F} \in \pi_0({\mc O}(H))} HF_\zp^\loc(H^{(p^k)}, {\mc F}^{(p^k)}; \Lambda_{\kzp}^\Gamma).
\eeqn
\end{enumerate}
\end{prop}



\begin{proof}
We follow the argument of \cite[Section 6.3]{AS-torsion}. First, because $\phi$ has order $p'$, there holds
\beqn
\max_{1 \leq j \leq p'-1}\beta_{\mathrm{tot}}(\phi^j) \geq \beta_{\mathrm{tot}}(\phi^{i \cdot p^k}).
\eeqn
Now we make use of the quantitative Smith inequality (Theorem \ref{thm219}), which is established for nondegenerate Hamiltonians. As the barcodes vary continuously with the Hamiltonian diffeomorphism, together with (2) of Proposition \ref{propc13}, for the degenerate Hamiltonian diffeomorphism $\phi^i$, one has
\beqn
\beta_{\rm tot}(\phi^{i \cdot p^k}) \geq p^k \beta_{\rm tot}(\phi^i).
\eeqn 
As $k$ is arbitrary, we see that $\beta_{\mathrm{tot}}(\phi^i) = 0$ for all $ i \in \{1, \ldots, p'\}$. In particular, the cochain complex $C(H^{(p^k)})$ does not admit any nontrivial differential.

For the complex $C_{\zp}(H^{(p^k)})$, suppose on the contrary it has a nontrivial differential $d_{\zp}$. Then the corresponding Tate complex $\wh C_{\zp}(H^{(p^k)})$ also admits a nontrivial differential. As shown in Proposition \ref{propc14}, $d_{\zp}$ strictly increases the filtration. Then the Tate bar-length spectrum has a nonzero shortest bar-length $\wh \beta_1(H^{(p^k)})$. It follows from Corollary \ref{corc16} that $\wh \beta_1(H^{(p^k)}) = p \beta_1( H^{(p^{k-1})}) \neq 0$, which contradicts the first assertion we just proved. 
\end{proof}

\bibliographystyle{amsalpha}
\bibliography{bibliors}

\end{document}